\input{glyphtounicode}
\RequirePackage{cmap}

\newif\ifjor
\jortrue

\ifjor
    \documentclass[cupthm,crop,info]{CUP-JNL-ETS}%
    \hypersetup{bookmarksnumbered=true}
    
    \usepackage{graphicx}
    \usepackage{multicol,multirow}
    \usepackage{amsmath,amssymb,amsfonts}
    \usepackage{mathrsfs}
    \usepackage{rotating}
    \usepackage{appendix}

    \usepackage{amsthm}
    \usepackage{lipsum}
    
    \theoremstyle{cupplain}
    \newtheorem{thm}{Theorem}[section]
    \newtheorem{lem}[thm]{Lemma}
    \newtheorem{cor}[thm]{Corollary}
    \newtheorem{prop}[thm]{Proposition}
    \theoremstyle{cupdefinition}
    \newtheorem{defn}{Definition}[section]
    \theoremstyle{cupremark}
    \newtheorem{rem}[thm]{Remark}
    \newtheorem{example}[thm]{Example}
    \theoremstyle{cupproof}
    \newtheorem{proof}{Proof}
    
    \numberwithin{equation}{section}
\else
    \documentclass[english]{article}
    \usepackage[bookmarks=true,bookmarksnumbered=true,bookmarksopen=true,bookmarksopenlevel=3, breaklinks=false,pdfborder={0 0 1},backref=false,colorlinks=false]{hyperref}
    \hypersetup{hidelinks}
    \usepackage{geometry}
    \providecommand{\corollaryname}{Corollary}
    \providecommand{\definitionname}{Definition}
    \providecommand{\examplename}{Example}
    \providecommand{\lemmaname}{Lemma}
    \providecommand{\propositionname}{Proposition}
    \providecommand{\remarkname}{Remark}
    \providecommand{\theoremname}{Theorem}
    
    \numberwithin{equation}{section}
    \numberwithin{figure}{section}
    \numberwithin{table}{section}
    \theoremstyle{plain}
    \newtheorem{thm}{\protect\theoremname}[section]
    \theoremstyle{remark}
    \newtheorem{rem}[thm]{\protect\remarkname}
    \theoremstyle{definition}
    \newtheorem{defn}[thm]{\protect\definitionname}
    \theoremstyle{plain}
    \newtheorem{prop}[thm]{\protect\propositionname}
    \theoremstyle{definition}
    \newtheorem{example}[thm]{\protect\examplename}
    \theoremstyle{plain}
    \newtheorem{cor}[thm]{\protect\corollaryname}
    \theoremstyle{plain}
    \newtheorem{lem}[thm]{\protect\lemmaname}
\fi

\usepackage[T1]{fontenc}
\usepackage[utf8]{inputenc}

\usepackage[english]{babel}
\usepackage[style=numeric,maxbibnames=99,giveninits=true]{biblatex}
\usepackage{csquotes}

\usepackage{parskip}
\usepackage{xcolor}
\usepackage{babel}
\usepackage{mathtools}
\usepackage{amsmath}
\usepackage{amsthm}
\usepackage{amssymb}
\usepackage{graphicx}
\usepackage{comment}

\makeatletter

\def\@opjournalheader{}

\@ifundefined{date}{}{\date{}}
\usepackage{bbm}
\usepackage{faktor}
\usepackage[noabbrev,capitalize]{cleveref}
\usepackage{tikz}

\usepackage{pgfplots}
\usepgfplotslibrary{groupplots,dateplot}
\usetikzlibrary{patterns,shapes.arrows}
\pgfplotsset{compat=newest}

\usepackage[font=footnotesize,width=0.95\textwidth]{caption}

\let\oldsection\section
\renewcommand{\section}{%
    \setcounter{equation}{0}%
    \oldsection%
}

\makeatother

\jyear{2026}
\renewcommand{\keywordfont}{\small}

\begin{document}

\begin{Frontmatter}

\title{Perturbed Families of Symmetric Interval Exchange Maps}

\author{\gname{Idan} \sname{Pazi}}
\author{\gname{Vered} \sname{Rom-Kedar}}

%\author{Idan Pazi and Vered Rom-Kedar}

\address{\orgdiv{Weizmann Institute of Science}, \orgname{Department of Computer Science and Mathematics}, \orgaddress{\city{Rehovot}, \postcode{76100}, \state{Israel}}\\ 
 (\email{idan.pazi@weizmann.ac.il},\email{vered.rom-kedar@weizmann.ac.il})
}

%\Received{\textup{00} November \textup{2019}}
%\Accepted[ and accepted in revised form]{\textup{00} February \textup{2020}}

\maketitle

\authormark{I. Pazi and V. Rom-Kedar}
\titlemark{Perturbed Families of Symmetric Interval Exchange Maps}
\global\long\def\r{\mathbb{R}}%
\global\long\def\rn{\mathbb{R}^{n}}%
\global\long\def\p{\mathbb{P}}%
\global\long\def\c{\mathbb{C}}%
\global\long\def\s{\mathbb{S}}%
\global\long\def\t{\mathbb{T}}%
\global\long\def\z{\mathbb{Z}}%
\global\long\def\cds{,\dots,}%
\global\long\def\setminus{\smallsetminus}%
\global\long\def\n#1{\left\Vert #1\right\Vert }%
\global\long\def\phi{\varphi}%
\global\long\def\eps{\varepsilon}%
\global\long\def\qfrac#1#2{\faktor{#1}{#2}}%
\global\long\def\O#1{\operatorname{#1}}%
\global\long\def\Id{\operatorname{Id}}%
\global\long\def\Ker{\operatorname{Ker}}%
\global\long\def\Im{\operatorname{Im}}%
\global\long\def\Re{\operatorname{Re}}%
\global\long\def\Fix{\operatorname{Fix}}%
\global\long\def\liminf{\operatorname*{lim\,inf}}%
\global\long\def\limsup{\operatorname*{lim\,sup}}%
\global\long\def\conv{\operatorname{conv}}%
\global\long\def\Span{\operatorname{Span}}%
\global\long\def\Int{\operatorname{Int}}%
\global\long\def\diag{\operatorname{diag}}%
\global\long\def\Char{\operatorname{Char}}%
\global\long\def\supp{\operatorname{supp}}%
\global\long\def\rank{\operatorname{rank}}%
\global\long\def\vol{\operatorname{vol}}%
\global\long\def\tr{\operatorname{tr}}%
\global\long\def\one{\mathbbm{1}}%
\global\long\def\act{\curvearrowright}%
\global\long\def\tilde#1{\widetilde{#1}}%
\global\long\def\hat#1{\widehat{#1}}%
\global\long\def\bar#1{\overline{#1}}%
\global\long\def\Sum{\operatorname*{\scalebox{1.6}{\ensuremath{{\displaystyle \sum}}}}}%
\global\long\def\SSum{\operatorname*{\scalebox{2}{\ensuremath{{\displaystyle \sum}}}}}%
\global\long\def\disc{\operatorname{disc}}%
\global\long\def\verteq{\rotatebox{90}{\ensuremath{=}}}%

\begin{abstract}
A perturbed family of interval exchange maps (FIEMs) provides a natural two-\linebreak{}dimensional area-preserving extension of interval exchange maps, with each IEM parameterized by an action variable $y$. Such families arise, for example, as models for iso-energy return maps of perturbed pseudointegrable Hamiltonian impact systems. These maps inherit a time-reversal symmetry, motivating the study of symmetric FIEMs.
In the unperturbed case, the dynamics are generically uniquely ergodic for almost every value of $y$, while a dense set of action values supports periodic intervals. Exploiting time-reversal symmetry, we characterize these intervals and show that symmetric periodic orbits correspond to their midpoints.
Under perturbation, the action variable is no longer conserved and generically periodic intervals break into isolated elliptic or hyperbolic periodic orbits. For sufficiently small perturbations, symmetric periodic orbits persist and can be located by a one-dimensional search along symmetry lines. Associated bifurcations generating symmetric and asymmetric periodic orbits are described and connected to those of the standard map, viewed here as a perturbed family of two-interval exchange maps. 
\end{abstract}



\end{Frontmatter}

\section{Introduction}

An \emph{interval exchange map} (IEM) is a piecewise translation of the interval (or the circle, see Section \ref{subsec:IEM}), defined by partitioning the interval into subintervals (or the circle into arcs) and rearranging them. IEMs were introduced by Keane in his seminal paper \cite{keane1986iet}, and have since been studied extensively, particularly as the return maps of polygonal billiards \cite{boldrighini1978billiards} and of flows on translation surfaces \cite{masur2002rational,zorich2006flat,massart2022short}. IEMs provide a rich class of low-complexity dynamical systems.

Here we study a two-dimensional area-preserving generalization of an IEM by considering a \emph{continuous family of interval exchange maps} $F_{y}:I\to I$, where the IEM $F_y$ depends on the variable $y\in P$. Such a family arises as a model for the iso-energy return map of certain classes of Hamiltonian impact systems, which combine smooth Hamiltonian dynamics with discontinuous mathematical billiards. Similar to the case of planar pseudointegrable billiards \cite{richens1981pseudointegrable}, there are classes of pseudointegrable Hamiltonian impact systems with two degrees of freedom \cite{becker2020impact,pazi2025mechanically}. These systems have two constants of motion and are conjugated, on each common level set, to a directional, possibly pseudointegrable, billiard \cite{becker2020impact,fraczek2023ergodichamiltonian}. The return map on each such level set is given by a piecewise continuous family of IEMs. 
The iso-energy return map of the perturbed system can be modeled by perturbed families of IEMs \cite{romkedar2025hovering},  where the variable $y\in P$ represents an integral of the unperturbed motion. We model perturbations of the underlying Hamiltonian impact system by allowing the action variable $y$ to change at each return: 
\[
\left(x,y\right)\mapsto\left(F_{y}\left(x\right),y+\eps f\left(x,y\right)\right),
\] where area-preservation implies that the perturbation must be of the form $\eps \tilde f(F_y(x))$. The introduction of a drift between the different IEMs of the family results in nontrivial dynamics on the product space of the section interval $I$ and action interval $P$ \cite{romkedar2025hovering}.

A central feature of many dynamical systems, including the impact systems considered here, is being \emph{reversible} by a \emph{time-reversal symmetry} \cite{lamb1998timereversal}. The symmetry is typically a diffeomorphism $S$ of the phase space that conjugates the flow $\phi_{t}$ to its inverse 
\[
\phi_{-t}\left(S\left(z\right)\right)=S\left(\phi_{t}\left(z\right)\right).
\]
Namely, motion forward in time is conjugated to motion backward in time. For example, this symmetry naturally arises in mechanical Hamiltonian systems whenever $H\left(q,p\right)=H\left(q,-p\right)$. In such systems, if $\left(q\left(t\right),p\left(t\right)\right)$ is a solution of the equations of motion, then $\left(q\left(-t\right),-p\left(-t\right)\right)$ is another solution and the time-reversal symmetry $S\left(q,p\right)=\left(q,-p\right)$ is an involution, $S^{2}=\Id$. The time-reversal symmetry enforces a form of symmetry also on the return maps of the system. We exploit this symmetry to facilitate the characterization of periodic orbits of perturbed families of symmetric interval exchange maps, whose structure will be presented in the sections that follow.

The study of periodic orbits provides a fundamental structure for understanding the overall behavior of a dynamical system. According to Poincaré, the periodic points are ``...the only breach by which we can penetrate a fortress hitherto considered inaccessible'' \cite{poincare1967newmethods}. The local dynamics around a periodic orbit influences the behavior of nearby trajectories, and the creation or destruction of periodic points is often associated with the onset of chaotic dynamics \cite{guckenheimer1983nonlinear}.

Here, we identify mechanisms that give rise to both symmetric and non-symmetric periodic orbits of the perturbed FIEMs. Notably, even in the smooth case, non-symmetric periodic orbits are peculiar and more difficult to detect numerically compared to symmetric periodic orbits. A bifurcation of a $4$-periodic point of the standard map where non-symmetric periodic orbits emerge was studied in \cite{murakami2001bifurcation}. There, the authors conjectured that these non-symmetric periodic orbits arise from a ``hidden symmetry''. However, to the best of our knowledge, no such hidden symmetry has ever been found for the standard map. We therefore use the term ``non-symmetric'' to mean ``non-symmetric
with respect to the known symmetries''. Additional discussion of non-symmetric periodic orbits of the standard map appears in \cite{tanikawa2002nonsymmetric} where the authors provide numerical evidence suggesting that a hidden symmetry is not necessary to explain the existence of non-symmetric periodic orbits. They demonstrate an equi-period bifurcation, suggest the existence of a saddle-node bifurcation, and leave the question about the existence of other types of bifurcations open.

%\subsection{Related Works}

Cánovas et al. \cite{canovas2022periodsofiem} gave a complete characterization of the periodic orbits of $2$--IEMs (with flips) and proved some results on $3$--IEMs. Here we utilize symmetry lines of families of symmetric IEMs to deduce the appearance of periodic points of symmetric $d$--IEMs for any $d>1$ (without flips).

The study of skew-products over interval exchange maps,  i.e. maps of the form $\left(x,y\right)\mapsto\left(F\left(x\right),y+f\left(x\right)\right)$ where $F$ is an IEM and $f$ belongs to an appropriate regularity class, is closely related to the behavior of ergodic sums and to the associated cohomological equation \[ u\circ F -u=f.\]
Forni \cite{Forni1997,Forni2002} established deep results on the cohomological equation and the deviation of ergodic averages for area-preserving flows on compact surfaces of higher genus. Marmi--Moussa--Yoccoz \cite{MMY2005} studied the cohomological equation for interval exchange maps of Roth type, showing that it admits bounded solutions for sufficiently regular observables satisfying finitely many obstruction conditions. Berk--Fr{\k{a}}czek--Trujillo \cite{berk2024ergodicityantisymmetricskewproducts} show that the skew-product is ergodic when $F$ is a symmetric IEM  and $f$ is antisymmetric with respect to the midpoints of the elemental subintervals, piecewise monotonic on each subinterval and has a certain singularity at the boundary of the subintervals. Athreya  and Boshernitzan \cite{athreya2013ergodic} prove that composing an IEM with a rotation $\alpha$ leads, for almost all $\alpha$, to uniquely ergodic dynamics and discuss the implications of these results on ergodic sums.

Here, we study the fully coupled area-preserving invertible map $\left(x,y\right)\mapsto\left(F_{y}\left(x\right),y+\eps  f\left(F_y(x)\right)\right)$ where $f$ is smooth and antisymmetric on the full interval and $F_y$ corresponds to a composition of a symmetric IEM and a rotation. This setting differs substantially from the skew-product models since the base dynamics itself depends on the fiber variable, leading to a fully coupled area-preserving piecewise smooth system rather than a cocycle over a fixed IEM. Nonetheless, the above-mentioned cohomological equation appears as a formal perturbative expansion of our map, see Section \ref{sec:invariantcurves}.

Another two-dimensional generalization of IEMs is \emph{piecewise isometries} \cite{goetz2000piecewise}, instead of shuffling segments on intervals by translations, these maps rearrange convex regions on the plane by isometries. Unlike IEMs, piecewise isometries are not necessarily injective. Certain piecewise isometries have invariant curves, and such invariant curves are embeddings of an IEM into the plane \cite{ashwin2020embeddings,cockram2025constructioninvariantcurves}.  Finally, the classical example of a piecewise smooth map arising from an impact Hamiltonian system is the Fermi--Ulam ping pong model, with a corresponding piecewise smooth return map \cite{dolgopyat2012piecewisesmooth}, where the piecewise linear sawtooth map naturally arises. This map has been extensively investigated in the context of a piecewise smooth generalization of the standard map, see \cite{chen1990sawtoothmap} and references therein. In these works the non-smoothness arises from the forcing function $f$, whereas the unperturbed dynamics is smooth and corresponds to rotational motion.

% make main results stronger and more clear
\subsection{Main Results}

Let $F_{y}:\s^{1}\to\s^{1}$ denote a family of interval exchange maps on a circle (or, correspondingly, on a fixed interval, $I$), depending on a variable $y\in P$, where $P$ is an interval. It is assumed that for all $y\in P$ the subinterval reordering of the IEM is the same, that all intervals attain positive lengths, and that these lengths are $C^1$ in $y$. We consider a perturbation in which only the $y$ variable is modified. In Section~\ref{sec:pertform} we show that requiring the map to be area-preserving implies that the perturbation must have the form
\[
T(x,y)=(F_y(x),\,y+\varepsilon f(F_y(x))),
\]
where $f$ is a smooth function.

This map is a two-dimensional area-preserving generalization of an interval exchange map. Restricting attention to \emph{symmetric} interval exchange maps on the interval \(I\), that is, maps reversing the order of subintervals, we show that if \(f(x)\) is antisymmetric with respect to the midpoint of \(I\), the resulting area-preserving perturbed map possesses a time-reversal symmetry for all \(\varepsilon \ge 0\). The corresponding time-reversal \(x\)-periodic case is similarly defined, see the setup section.

In Section \ref{sec:symm-and-rever} we exploit the symmetry of the unperturbed FIEM to identify the symmetry lines of the perturbed FIEM and show that every periodic interval is symmetric, with its midpoint corresponding to a symmetric periodic orbit. We then provide, in Section \ref{sec:pertperiodicpoints}, a classification of all, symmetric and non-symmetric, persisting\footnote{Periodic points that have a continuous limit for arbitrarily small $\eps\ge 0$} periodic points. First, we establish in Theorem \ref{thm:symm-intersect-transv-persist} and Corollary \ref{cor:transversesym} that generically all the unperturbed symmetric periodic orbits persist under sufficiently small perturbations, independent of the form of the antisymmetric function $f(x)$. 
Then, for $f(x)=\sin(2\pi \ell x), \ell \ge 1$, we establish in Theorems \ref{thm:balancedorbits} and \ref{thm:balancedorbitslargeell} that the existence of persisting non-symmetric periodic orbits is tightly related to the balance property of the symmetric orbits and to the ratio between $\ell$  and $ \lceil 1/d_q\rceil$, where $d_q$ denotes the length of the unperturbed $q$-periodic interval. For a given period $N$ and sufficiently small $\eps<\eps_0(N)$, these results classify all persisting periodic orbits with $q\le N$. 
 Bifurcations that occur at some $\eps_b>0$ may alter this classification. For symmetric periodic orbits, such bifurcations correspond to the possible birth or death of intersections of the symmetry lines, via smooth (e.g. folding) or non-smooth (e.g. new crossings created at discontinuity points) mechanisms. Some examples of such bifurcations are demonstrated numerically.
The appearance of non-symmetric periodic orbits at some finite $\eps$ value, even in cases in which there are no persisting non-symmetric periodic orbits, is  demonstrated for the case of $f=\sin(2\pi x)$. We show that there exists an $\eps_b$ value at which a persisting elliptic symmetric period-$2$ orbit bifurcates, via a pitchfork bifurcation, leading to the birth of two non-symmetric period-$2$ orbits for $\eps>\eps_b$.
 These examples provide a glimpse of possible mechanisms for bifurcations that give rise to both symmetric and non-symmetric periodic orbits. The full classification of  bifurcations in such perturbed families is left for future studies. 

This symmetric FIEM framework offers a new perspective on the classical "standard map"; reinterpreting it as a perturbation of a family of $2$--IEMs. This perspective highlights the role of $f\left(x\right)=\sin 2\pi x$  in the well-known symmetric properties of the periodic orbits of the standard map \cite{meiss92symplecticmaps}. By generalizing the standard map to any perturbed family of symmetric $d$--IEMs, we relax some of the symmetries, and give rise to new phenomena such as more prominent resonance islands and the emergence of non-symmetric periodic orbits.

The standard map has, for all   $2\pi \eps< 2\pi \eps_{c}\approx 0.9716\ldots$, many rotational invariant curves. On the other hand, our numerical simulations of the perturbed FIEM suggest that the perturbed family has no such rotational invariant curves for all tested  $\eps>0$ values. 
A natural conjecture is that generically no such curves exist.
In Section~\ref{sec:invariantcurves} we show that if a continuous invariant curve exists it must satisfy strong symmetry and functional constraints. In particular, such a curve can intersect the symmetry lines of the map only at one or two distinct $y$ values. Moreover, a formal expansion near the unperturbed dynamics relates the existence of such curves to bounded solvability of a cohomological equation over the limiting IEM, of the type studied in \cite{MMY2005,marmi2012linearization}.

Finally, in Section \ref{sec:HIS},  the iso-energy return map of a pseudointegrable Hamiltonian impact system is explicitly constructed. For any given energy value, the action axis decomposes into finitely many $y$-intervals on which a piecewise constant integer $d_y$ is defined. On each such $y$-interval the return map is a family of $d_y$-interval exchange maps. The different families connect at singular $y$ values at which some of the intervals have zero length. The upper and lower families correspond to rotational families, where KAM theory applies. We show numerical simulations of a perturbation of this map, exhibiting the intricate dynamics that appears near regular and near singular $y$ values (see \cite{romkedar2025hovering} for the analysis of the behavior near the upper and lower boundaries between the rotational and discontinuous FIEM).

\section{Setup}

\subsection{Interval Exchange Maps}\label{subsec:IEM}

Our definitions and notations follow \cite{viana2006ergodic,MMY2005}. Henceforth, intervals (and subintervals) are bounded, closed on the left and open on the right.

Let $I\subset\r$ be an interval, partitioned into $d$ \emph{elemental subintervals }indexed by an alphabet $\mathcal{A}$ with $d$ symbols, $I=\bigsqcup_{\alpha\in A}I_{\alpha}$. An \emph{interval exchange map of $d$ subintervals (d--IEM)} is a bijection of the interval to itself, acting by translation on each subinterval. The partition is specified by a lengths vector $\lambda\in\r_{+}^{d}$ so that $\left|I_{\alpha}\right|=\lambda_{\alpha}$, and the translation is specified by combinatorial data $\pi=\left(\pi_{0},\pi_{1}\right)$, a pair of bijections $\pi_{\eps}:\mathcal{A}\to\left\{ 1\cds d\right\} $ where $\pi_{0}\left(\alpha\right)$ and $\pi_{1}\left(\alpha\right)$ describe the initial and final position of the $I_{\alpha}$ subinterval, respectively. We use the following notation to represent the combinatorial data $\pi=\left(\begin{smallmatrix}\pi_{0}^{-1}\left(1\right) & \cdots & \pi_{0}^{-1}\left(d\right)\\ \pi_{1}^{-1}\left(1\right) & \cdots & \pi_{1}^{-1}\left(d\right) \end{smallmatrix}\right)$.

The map itself $F:I\to I$ is defined by a \emph{translation vector} $\omega_{\alpha}\in\r^{d}$ whose coordinates are the amount of translation applied to each elemental subinterval 
\begin{equation}
{\displaystyle \omega_{\alpha}=\negthickspace\negthickspace\negthickspace\sum_{\substack{\beta\in\mathcal{A}\\
\pi_{1}(\beta)<\pi_{1}(\alpha)
}
}\negthickspace\negthickspace\negthickspace\lambda_{\beta}-\sum_{\substack{\beta\in\mathcal{A}\\
\pi_{0}(\beta)<\pi_{0}(\alpha)
}
}\lambda_{\beta}}\label{eq:translation-vec}
\end{equation}
then the interval exchange map is given by 
\begin{equation}
F(x)=x+\omega_{\alpha},\qquad x\in I_{\alpha}.\label{eq:IEM-map}
\end{equation}
With no loss of generality, one may assume that $|I|=1$, namely that  
\begin{equation}
\sum_{i=1}^d\lambda_{i} =1 \label{eq:sumlambdais1}
\end{equation}
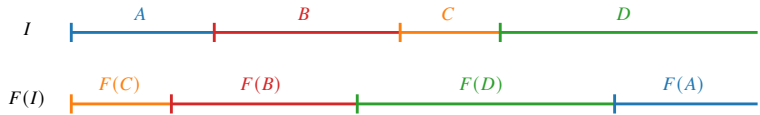
\begin{figure}[h]
\centering{}\resizebox{0.8\columnwidth}{!}{%

\tikzset{every picture/.style={line width=0.75pt}} %set default line width to 0.75pt        

\begin{tikzpicture}[x=0.75pt,y=0.75pt,yscale=-1,xscale=1]
%uncomment if require: \path (0,542); %set diagram left start at 0, and has height of 542

%Straight Lines [id:da1337068379907661] 
\draw [color={rgb, 255:red, 255; green, 127; blue, 14 }  ,draw opacity=1 ][line width=1.5]    (50,120) -- (120,120) ;
\draw [shift={(50,120)}, rotate = 180] [color={rgb, 255:red, 255; green, 127; blue, 14 }  ,draw opacity=1 ][line width=1.5]    (0,6.71) -- (0,-6.71)   ;
%Straight Lines [id:da18137279973936293] 
\draw [color={rgb, 255:red, 214; green, 39; blue, 40 }  ,draw opacity=1 ][line width=1.5]    (120,120) -- (250,120) ;
\draw [shift={(120,120)}, rotate = 180] [color={rgb, 255:red, 214; green, 39; blue, 40 }  ,draw opacity=1 ][line width=1.5]    (0,6.71) -- (0,-6.71)   ;
%Straight Lines [id:da8630689511221543] 
\draw [color={rgb, 255:red, 31; green, 119; blue, 180 }  ,draw opacity=1 ][line width=1.5]    (50,70) -- (150,70) ;
\draw [shift={(50,70)}, rotate = 180] [color={rgb, 255:red, 31; green, 119; blue, 180 }  ,draw opacity=1 ][line width=1.5]    (0,6.71) -- (0,-6.71)   ;
%Straight Lines [id:da28433781681133274] 
\draw [color={rgb, 255:red, 214; green, 39; blue, 40 }  ,draw opacity=1 ][line width=1.5]    (150,70) -- (280,70) ;
\draw [shift={(150,70)}, rotate = 180] [color={rgb, 255:red, 214; green, 39; blue, 40 }  ,draw opacity=1 ][line width=1.5]    (0,6.71) -- (0,-6.71)   ;
%Straight Lines [id:da5250747170253607] 
\draw [color={rgb, 255:red, 255; green, 127; blue, 14 }  ,draw opacity=1 ][line width=1.5]    (280,70) -- (350,70) ;
\draw [shift={(280,70)}, rotate = 180] [color={rgb, 255:red, 255; green, 127; blue, 14 }  ,draw opacity=1 ][line width=1.5]    (0,6.71) -- (0,-6.71)   ;
%Straight Lines [id:da3474122892687759] 
\draw [color={rgb, 255:red, 44; green, 160; blue, 44 }  ,draw opacity=1 ][line width=1.5]    (350,70) -- (530,70) ;
\draw [shift={(350,70)}, rotate = 180] [color={rgb, 255:red, 44; green, 160; blue, 44 }  ,draw opacity=1 ][line width=1.5]    (0,6.71) -- (0,-6.71)   ;
%Straight Lines [id:da7015640846968949] 
\draw [color={rgb, 255:red, 44; green, 160; blue, 44 }  ,draw opacity=1 ][fill={rgb, 255:red, 107; green, 0; blue, 193 }  ,fill opacity=1 ][line width=1.5]    (250,120) -- (430,120) ;
\draw [shift={(250,120)}, rotate = 180] [color={rgb, 255:red, 44; green, 160; blue, 44 }  ,draw opacity=1 ][line width=1.5]    (0,6.71) -- (0,-6.71)   ;
%Straight Lines [id:da11601998235824784] 
\draw [color={rgb, 255:red, 31; green, 119; blue, 180 }  ,draw opacity=1 ][line width=1.5]    (430,120) -- (435,120) -- (530,120) ;
\draw [shift={(430,120)}, rotate = 180] [color={rgb, 255:red, 31; green, 119; blue, 180 }  ,draw opacity=1 ][line width=1.5]    (0,6.71) -- (0,-6.71)   ;

% Text Node
\draw (19.2,67.2) node    {$I$};
% Text Node
\draw (19,117) node    {$F( I)$};
% Text Node
\draw (212.56,57) node  [color={rgb, 255:red, 214; green, 39; blue, 40 }  ,opacity=1 ]  {$B$};
% Text Node
\draw (98.13,57) node  [color={rgb, 255:red, 31; green, 119; blue, 180 }  ,opacity=1 ]  {$A$};
% Text Node
\draw (313.69,57) node  [color={rgb, 255:red, 255; green, 127; blue, 14 }  ,opacity=1 ]  {$C$};
% Text Node
\draw (436.63,57) node  [color={rgb, 255:red, 44; green, 160; blue, 44 }  ,opacity=1 ]  {$D$};
% Text Node
\draw (336.63,107) node  [color={rgb, 255:red, 44; green, 160; blue, 44 }  ,opacity=1 ]  {$F( D)$};
% Text Node
\draw (83.69,107) node  [color={rgb, 255:red, 255; green, 127; blue, 14 }  ,opacity=1 ]  {$F( C)$};
% Text Node
\draw (478.13,107) node  [color={rgb, 255:red, 31; green, 119; blue, 180 }  ,opacity=1 ]  {$F( A)$};
% Text Node
\draw (182.56,107) node  [color={rgb, 255:red, 214; green, 39; blue, 40 }  ,opacity=1 ]  {$F( B)$};

\end{tikzpicture}
}\caption{$4$--IEM $F$ with alphabet $\mathcal{A}=\left\{ A,B,C,D\right\} $
and combinatorial data $\pi=\left(\begin{smallmatrix}A & B & C & D\\
C & B & D & A
\end{smallmatrix}\right)$, the top row is the initial ordering, bottom row is its image, the
final ordering.}
 
\end{figure}

The combinatorial data $\pi$ is somewhat redundant, as we can assume without loss of generality that the initial ordering $\pi_{0}$ is trivial. In this case, the essential information is encoded entirely by the final ordering $\pi_{1}$. Henceforth, we will consider only this simplified setting. Retaining both $\pi_{0}$ and $\pi_{1}$ is useful for algorithmic procedures such as Rauzy induction, and is standard in the literature. 

Note that a $d$--IEM may have a \emph{degenerate discontinuity point}, a point between two consecutive subintervals that are mapped by the same translation. By merging the two subintervals, the map can be reduced to a $\left(d-1\right)$--IEM.

Similarly, we define a bijection of the circle $\s^{1}=\qfrac{\r}{\z}$ by a piecewise rotation \cite{carlos2010transitive,inoue2018piecewiserotation}. A \emph{circle exchange map} of $d$ arcs ($d$--CEM) is the map defined by the partition of the circle into $d$ arcs, starting at the initial rotation $\theta_{0}$ with their lengths defined by the lengths vector $\lambda\in\r_{+}^{d}$. Similar combinatorial data $\pi=\left(\pi_{0},\pi_{1}\right)$ describes the initial and final ordering of the arcs, and the final rotation $\theta_{1}$ denotes the starting point of the final ordering of the arcs on the circle. The circle exchange map is then given by (analogous to \eqref{eq:IEM-map})
\[
F(x)=x+\omega_{\alpha}-\theta_{0}+\theta_{1}\mod 1\qquad x\in I_{\alpha}
\]
where $\omega_{\alpha}$ is the original translation vector from \eqref{eq:translation-vec}.
An example is visualized in Figure \ref{fig:cem}. 

\begin{figure}[h]
\centering{}\resizebox{0.8\columnwidth}{!}{%

\tikzset{every picture/.style={line width=0.75pt}} %set default line width to 0.75pt        

\begin{tikzpicture}[x=0.75pt,y=0.75pt,yscale=-1,xscale=1]
%uncomment if require: \path (0,542); %set diagram left start at 0, and has height of 542

%Straight Lines [id:da7415937098513404] 
\draw [color={rgb, 255:red, 155; green, 155; blue, 155 }  ,draw opacity=1 ][line width=0.75]  [dash pattern={on 2.25pt off 3.75pt}]  (78.17,64.51) -- (131.75,225.25) ;
%Shape: Arc [id:dp03872291462985966] 
\draw  [draw opacity=0][line width=1.5]  (178.39,102.1) .. controls (193.18,127.45) and (194.42,159.77) .. (178.84,187.1) .. controls (155.59,227.88) and (103.68,242.09) .. (62.9,218.84) .. controls (54.03,213.78) and (46.41,207.37) .. (40.17,200) -- (105,145) -- cycle ; \draw [color={rgb, 255:red, 214; green, 39; blue, 40 }  ,draw opacity=1 ][line width=1.5]    (178.39,102.1) .. controls (193.18,127.45) and (194.42,159.77) .. (178.84,187.1) .. controls (155.59,227.88) and (103.68,242.09) .. (62.9,218.84) .. controls (54.03,213.78) and (46.41,207.37) .. (40.17,200) ;  \draw [shift={(178.39,102.1)}, rotate = 239.74] [color={rgb, 255:red, 214; green, 39; blue, 40 }  ,draw opacity=1 ][line width=1.5]    (0,6.71) -- (0,-6.71)   ;
%Shape: Arc [id:dp5036778300719057] 
\draw  [draw opacity=0][line width=1.5]  (40.17,200) .. controls (17.71,173.53) and (12.97,134.81) .. (31.16,102.9) -- (105,145) -- cycle ; \draw [color={rgb, 255:red, 255; green, 127; blue, 14 }  ,draw opacity=1 ][line width=1.5]    (40.17,200) .. controls (17.71,173.53) and (12.97,134.81) .. (31.16,102.9) ;  \draw [shift={(40.17,200)}, rotate = 53.92] [color={rgb, 255:red, 255; green, 127; blue, 14 }  ,draw opacity=1 ][line width=1.5]    (0,6.71) -- (0,-6.71)   ;
%Shape: Arc [id:dp6770031521023396] 
\draw  [draw opacity=0][line width=1.5]  (31.16,102.9) .. controls (54.41,62.12) and (106.32,47.91) .. (147.1,71.16) .. controls (160.55,78.83) and (171.11,89.62) .. (178.39,102.1) -- (105,145) -- cycle ; \draw [color={rgb, 255:red, 31; green, 119; blue, 180 }  ,draw opacity=1 ][line width=1.5]    (31.16,102.9) .. controls (54.41,62.12) and (106.32,47.91) .. (147.1,71.16) .. controls (160.55,78.83) and (171.11,89.62) .. (178.39,102.1) ;  \draw [shift={(31.16,102.9)}, rotate = 119.69] [color={rgb, 255:red, 31; green, 119; blue, 180 }  ,draw opacity=1 ][line width=1.5]    (0,6.71) -- (0,-6.71)   ;
%Shape: Arc [id:dp07582984903558987] 
\draw  [draw opacity=0][line width=0.75]  (415,70) .. controls (415,70) and (415,70) .. (415,70) .. controls (425.87,70) and (436.19,72.31) .. (445.51,76.47) -- (415,145) -- cycle ; \draw [color={rgb, 255:red, 0; green, 0; blue, 0 }  ,draw opacity=1 ][line width=0.75]    (415,70) .. controls (415,70) and (415,70) .. (415,70) .. controls (425.16,70) and (434.85,72.02) .. (443.68,75.68) ; \draw [shift={(445.51,76.47)}, rotate = 202.5] [color={rgb, 255:red, 0; green, 0; blue, 0 }  ,draw opacity=1 ][line width=0.75]    (8.74,-2.63) .. controls (5.56,-1.12) and (2.65,-0.24) .. (0,0) .. controls (2.65,0.24) and (5.56,1.12) .. (8.74,2.63)   ; 
%Shape: Arc [id:dp17436370762142162] 
\draw  [draw opacity=0][line width=1.5]  (450.93,67.94) .. controls (482.4,82.62) and (502.88,115.82) .. (499.68,152.41) -- (415,145) -- cycle ; \draw [color={rgb, 255:red, 255; green, 127; blue, 14 }  ,draw opacity=1 ][line width=1.5]    (450.93,67.94) .. controls (482.4,82.62) and (502.88,115.82) .. (499.68,152.41) ;  \draw [shift={(450.93,67.94)}, rotate = 205] [color={rgb, 255:red, 255; green, 127; blue, 14 }  ,draw opacity=1 ][line width=1.5]    (0,6.71) -- (0,-6.71)   ;
%Shape: Arc [id:dp6697150106888639] 
\draw  [draw opacity=0][line width=1.5]  (499.69,152.41) .. controls (497.11,181.64) and (479.43,208.74) .. (450.92,222.04) .. controls (408.38,241.88) and (357.8,223.47) .. (337.96,180.92) .. controls (333.65,171.67) and (331.14,162.03) .. (330.3,152.41) -- (415,145) -- cycle ; \draw [color={rgb, 255:red, 214; green, 39; blue, 40 }  ,draw opacity=1 ][line width=1.5]    (499.69,152.41) .. controls (497.11,181.64) and (479.43,208.74) .. (450.92,222.04) .. controls (408.38,241.88) and (357.8,223.47) .. (337.96,180.92) .. controls (333.65,171.67) and (331.14,162.03) .. (330.3,152.41) ;  \draw [shift={(499.69,152.41)}, rotate = 275.05] [color={rgb, 255:red, 214; green, 39; blue, 40 }  ,draw opacity=1 ][line width=1.5]    (0,6.71) -- (0,-6.71)   ;
%Shape: Arc [id:dp20933724816377253] 
\draw  [draw opacity=0][line width=1.5]  (330.47,153.88) .. controls (330.41,153.39) and (330.37,152.9) .. (330.32,152.41) .. controls (326.23,105.64) and (360.83,64.41) .. (407.59,60.32) .. controls (423.02,58.97) and (437.84,61.83) .. (450.93,67.95) -- (415,145) -- cycle ; \draw [color={rgb, 255:red, 31; green, 119; blue, 180 }  ,draw opacity=1 ][line width=1.5]    (330.47,153.88) .. controls (330.41,153.39) and (330.37,152.9) .. (330.32,152.41) .. controls (326.23,105.64) and (360.83,64.41) .. (407.59,60.32) .. controls (423.02,58.97) and (437.84,61.83) .. (450.93,67.95) ;  \draw [shift={(330.47,153.88)}, rotate = 84.02] [color={rgb, 255:red, 31; green, 119; blue, 180 }  ,draw opacity=1 ][line width=1.5]    (0,6.71) -- (0,-6.71)   ;
%Straight Lines [id:da4159015664120759] 
\draw    (105,55) -- (105,65) ;
%Straight Lines [id:da5973001545716426] 
\draw    (415,55) -- (415,65) ;
%Shape: Arc [id:dp5573057758968234] 
\draw  [draw opacity=0][line width=0.75]  (40.7,106.37) .. controls (53.82,84.58) and (77.71,70) .. (105,70) -- (105,145) -- cycle ; \draw [color={rgb, 255:red, 0; green, 0; blue, 0 }  ,draw opacity=1 ][line width=0.75]    (41.91,104.42) .. controls (55.26,83.71) and (78.53,70) .. (105,70) ;  \draw [shift={(40.7,106.37)}, rotate = 302.8] [color={rgb, 255:red, 0; green, 0; blue, 0 }  ,draw opacity=1 ][line width=0.75]    (8.74,-2.63) .. controls (5.56,-1.12) and (2.65,-0.24) .. (0,0) .. controls (2.65,0.24) and (5.56,1.12) .. (8.74,2.63)   ;
%Straight Lines [id:da011352356463613655] 
\draw [color={rgb, 255:red, 155; green, 155; blue, 155 }  ,draw opacity=1 ][line width=0.75]  [dash pattern={on 2.25pt off 3.75pt}]  (388.17,64.51) -- (441.75,225.25) ;

% Text Node
\draw (417,216) node  [color={rgb, 255:red, 214; green, 39; blue, 40 }  ,opacity=1 ]  {$F( B)$};
% Text Node
\draw (101.5,80) node  [color={rgb, 255:red, 31; green, 119; blue, 180 }  ,opacity=1 ]  {$A$};
% Text Node
\draw (472,110) node  [color={rgb, 255:red, 255; green, 127; blue, 14 }  ,opacity=1 ]  {$F( C)$};
% Text Node
\draw (32,150) node  [color={rgb, 255:red, 255; green, 127; blue, 14 }  ,opacity=1 ]  {$C$};
% Text Node
\draw (376,93) node  [color={rgb, 255:red, 31; green, 119; blue, 180 }  ,opacity=1 ]  {$F( A)$};
% Text Node
\draw (143,200) node  [color={rgb, 255:red, 214; green, 39; blue, 40 }  ,opacity=1 ]  {$B$};
% Text Node
\draw (11,60.4) node [anchor=north west][inner sep=0.75pt]    {$\mathbb{S}^{1}$};
% Text Node
\draw (301,55.4) node [anchor=north west][inner sep=0.75pt]    {$F\left(\mathbb{S}^{1}\right)$};
% Text Node
\draw (421,74.4) node [anchor=north west][inner sep=0.75pt]  [font=\small]  {$\theta _{1}$};
% Text Node
\draw (106.22,42.2) node    {$0$};
% Text Node
\draw (416.22,42.2) node    {$0$};
% Text Node
\draw (57,83.4) node [anchor=north west][inner sep=0.75pt]  [font=\small]  {$-\theta _{0}$};

\end{tikzpicture}
}\caption{\label{fig:cem}$3$--CEM $F$ with alphabet $\mathcal{A}=\left\{ A,B,C\right\} $,
combinatorial data $\pi=\left(\begin{smallmatrix}A & B & C\\
C & B & A
\end{smallmatrix}\right)$, initial and final rotation $\theta_{0},\theta_{1}$. The symmetry
of this map about the gray dashed diameter passing through $\frac{\theta_{0}+\theta_{1}}{2}$
will be discussed in Section \ref{sec:symm-and-rever}.}
\end{figure}
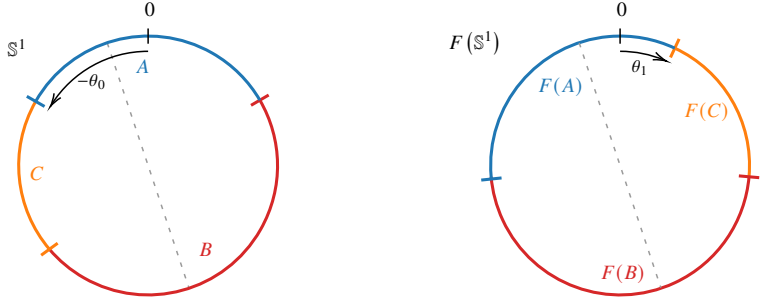

\begin{rem}
Every CEM may be represented by a finite number of different $\pi,\lambda,\theta_{0},\theta_{1}$ combinations. For example, the CEM in Figure \ref{fig:cem} is equivalent to the CEM given by $\pi'=\left(\begin{smallmatrix}A & B & C\\
A & C & B
\end{smallmatrix}\right)$, $\theta_{0}'=\theta_{0}$, $\theta_{1}'=\theta_{1}-\lambda_{A}$.
\end{rem}

\begin{rem}\label{rem:identifyboundaries}
Any $d$--CEM is a $\left(d+2\right)$--IEM when the circle is viewed as an interval starting at some arbitrary point (if the interval is starting at one of the discontinuity points, the map corresponds to a $\left(d+1\right)$--IEM). Conversely, any $d$--IEM on an interval can be viewed as a $d$--CEM with $\theta_{0}=\theta_{1}=0$, by identifying the interval's endpoints. Henceforth we will take this viewpoint and consider the interval as a circle by such identification.
\end{rem}

The equivalence of IEM and CEM in terms of their corresponding translation surfaces was established in \cite{inoue2018piecewiserotation}. Any $1$--CEM or $2$--CEM is a circle rotation, and any $3$--IEM is a return map of a circle rotation to an arc, see \cite{yoccoz2005continued}. Thus, genuinely nontrivial behavior appears at $d>2$ for CEM and $d>3$ for IEM. 

\begin{defn}
An interval $J\subset I$ is a \emph{periodic interval of period $q$ } for the IEM given by $F$ if $q$ is the smallest positive integer such that $F^{q}\left(J\right)=J$, and $J$ is the maximal interval with this property, with every interval along its orbit contained in some elemental subinterval.
\end{defn}

A periodic interval represents a \emph{maximal cylinder} of the corresponding translation surface \cite{masur2002rational}.
\begin{prop}
\label{prop:periodic-orbit-disjoint}The orbit of a periodic interval is a disjoint set of periodic intervals.
\end{prop}

\begin{proof}
Suppose $J$ is a $q$-periodic interval with $0<k<q$ such that $J$ and $T^{k}\left(J\right)$ overlap, it follows that $T^{i}\left(J\right)$
and $T^{i+k}\left(J\right)$ overlap for any $i$, thus they must be contained in the same elemental subinterval for any $i$, it follows that $J\cup T^{k}\left(J\right)$ is periodic interval, contradictory to $J$ being maximal.
\end{proof}

\begin{defn}
A \emph{left saddle connection} (or \emph{connexion} \cite{yoccoz2005continued}) is a triple $\left(\alpha,\beta,m\right)$ such that  the left boundary point of the interval $I_{\alpha}$ (on its initial position) is mapped by the $m$ iteration of the IEM to the left boundary point of the interval $I_{\beta}$ (on its initial position).

A \emph{right saddle connection} is a triple $\left(\alpha,\beta,m\right)$ such that a sequence of points converging to the right boundary point of the interval $I_{\alpha}$ (on its initial position) is mapped by the $m$ iteration of the IEM to a sequence of points converging to the right boundary point of the interval $I_{\beta}$ (on its initial position).
\end{defn}

\begin{prop}
\label{prop:periodic-pts-iem-are-periodic-intervals}Periodic points of an IEM of period $q$ are always part of a periodic interval, bounded by a left saddle connection $\left(\alpha,\alpha,q\right)$ and a right saddle connection $\left(\beta,\beta,q\right)$. In particular, fixed points of an IEM lie in a fixed elemental subinterval.
\end{prop}

\begin{proof}
As elemental subintervals are non-empty intervals closed from the left and open from the right, the periodic point is not isolated.

Consider the maximal periodic interval containing the periodic point. Along the interval orbit (a collection of intervals), one of the intervals must start with an elemental subinterval boundary and one of the intervals must end in an elemental subinterval boundary, otherwise, the periodic interval is not maximal. By periodicity, this implies the existence of the stated left and right saddle connections.
\end{proof}

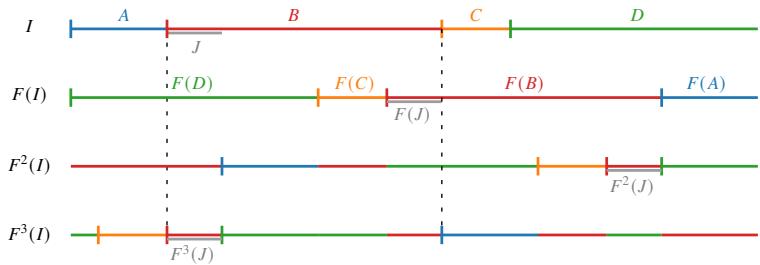
\begin{figure}[h]
\centering{}\resizebox{0.8\columnwidth}{!}{%

\tikzset{every picture/.style={line width=0.75pt}} %set default line width to 0.75pt        

\begin{tikzpicture}[x=0.75pt,y=0.75pt,yscale=-1,xscale=1]
%uncomment if require: \path (0,542); %set diagram left start at 0, and has height of 542

%Straight Lines [id:da9421959461923883] 
\draw  [dash pattern={on 2.25pt off 5.25pt}]  (150,70) -- (150,220) ;
%Straight Lines [id:da28433781681133274] 
\draw [color={rgb, 255:red, 31; green, 119; blue, 180 }  ,draw opacity=1 ][line width=1.5]    (80,70) -- (121,70) -- (150,70) ;
\draw [shift={(80,70)}, rotate = 180] [color={rgb, 255:red, 31; green, 119; blue, 180 }  ,draw opacity=1 ][line width=1.5]    (0,6.71) -- (0,-6.71)   ;
%Straight Lines [id:da5250747170253607] 
\draw [color={rgb, 255:red, 255; green, 127; blue, 14 }  ,draw opacity=1 ][line width=1.5]    (350,70) -- (400,70) ;
\draw [shift={(350,70)}, rotate = 180] [color={rgb, 255:red, 255; green, 127; blue, 14 }  ,draw opacity=1 ][line width=1.5]    (0,6.71) -- (0,-6.71)   ;
%Straight Lines [id:da3474122892687759] 
\draw [color={rgb, 255:red, 44; green, 160; blue, 44 }  ,draw opacity=1 ][line width=1.5]    (400,70) -- (580,70) ;
\draw [shift={(400,70)}, rotate = 180] [color={rgb, 255:red, 44; green, 160; blue, 44 }  ,draw opacity=1 ][line width=1.5]    (0,6.71) -- (0,-6.71)   ;
%Straight Lines [id:da8999481685664672] 
\draw [color={rgb, 255:red, 44; green, 160; blue, 44 }  ,draw opacity=1 ][line width=1.5]    (80,120) -- (260,120) ;
\draw [shift={(80,120)}, rotate = 180] [color={rgb, 255:red, 44; green, 160; blue, 44 }  ,draw opacity=1 ][line width=1.5]    (0,6.71) -- (0,-6.71)   ;
%Straight Lines [id:da5735483286053035] 
\draw [color={rgb, 255:red, 255; green, 127; blue, 14 }  ,draw opacity=1 ][line width=1.5]    (260,120) -- (310,120) ;
\draw [shift={(260,120)}, rotate = 180] [color={rgb, 255:red, 255; green, 127; blue, 14 }  ,draw opacity=1 ][line width=1.5]    (0,6.71) -- (0,-6.71)   ;
%Straight Lines [id:da3293479339265465] 
\draw [color={rgb, 255:red, 214; green, 39; blue, 40 }  ,draw opacity=1 ][line width=1.5]    (310,120) -- (510,120) ;
\draw [shift={(310,120)}, rotate = 180] [color={rgb, 255:red, 214; green, 39; blue, 40 }  ,draw opacity=1 ][line width=1.5]    (0,6.71) -- (0,-6.71)   ;
%Straight Lines [id:da8017226323311231] 
\draw [color={rgb, 255:red, 31; green, 119; blue, 180 }  ,draw opacity=1 ][line width=1.5]    (510,120) -- (580,120) ;
\draw [shift={(510,120)}, rotate = 180] [color={rgb, 255:red, 31; green, 119; blue, 180 }  ,draw opacity=1 ][line width=1.5]    (0,6.71) -- (0,-6.71)   ;
%Straight Lines [id:da3335120255822086] 
\draw [color={rgb, 255:red, 214; green, 39; blue, 40 }  ,draw opacity=1 ][line width=1.5]    (80,170) -- (190,170) ;
%Straight Lines [id:da8096459073423068] 
\draw [color={rgb, 255:red, 31; green, 119; blue, 180 }  ,draw opacity=1 ][line width=1.5]    (190,170) -- (260,170) ;
\draw [shift={(190,170)}, rotate = 180] [color={rgb, 255:red, 31; green, 119; blue, 180 }  ,draw opacity=1 ][line width=1.5]    (0,6.71) -- (0,-6.71)   ;
%Straight Lines [id:da3070892009350735] 
\draw [color={rgb, 255:red, 214; green, 39; blue, 40 }  ,draw opacity=1 ][line width=1.5]    (260,170) -- (310,170) ;
%Straight Lines [id:da8185633462414117] 
\draw [color={rgb, 255:red, 44; green, 160; blue, 44 }  ,draw opacity=1 ][line width=1.5]    (310,170) -- (420,170) ;
%Straight Lines [id:da08299389403701862] 
\draw [color={rgb, 255:red, 255; green, 127; blue, 14 }  ,draw opacity=1 ][line width=1.5]    (420,170) -- (470,170) ;
\draw [shift={(420,170)}, rotate = 180] [color={rgb, 255:red, 255; green, 127; blue, 14 }  ,draw opacity=1 ][line width=1.5]    (0,6.71) -- (0,-6.71)   ;
%Straight Lines [id:da5275602831672278] 
\draw [color={rgb, 255:red, 214; green, 39; blue, 40 }  ,draw opacity=1 ][line width=1.5]    (470,170) -- (510,170) ;
\draw [shift={(470,170)}, rotate = 180] [color={rgb, 255:red, 214; green, 39; blue, 40 }  ,draw opacity=1 ][line width=1.5]    (0,6.71) -- (0,-6.71)   ;
%Straight Lines [id:da8946345390553851] 
\draw [color={rgb, 255:red, 44; green, 160; blue, 44 }  ,draw opacity=1 ][line width=1.5]    (510,170) -- (544.33,170) -- (580,170) ;
\draw [shift={(510,170)}, rotate = 180] [color={rgb, 255:red, 44; green, 160; blue, 44 }  ,draw opacity=1 ][line width=1.5]    (0,6.71) -- (0,-6.71)   ;
%Straight Lines [id:da29286586226591615] 
\draw [color={rgb, 255:red, 214; green, 39; blue, 40 }  ,draw opacity=1 ][line width=1.5]    (510,220) -- (580,220) ;
%Straight Lines [id:da947576343844493] 
\draw [color={rgb, 255:red, 214; green, 39; blue, 40 }  ,draw opacity=1 ][line width=1.5]    (310,220) -- (350,220) ;
%Straight Lines [id:da8051129807291978] 
\draw [color={rgb, 255:red, 31; green, 119; blue, 180 }  ,draw opacity=1 ][line width=1.5]    (350,220) -- (420,220) ;
\draw [shift={(350,220)}, rotate = 180] [color={rgb, 255:red, 31; green, 119; blue, 180 }  ,draw opacity=1 ][line width=1.5]    (0,6.71) -- (0,-6.71)   ;
%Straight Lines [id:da04474702104046824] 
\draw [color={rgb, 255:red, 214; green, 39; blue, 40 }  ,draw opacity=1 ][line width=1.5]    (420,220) -- (470,220) ;
%Straight Lines [id:da38227541963675504] 
\draw [color={rgb, 255:red, 44; green, 160; blue, 44 }  ,draw opacity=1 ][line width=1.5]    (470,220) -- (510,220) ;
%Straight Lines [id:da5047544154706469] 
\draw [color={rgb, 255:red, 44; green, 160; blue, 44 }  ,draw opacity=1 ][line width=1.5]    (260,220) -- (310,220) ;
%Straight Lines [id:da8260292041929443] 
\draw [color={rgb, 255:red, 44; green, 160; blue, 44 }  ,draw opacity=1 ][line width=1.5]    (80,220) -- (100,220) ;
%Straight Lines [id:da059836928268028] 
\draw [color={rgb, 255:red, 255; green, 127; blue, 14 }  ,draw opacity=1 ][line width=1.5]    (100,220) -- (150,220) ;
\draw [shift={(100,220)}, rotate = 180] [color={rgb, 255:red, 255; green, 127; blue, 14 }  ,draw opacity=1 ][line width=1.5]    (0,6.71) -- (0,-6.71)   ;
%Straight Lines [id:da2328655606736637] 
\draw [color={rgb, 255:red, 214; green, 39; blue, 40 }  ,draw opacity=1 ][line width=1.5]    (150,220) -- (190,220) ;
\draw [shift={(150,220)}, rotate = 180] [color={rgb, 255:red, 214; green, 39; blue, 40 }  ,draw opacity=1 ][line width=1.5]    (0,6.71) -- (0,-6.71)   ;
%Straight Lines [id:da663163258245515] 
\draw [color={rgb, 255:red, 44; green, 160; blue, 44 }  ,draw opacity=1 ][line width=1.5]    (190,220) -- (224.33,220) -- (260,220) ;
\draw [shift={(190,220)}, rotate = 180] [color={rgb, 255:red, 44; green, 160; blue, 44 }  ,draw opacity=1 ][line width=1.5]    (0,6.71) -- (0,-6.71)   ;
%Straight Lines [id:da6799499043799566] 
\draw [color={rgb, 255:red, 155; green, 155; blue, 155 }  ,draw opacity=1 ][line width=1.5]    (150,73) -- (190,73) ;
%Straight Lines [id:da8630689511221543] 
\draw [color={rgb, 255:red, 214; green, 39; blue, 40 }  ,draw opacity=1 ][line width=1.5]    (150,70) -- (350,70) ;
\draw [shift={(150,70)}, rotate = 180] [color={rgb, 255:red, 214; green, 39; blue, 40 }  ,draw opacity=1 ][line width=1.5]    (0,6.71) -- (0,-6.71)   ;
%Straight Lines [id:da3292136893776767] 
\draw [color={rgb, 255:red, 155; green, 155; blue, 155 }  ,draw opacity=1 ][line width=1.5]    (470,173) -- (510,173) ;
%Straight Lines [id:da6157888425019918] 
\draw [color={rgb, 255:red, 155; green, 155; blue, 155 }  ,draw opacity=1 ][line width=1.5]    (310,123) -- (350,123) ;
%Straight Lines [id:da6411001679070135] 
\draw [color={rgb, 255:red, 155; green, 155; blue, 155 }  ,draw opacity=1 ][line width=1.5]    (150,223) -- (190,223) ;
%Straight Lines [id:da2352734922697911] 
\draw  [dash pattern={on 2.25pt off 5.25pt}]  (350,70) -- (350,220) ;

% Text Node
\draw (50,68.2) node    {$I$};
% Text Node
\draw (50,118) node    {$F( I)$};
% Text Node
\draw (118.5,60) node  [color={rgb, 255:red, 31; green, 119; blue, 180 }  ,opacity=1 ]  {$A$};
% Text Node
\draw (242.13,60) node  [color={rgb, 255:red, 214; green, 39; blue, 40 }  ,opacity=1 ]  {$B$};
% Text Node
\draw (375,60) node  [color={rgb, 255:red, 255; green, 127; blue, 14 }  ,opacity=1 ]  {$C$};
% Text Node
\draw (492.63,60) node  [color={rgb, 255:red, 44; green, 160; blue, 44 }  ,opacity=1 ]  {$D$};
% Text Node
\draw (168.5,110) node  [color={rgb, 255:red, 44; green, 160; blue, 44 }  ,opacity=1 ]  {$F( D)$};
% Text Node
\draw (286.69,110) node  [color={rgb, 255:red, 255; green, 127; blue, 14 }  ,opacity=1 ]  {$F( C)$};
% Text Node
\draw (410.13,110) node  [color={rgb, 255:red, 214; green, 39; blue, 40 }  ,opacity=1 ]  {$F( B)$};
% Text Node
\draw (542.56,110) node  [color={rgb, 255:red, 31; green, 119; blue, 180 }  ,opacity=1 ]  {$F( A)$};
% Text Node
\draw (50,169) node    {$F^{2}( I)$};
% Text Node
\draw (50,219) node    {$F^{3}( I)$};
% Text Node
\draw (171.5,83) node  [color={rgb, 255:red, 128; green, 128; blue, 128 }  ,opacity=1 ]  {$J$};
% Text Node
\draw (328.5,133) node  [color={rgb, 255:red, 128; green, 128; blue, 128 }  ,opacity=1 ]  {$F( J)$};
% Text Node
\draw (488.5,184) node  [color={rgb, 255:red, 128; green, 128; blue, 128 }  ,opacity=1 ]  {$F^{2}( J)$};
% Text Node
\draw (168.5,234) node  [color={rgb, 255:red, 128; green, 128; blue, 128 }  ,opacity=1 ]  {$F^{3}( J)$};

\end{tikzpicture}
}\caption{\label{fig:iem-periodic}Symmetric $4$--IEM, $\lambda=\left(0.14,0.4,0.1,0.36\right)$
 with the orbit of periodic intervals $J,F\left(J\right),F^{2}\left(J\right),F^{3}\left(J\right)=J$.
The interval is bounded between the left saddle connection $\left(B,B,3\right)$
and the right saddle connection $\left(B,B,3\right)$ (dashed lines).}
\end{figure}

An IEM with no saddle connections is said to satisfy the \emph{infinite
distinct orbit condition}. Keane showed in \cite{keane1986iet} that
such IEMs are minimal, meaning they have no invariant closed proper
subset and every orbit is dense. In \cite{boshernitzan1988rank2}
it was shown how any IEM can be decomposed into its periodic and minimal
components by partitioning the interval along the orbits of its saddle
connections, such decomposition was done in the context of Hamiltonian
impact system in \cite{fraczek2023ergodichamiltonian}.

\subsection{Families of Exchange Maps}

Interval and circle exchange maps appear in natural settings as the return map to some global section. In many cases, the exchange map smoothly depends on some variable, giving rise to a continuous family of interval exchange maps.
\begin{defn}
A family of $d$-interval exchange maps ($d$-FIEM) is a collection of maps $F_{y}:I\to I$ where each $F_{y}$ is a $d$--IEM defined by $\left(\pi,\lambda\left(y\right)\right)$, with the positive lengths vector depending smoothly on the variable $y\in P$. The translation vector  now depends on $y$ and the map is
\begin{equation}
F(x)=x+\omega_{\alpha}\left(y\right),\qquad x\in I_{\alpha}\left(y\right).
\end{equation}
\end{defn}

\begin{defn}
A family of $d$-circle exchange maps ($d$-FCEM) is a collection of maps $F_{y}:I\to I$ where each $F_{y}$ is a $d$--CEM defined by $\pi,\lambda\left(y\right),\theta_{0}\left(y\right),\theta_{1}\left(y\right)$, with the lengths vector and the rotations depending smoothly on the variable $y\in P$.
\end{defn}

We define an \emph{elemental subregion} as the region in the $I\times P$ space that corresponds to the same elemental subinterval (arc) in each of the IEMs (CEMs).

The assumption that the family of maps acts on a fixed interval $I$ (so its length is independent of $y$) is natural for our setting of symplectic maps (see below), yet it may be interesting to explore other cases in future studies. 

Since each interval exchange map corresponds to a translation surface whose genus and singularities are determined solely by the combinatorial data, a $d$-FIEM defines a curve in the moduli space of Riemann surfaces of a certain genus and punctures \cite{zorich2006flat,masur2002rational}. 

The simplest nontrivial FIEM is the linear interpolation between two IEMs $\left(\pi,\lambda_{0}\right)$ and $\left(\pi,\lambda_{1}\right)$. Such a family provides a linear approximation to a general FIEM. Families of this form also arise naturally in the study of rectilinear billiards, where the direction of motion generates a corresponding family of return maps. Locally, away from singularities, the lengths of the intervals depend linearly on the direction of motion. 
\begin{defn}
A \emph{linear $d$-FIEM on $I$ }is specified by a single combinatorial data $\pi$ and two lengths vectors $\lambda_{0},\lambda_{1}\in\r_{+}^{d}$
satisfying $\left|\lambda_{0}\right|_{1}=\left|\lambda_{1}\right|_{1}=\left|I\right|$. The resulting family is the $d$-FIEM given by $\left(\pi,\lambda(y)=y\lambda_{1}+\left(1-y\right)\lambda_{0}\right)$
for $y\in\left[0,1\right]$.
\end{defn}

With respect to the Lebesgue measure on $\r_{+}^{d}$, for almost every choice of $\lambda\in\r_{+}^{d}$ the $\left(\pi,\lambda\right)$ IEM is uniquely ergodic \cite{masur1982interval,veech1982gauss}. It is trivial that on this space, the IEMs with periodic points are dense (we can approximate the lengths vector by rational lengths $\lambda\in\mathbb{Q}_{+}^{d}$). We ask whether a family of interval exchange maps contains a periodic point. In what follows, we show that usually (with precise non-degeneracy conditions) a symmetric FIEM does admit periodic points. In fact, it admits only symmetric periodic intervals, each centered at the intersection of symmetry lines at which a symmetric periodic orbit resides. An analogous argument applies to families of symmetric CEMs.

\begin{figure}[h]
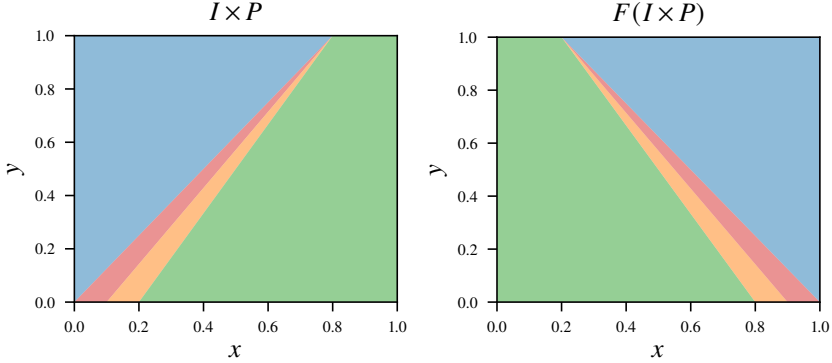

\begin{centering}
\input{drawings/fiem-example_1.pgf}\input{drawings/fiem-example_2.pgf}\caption{\label{fig:fiem-example}The domain (left) and image (right) of a
linear $4$-FIEM with $\pi=\left(\begin{smallmatrix}A & B & C & D\\
C & B & D & A
\end{smallmatrix}\right)$, $\lambda_{0}=\left(0,0.1,0.1,0.8\right)$, $\lambda_{1}=\left(0.8,0,0,0.2\right)$.
Different elemental subregions are in different colors.}
\par\end{centering}
\end{figure}

\subsection{\label{subsec:area-presv-symmetries}Symmetries and Symmetry Lines}

This subsection discusses the symmetries of $2$-dimensional area-preserving maps and their relation to periodic points, following closely the treatment in \cite{mackay1993renormalisation}, where this topic is explored extensively in the continuous setting (see also \cite{vogelaere1958onthestructure}). The $2$-dimensional map $T$ considered here is assumed to be only piecewise continuous (yet the primary symmetry $S$ is assumed to be continuous, so the results and techniques of \cite{mackay1993renormalisation} apply and the proofs are included here for completeness of presentation). 
\begin{defn}
A \emph{symmetry} for an invertible map $T$ is an orientation-reversing involution $S$ such that 
\[
S\circ T\circ S=T^{-1}.
\]
A map $T$ that has a symmetry is called \emph{reversible}
\end{defn}

In physical settings, such symmetry is called time-reversal symmetry, as it conjugates the motion forward in time $T$, to the motion backward in time $T^{-1}$.

Given a map $T$ with a symmetry $S$, the map admits a decomposition into the involutions $S$, $S\circ T$ by $T=S\circ\left(S\circ T\right)$.

The composition of a symmetry $S$ with $T$ is another symmetry (now piecewise continuous), thus any given symmetry $S$ generates a \emph{family of symmetries} denoted by: 
\[
S_{i}\coloneqq T^{i}\circ S\qquad i\in\z
\]

\begin{defn}
Each symmetry has the associated fixed set, called a \emph{symmetry line} 
\[
\Gamma_{i}\coloneqq\Fix\left(S_{i}\right)=\left\{ z:S_{i}\left(z\right)=z\right\} .
\]
\end{defn}

\begin{prop}
For $2$-dimensional area-preserving maps, a symmetry line is a collection
of curves which are locally lines.
\end{prop}

\begin{proof}
At a fixed point $x$ of $S$ we have that $\Id=D_{x}\left(S^{2}\right)=D_{S\left(x\right)}S\circ D_{x}S=\left(D_{x}S\right)^{2}$,
so eigenvalues of $D_{x}S$ are only $1,-1$ and it has local coordinates
where $D_{x}S=\left(\begin{smallmatrix}1 & 0\\
0 & -1
\end{smallmatrix}\right)$. In these coordinates $S:\begin{smallmatrix}x' & = & x+f\left(x,y\right)\\
y' & = & -y+g\left(x,y\right)
\end{smallmatrix}$, $f,g$ of higher order, and $S^{2}=\Id$ gives 
\[
S^{2}\left(\begin{smallmatrix}x\\
y
\end{smallmatrix}\right)=S\left(\begin{smallmatrix}x+f\left(x,y\right)\\
-y+g\left(x,y\right)
\end{smallmatrix}\right)=\left(\begin{smallmatrix}x+f\left(x,y\right)+f\circ S\left(x,y\right)\\
y-g\left(x,y\right)+g\circ S\left(x,y\right)
\end{smallmatrix}\right)=\left(\begin{smallmatrix}x\\
y
\end{smallmatrix}\right)
\]

so we have that $f\left(x,y\right)=-f\circ S\left(x,y\right)=-f\left(x',y'\right)$
and $g\left(x,y\right)=g\circ S\left(x,y\right)=g\left(x',y'\right)$
where $\left(x',y'\right)\coloneqq S\left(x,y\right)$.

Now, define the \emph{symmetry coordinates} of the symmetry $S$ by
$X=x+\frac{1}{2}f\left(x,y\right)$ and $Y=y-\frac{1}{2}g\left(x,y\right)$.
These coordinates reduce $S$ to the form $S\left(X,Y\right)=\left(X,-Y\right)$:
\[
\begin{smallmatrix}x' & = & x+f\left(x,y\right)\\
y' & = & -y+g\left(x,y\right)
\end{smallmatrix}\rightarrow\begin{smallmatrix}x'-\frac{1}{2}f\left(x,y\right) & = & x+\frac{1}{2}f\left(x,y\right)\\
y'-\frac{1}{2}g\left(x,y\right) & = & -\left(y-\frac{1}{2}g\left(x,y\right)\right)
\end{smallmatrix}\to\begin{smallmatrix}x'+\frac{1}{2}f\left(x',y'\right) & = & x+\frac{1}{2}f\left(x,y\right)\\
y'-\frac{1}{2}g\left(x',y'\right) & = & -\left(y-\frac{1}{2}g\left(x,y\right)\right)
\end{smallmatrix}\to\begin{smallmatrix}X' & = & X\\
Y' & = & -Y
\end{smallmatrix}
\]

Thus, locally in these coordinates, $\Fix\left(S\right)$ is the $Y=0$
line.

When the fixed point $x$ is at a discontinuity of the symmetry map
$S$ we may still define $D_{x}$ on a half neighborhood of $x$,
where $S$ is continuous. In such cases the symmetry lines are scattered
curves in space.
\end{proof}

\begin{prop}
Intersection of symmetry lines $\Gamma_{j},\Gamma_{k}$ is a periodic
point of the map $T$, with period dividing $\left|j-k\right|$
\end{prop}

\begin{proof}
Let $x$ be such an intersection point, $T^{j-k}\left(x\right)=S_{j}\left(S_{k}\left(x\right)\right)=x$
\end{proof}

The following proposition shows it is enough to find the first two
symmetry lines $\Gamma_{0}$ and $\Gamma_{1}$, as the other symmetry
lines are their images under $T$ and $T^{-1}$.
\begin{prop}  \label{prop:image-of-symm-line}
\label{prop:symm-lines-are-images-of-primary}$T^{n}\Gamma_{i}=\Gamma_{2n+i}$
\end{prop}

\begin{proof}
$S_{i}\left(x\right)=x\Leftrightarrow T^{n}\circ S_{i}\left(x\right)=T^{n}\left(x\right)\Leftrightarrow S_{i+n}\left(x\right)=T^{n}\left(x\right)\Leftrightarrow S_{i+2n}\left(T^{n}\left(x\right)\right)=T^{n}\left(x\right)$
\end{proof}

The reversibility implies that the symmetry sends any $q$-periodic
orbit 
\[
\begin{matrix}x, & T\left(x\right), & \dots\,, & T^{q-1}\left(x\right), & T^{q}\left(x\right)=x\end{matrix}
\]
 to a $q$-periodic orbit (in reversed order)
\[
\begin{matrix}S\left(x\right), & S\circ T\left(x\right), & \dots\,, & S\circ T^{q-1}\left(x\right), & S\circ T^{q}\left(x\right)=S\left(x\right)\\
\verteq & \verteq &  & \verteq & \verteq\qquad\\
S\left(x\right), & T^{-1}\left(S\left(x\right)\right), & \dots\,, & T^{-\left(q-1\right)}\left(S\left(x\right)\right), & T^{-q}\left(S\left(x\right)\right)=S\left(x\right)
\end{matrix}
\]
when the two orbits coincide we say that the orbit is a \emph{symmetric
periodic orbit}, on the other case, we call the two disjoint orbits
a \emph{related pair of non-symmetric periodic orbit}s. The following
shows that the search for symmetric periodic orbits of area-preserving
maps can be reduced to a one-dimensional search along the primary
symmetry lines $\Gamma_{0}$ and $\Gamma_{1}$.
\begin{prop}\label{prop:symperiodicpointsonsymlines}
Symmetric periodic orbits of even period either have exactly two points
on each of the even symmetry lines $\Gamma_{2i}$ or have exactly
two points on each of the odd symmetry lines $\Gamma_{2i+1}$. Symmetric
periodic orbits of odd period have exactly one point on each of the
symmetry lines.
\end{prop}

\begin{proof}
Let $x$ be a symmetric periodic point of period $q$. We can replace
the primary symmetry $S$ with any other symmetry on its family $S_{i}=T^{i}\circ S$,
so without loss of generality we may assume that $x$ is on the symmetry
line of $S$ 
\[
S\left(x\right)=x.
\]
We show that the period is even if and only if there is another orbit
point on the same symmetry line. When the period $q$ is even, the
orbit point $T^{\frac{q}{2}}\left(x\right)$ is on the same symmetry
line: 
\[
S\left(T^{\frac{q}{2}}\left(x\right)\right)=T^{-\frac{q}{2}}\left(S\left(x\right)\right)=T^{-\frac{q}{2}}\left(x\right)=T^{-\frac{q}{2}}\left(T^{q}\left(x\right)\right)=T^{\frac{q}{2}}\left(x\right)
\]

For proving the converse, assume there is another orbit point on the
same symmetry line, namely there exists a positive integer $k<q$ such
that $T^{k}\left(x\right)\neq x$ and $S\left(T^{k}\left(x\right)\right)=T^{k}\left(x\right)$
then:
\[
T^{-k}\left(S\left(x\right)\right)=T^{k}\left(x\right)\Rightarrow T^{-k}\left(x\right)=T^{k}\left(x\right)\Rightarrow T^{2k}\left(x\right)=x
\]
thus the period of $x$ divides $2k$, namely exists an integer $k'$
such that $2k=k' q$, as $k<q$ we get that
\[
k' q<2q
\]
so $k'=1$, then the period is even $q=2k$. Moreover, we have shown
that the only other orbit point on the symmetry line is $T^{\frac{q}{2}}\left(x\right)$.
 So the primary symmetry line $\Gamma_{0}$ has exactly two orbit
points, from Proposition \ref{prop:symm-lines-are-images-of-primary}
it follows that each symmetry line $\Gamma_{2i}$ has exactly two
orbit points.

Similarly, the period is odd if and only if there is another orbit
point on the next symmetry line, a fixed point of $T\circ S$. When
the period is odd, the orbit point $T^{\frac{q+1}{2}}\left(x\right)$
is on the next symmetry line $\Fix\left(T\circ S\right)$: 
\[
T\circ S\left(T^{\frac{q+1}{2}}\left(x\right)\right)=T^{1-\frac{q+1}{2}}\left(S\left(x\right)\right)=T^{\frac{1-q}{2}}\left(x\right)=T^{\frac{1-q}{2}}\left(T^{q}\left(x\right)\right)=T^{\frac{q+1}{2}}\left(x\right)
\]
For proving the converse, suppose the period $q$ is even but there
is another orbit point on the next symmetry line, namely exist $k$
such that $T^{k}\left(x\right)\neq x$ and $T\circ S\left(T^{k}\left(x\right)\right)=T^{k}\left(x\right)$
then: 
\[
T^{1-k}\left(S\left(x\right)\right)=T^{k}\left(x\right)\Rightarrow T^{1-k}\left(x\right)=T^{k}\left(x\right)\Rightarrow T^{2k-1}\left(x\right)=x
\]
so the period $q$ which is even divides the odd number $2k-1$, contradiction. 

By applying the same claim iteratively it follows that for an orbit
of odd period, there is exactly one point on each of the symmetry
lines.
\end{proof}

\begin{example}
The generalized standard map $T:\begin{cases}
x'=x+g\left(y\right)\\
y'=y+f\left(x'\right)
\end{cases}$is a class of area-preserving maps. When $f$ is antisymmetric it admits the
symmetry $S\left(\begin{smallmatrix}x\\
y
\end{smallmatrix}\right)=\left(\begin{smallmatrix}-x\\
y-f\left(x\right)
\end{smallmatrix}\right)$ and when $g$ is antisymmetric, it admits the symmetry $\tilde S\left(\begin{smallmatrix}x\\
y
\end{smallmatrix}\right)=\left(\begin{smallmatrix}x\\
-y+f\left(x\right)
\end{smallmatrix}\right)$.

\end{example}

\begin{example}
\emph{\label{exa:std-map}Chirikov standard map} \cite{chirikov1979universal},
is the map on the torus $\qfrac{\r^{2}}{\z^{2}}$ defined by
\[
T:\begin{cases}
x'=x+y\mod 1\\
y'=y+\eps\sin\left(2\pi x'\right)\mod 1
\end{cases}
\]

The standard map arises in various physical contexts, including the
study of magnetic field lines, particle accelerators, and kicked rotors.
It serves as a canonical model for the transition from regular to
chaotic dynamics in Hamiltonian systems. For a detailed account, see
\cite{lichtenberg1992regular,meyer2017introduction}. It is reversible,
admitting two different symmetries (``doubly reversible'') \cite{mackay1993renormalisation}:
\[
S\left(\begin{matrix}x\\
y
\end{matrix}\right)=\left(\begin{matrix}-x\\
y-\eps\sin\left(2\pi x\right)
\end{matrix}\right),\quad\tilde S\left(\begin{matrix}x\\
y
\end{matrix}\right)=\left(\begin{matrix}x\\
-y+\eps\sin\left(2\pi x\right)
\end{matrix}\right)
\]

The symmetry lines of the standard map were discussed specifically
in \cite{pina1987symmetry} (where they follow \cite{mackay1993renormalisation}).
They are given by 
{\small
\begin{align*}
\Gamma_{0}\! & \coloneqq\!\Fix\left(S_{0}\right)\!=\!\left\{ 0,\frac{1}{2}\right\} \times\left[0,1\right) & \tilde{\Gamma}_{0}\! & \coloneqq\!\Fix\left(\tilde S_{0}\right)\!=\!\left\{ \left(x,\frac{\eps}{2}\sin\left(2\pi x\right)\right):x\!\in\!\left[0,1\right)\right\} \\
\Gamma_{1}\! & \coloneqq\!\Fix\left(S_{1}\right)\!=\!\left\{ \left(x,2x+\eps\sin\left(2\pi x\right)\right):x\!\in\!\left[0,1\right)\right\}  & \tilde{\Gamma}_{1}\! & \coloneqq\!\Fix\left(\tilde S_{1}\right)\!=\!\left\{ \left(x,\eps\sin\left(2\pi x\right)\right):x\!\in\!\left[0,1\right)\right\} 
\end{align*}
}
Figure \ref{fig:std-map} demonstrates elliptic and hyperbolic periodic
points of the standard map occurring at the intersection points of
the symmetry lines. Interestingly, certain intersection points do
not exhibit a visible resonance island or stable and unstable manifolds.
Proposition \ref{prop:fiem-sym-lines-0-min-1} provides an explicit formula
for the symmetry lines of perturbed families of IEMs, also specifically
for the standard map when viewed as such a map. 

\begin{figure}[h]
\begin{centering}
\includegraphics[width=0.5\textwidth]{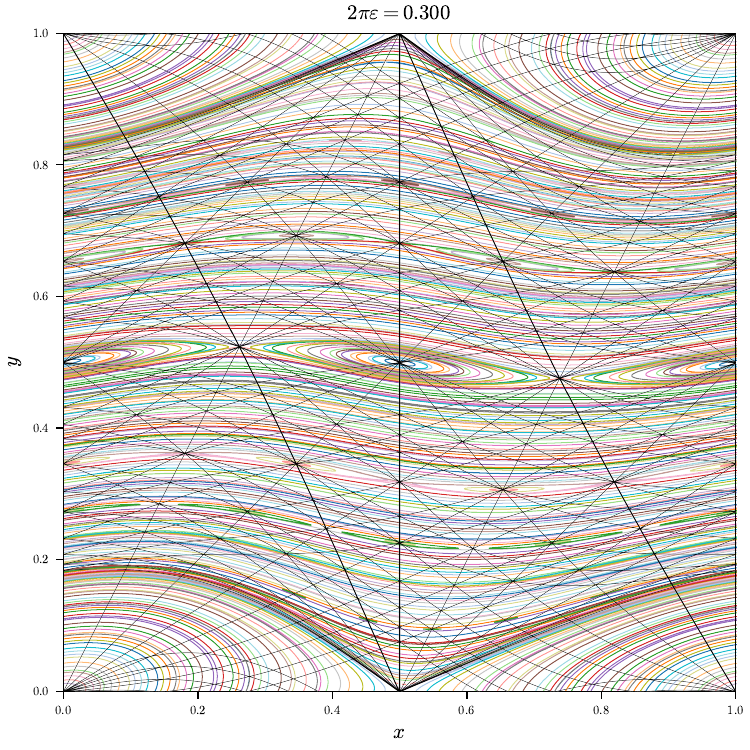}\includegraphics[width=0.5\textwidth]{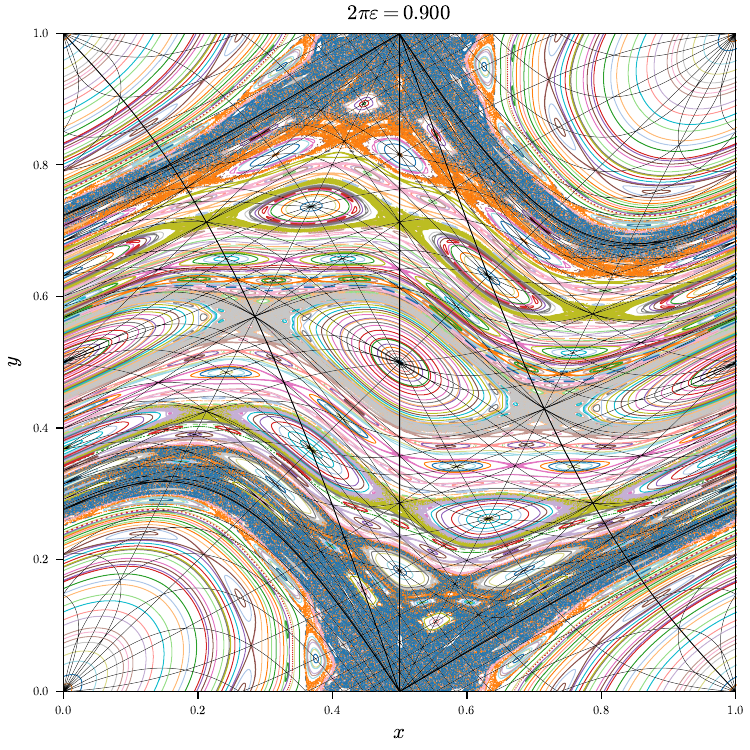}
\par\end{centering}
\centering{}\caption{\label{fig:std-map}The standard map for different perturbation values.
The symmetry lines $\Gamma_{-7}\cds\Gamma_{7}$ are shown
in black, with the primary symmetry lines $\Gamma_{0},\Gamma_{1}$
in bold. Trajectories are plotted in different colors.}
 
\end{figure}
\end{example}

\section{Symmetric and Reversible Interval Exchange Maps \label{sec:symm-and-rever}}
\begin{defn}
\emph{A symmetric $d$--IEM} or $d$--CEM is an exchange map that
reverses the order of intervals, that is $\pi_{1}\circ\pi_{0}^{-1}\left(i\right)=d+1-i$
for any $i$. See Figure \ref{fig:symm-iem}.

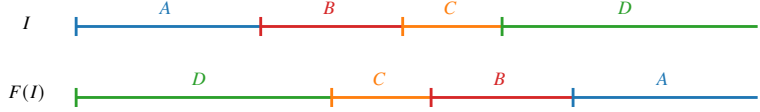
\begin{figure}[h]
\centering{}\resizebox{0.8\columnwidth}{!}{%

\tikzset{every picture/.style={line width=0.75pt}} %set default line width to 0.75pt        

\begin{tikzpicture}[x=0.75pt,y=0.75pt,yscale=-1,xscale=1]
%uncomment if require: \path (0,542); %set diagram left start at 0, and has height of 542

%Straight Lines [id:da28433781681133274] 
\draw [color={rgb, 255:red, 31; green, 119; blue, 180 }  ,draw opacity=1 ][line width=1.5]    (60,70) -- (190,70) ;
\draw [shift={(60,70)}, rotate = 180] [color={rgb, 255:red, 31; green, 119; blue, 180 }  ,draw opacity=1 ][line width=1.5]    (0,6.71) -- (0,-6.71)   ;
%Straight Lines [id:da8630689511221543] 
\draw [color={rgb, 255:red, 214; green, 39; blue, 40 }  ,draw opacity=1 ][line width=1.5]    (190,70) -- (290,70) ;
\draw [shift={(190,70)}, rotate = 180] [color={rgb, 255:red, 214; green, 39; blue, 40 }  ,draw opacity=1 ][line width=1.5]    (0,6.71) -- (0,-6.71)   ;
%Straight Lines [id:da5250747170253607] 
\draw [color={rgb, 255:red, 255; green, 127; blue, 14 }  ,draw opacity=1 ][line width=1.5]    (290,70) -- (360,70) ;
\draw [shift={(290,70)}, rotate = 180] [color={rgb, 255:red, 255; green, 127; blue, 14 }  ,draw opacity=1 ][line width=1.5]    (0,6.71) -- (0,-6.71)   ;
%Straight Lines [id:da3474122892687759] 
\draw [color={rgb, 255:red, 44; green, 160; blue, 44 }  ,draw opacity=1 ][line width=1.5]    (360,70) -- (540,70) ;
\draw [shift={(360,70)}, rotate = 180] [color={rgb, 255:red, 44; green, 160; blue, 44 }  ,draw opacity=1 ][line width=1.5]    (0,6.71) -- (0,-6.71)   ;
%Straight Lines [id:da7015640846968949] 
\draw [color={rgb, 255:red, 44; green, 160; blue, 44 }  ,draw opacity=1 ][fill={rgb, 255:red, 107; green, 0; blue, 193 }  ,fill opacity=1 ][line width=1.5]    (60,120) -- (240,120) ;
\draw [shift={(60,120)}, rotate = 180] [color={rgb, 255:red, 44; green, 160; blue, 44 }  ,draw opacity=1 ][line width=1.5]    (0,6.71) -- (0,-6.71)   ;
%Straight Lines [id:da1337068379907661] 
\draw [color={rgb, 255:red, 255; green, 127; blue, 14 }  ,draw opacity=1 ][line width=1.5]    (240,120) -- (310,120) ;
\draw [shift={(240,120)}, rotate = 180] [color={rgb, 255:red, 255; green, 127; blue, 14 }  ,draw opacity=1 ][line width=1.5]    (0,6.71) -- (0,-6.71)   ;
%Straight Lines [id:da11601998235824784] 
\draw [color={rgb, 255:red, 214; green, 39; blue, 40 }  ,draw opacity=1 ][line width=1.5]    (310,120) -- (315,120) -- (410,120) ;
\draw [shift={(310,120)}, rotate = 180] [color={rgb, 255:red, 214; green, 39; blue, 40 }  ,draw opacity=1 ][line width=1.5]    (0,6.71) -- (0,-6.71)   ;
%Straight Lines [id:da18137279973936293] 
\draw [color={rgb, 255:red, 31; green, 119; blue, 180 }  ,draw opacity=1 ][line width=1.5]    (410,120) -- (540,120) ;
\draw [shift={(410,120)}, rotate = 180] [color={rgb, 255:red, 31; green, 119; blue, 180 }  ,draw opacity=1 ][line width=1.5]    (0,6.71) -- (0,-6.71)   ;

% Text Node
\draw (25.21,67.2) node    {$I$};
% Text Node
\draw (25.21,117.2) node    {$F( I)$};
% Text Node
\draw (122.56,57) node  [color={rgb, 255:red, 31; green, 119; blue, 180 }  ,opacity=1 ]  {$A$};
% Text Node
\draw (238.13,57) node  [color={rgb, 255:red, 214; green, 39; blue, 40 }  ,opacity=1 ]  {$B$};
% Text Node
\draw (323.69,57) node  [color={rgb, 255:red, 255; green, 127; blue, 14 }  ,opacity=1 ]  {$C$};
% Text Node
\draw (446.63,57) node  [color={rgb, 255:red, 44; green, 160; blue, 44 }  ,opacity=1 ]  {$D$};
% Text Node
\draw (146.63,107) node  [color={rgb, 255:red, 44; green, 160; blue, 44 }  ,opacity=1 ]  {$D$};
% Text Node
\draw (273.69,107) node  [color={rgb, 255:red, 255; green, 127; blue, 14 }  ,opacity=1 ]  {$C$};
% Text Node
\draw (358.13,107) node  [color={rgb, 255:red, 214; green, 39; blue, 40 }  ,opacity=1 ]  {$B$};
% Text Node
\draw (472.56,107) node  [color={rgb, 255:red, 31; green, 119; blue, 180 }  ,opacity=1 ]  {$A$};

\end{tikzpicture}
}\caption{\label{fig:symm-iem}Symmetric $4$--IEM, $\pi=\left(\begin{matrix}A & B & C & D\\
D & C & B & A
\end{matrix}\right)$}
\end{figure}
\end{defn}

In this section, we investigate how the reversibility of an IEM relates to symmetric IEMs. The only natural candidate for a symmetry map of the IEM domain is the reflection of the interval about its center, which we denote by $R$. Explicitly, if the interval is $\left[0,1\right)$ the symmetry map is $R\left(x\right)=-x\mod 1$. Henceforth we call an IEM reversible when $R$ is a symmetry for the IEM. In the case of a CEM, the natural candidate is again the reflection of the circle, but now there is a freedom in the choice of the reflection center. We show here that the natural choice for symmetric CEM is the reflection of the circle about the $\frac{\theta_{0}+\theta_{1}}{2}$ angle, which we also denote by $R$, namely $R\left(x\right)=-x+\theta_{0}+\theta_{1}\mod 1$.

Throughout this section, when we apply $R$ to a subinterval, we will use equality of subintervals; such equality is a set equality for the \emph{interior} of the intervals.
For example, for the IEM interval domain $I$ and the reflection we have $R\left(I\right)=I$, only the interiors of $I$ and $R\left(I\right)$ are equal (as $I$ is closed from the right and $R\left(I\right)$ closed from the left), and the pointwise equality does not hold ($R\left(x\right)=x$ holds only for the midpoint of the interval).
\begin{thm}
\label{thm:iem-symm-iff-R-F}Let $F$ be a $d$--IEM of elemental subintervals $J_{1}\cds J_{d}$ and $R$ be the reflection of the interval, then $F$ is symmetric if and only if $F\left(J_{i}\right)=R\left(J_{i}\right)$ for $i=1\cds d$.
\end{thm}

\begin{proof}
$\left(\boldsymbol{\Rightarrow}\right)$\textbf{ }Suppose $F$ is symmetric, then the first interval $J_{1}$ is mapped to become the last interval, so $F\left(J_{1}\right)=R\left(J_{1}\right)$, the second interval is mapped adjacent to the left of $F\left(J_{1}\right)$, so $F\left(J_{2}\right)=R\left(J_{2}\right)$, and this pattern continues for all remaining subintervals.

$\left(\Leftarrow\right)$ Conversely, suppose $F$ fulfills $F\left(J_{i}\right)=R\left(J_{i}\right)$ for $i=1\cds d$, then the first interval must be mapped to become the last interval, the second must be mapped to become the one before the last, and so on, reversing the original order. Therefore, $F$ is symmetric.
\end{proof}

\begin{cor}
Let $F$ be a $d$--CEM of elemental arcs $J_{1}\cds J_{d}$ and $R$ be the reflection of the circle about the $\frac{\theta_{0}+\theta_{1}}{2}$ angle, then $F$ is symmetric if and only if $F\left(J_{i}\right)=R\left(J_{i}\right)$ for $i=1\cds d$.
\end{cor}

\begin{proof}
Denote the global rotation of the circle by $\tau_{\theta}\left(x\right)\coloneqq x+\theta\mod 1$ and the reflection of the circle about the angle $\theta$ by $R_{\theta}\left(x\right)\coloneqq-x+2\theta\mod 1$.
Suppose $F$ is defined by $\pi,\lambda,\theta_{0},\theta_{1}$ and denote $F'$ to be the IEM on the $\left[0,1\right)$ interval defined by $\pi,\lambda$, then
\[
F=\tau_{\theta_{1}}\circ F'\circ\tau_{-\theta_{0}}
\]
Specifically, for any elemental arc $J_{i}$ we have that 
\[
F\left(J_{i}\right)=\tau_{\theta_{1}}\circ F'\circ\tau_{-\theta_{0}}\left(J_{i}\right)
\]
By $R_{\theta}=\tau_{2\theta}\circ R_{0}=R_{0}\circ\tau_{-2\theta}$ we have that
\[
R_{\frac{\theta_{0}+\theta_{1}}{2}}\left(J_{i}\right)=R_{0}\circ\tau_{-\theta_{1}-\theta_{0}}\left(J_{i}\right)=\tau_{\theta_{1}}\circ R_{0}\circ\tau_{-\theta_{0}}\left(J_{i}\right)
\]
The corollary now follows from the previous theorem.
\end{proof}

For example, the $3$--CEM in Figure \ref{fig:cem} reverses the order of intervals and is symmetric with respect to the reflection about $\frac{\theta_{0}+\theta_{1}}{2}$, illustrated by the gray dashed line. Note that any $3$--CEM with no degenerate discontinuity points is a symmetric $3$--CEM.
\begin{cor}
\label{cor:RF-local-reflection}Let $F$ be a symmetric $d$--IEM of elemental subintervals $J_{1}\cds J_{d}$ and $R$ be the reflection of the interval, then $R\circ F=F^{-1}\circ R$ is the local reflection of each of the elemental subintervals, i.e. it maps $J_{i}$ to itself, reflected around its mid-point, for any $i=1\cds d$.
\end{cor}

\begin{proof}
$R\circ F=F^{-1}\circ R$ by reversibility. According to the previous theorem, it maps the elemental subintervals to themselves. On each elemental subinterval, the map is a composition of a translation and a reflection, thus it is the local reflection of the elemental subintervals around their midpoints.
\end{proof}

We call an IEM that flips the position of pairs of equal-length intervals and keeps the rest of the intervals in their place (identity elsewhere) a \emph{swap}.
\begin{thm}
\label{thm:IEM-reversible-iff-symm-swap}An IEM (CEM) is reversible if and only if it is a composition of a symmetric IEM (CEM) with a swap of some of its equal-length subintervals.
\end{thm}

\begin{proof}
$\left(\Leftarrow\right)$ Assume the map $G=F\circ W$ is the composition of $F$, a symmetric IEM and $W$, a swap of some of its (equal-length) elemental subintervals. 

Denote $W\left(J_{i}\right)=J_{\sigma\left(i\right)}$, since $W$ is a swap we have that $W^{-1}=W$.

From the previous theorem the maps $R\circ F$ and $F^{-1}\circ R$ map the interval $J_{i}$ to itself reflected, hence for any given interval $J_{i}$ we have that 
\[
R\circ G\left(J_{i}\right)=R\circ F\circ W\left(J_{i}\right)=R\circ F\left(J_{\sigma\left(i\right)}\right)=J_{\sigma\left(i\right)}
\]
and 
\[
G^{-1}\circ R\left(J_{i}\right)=W^{-1}\circ F^{-1}\circ R\left(J_{i}\right)=W\circ F^{-1}\circ R\left(J_{i}\right)=J_{\sigma\left(i\right)}
\]
We have that for an arbitrary interval $J_{i}$, both $R\circ G$ and $G^{-1}\circ R$ map it to $J_{\sigma\left(i\right)}$ by reflection and translation, then we have that pointwise, reversibility holds:
\[
G^{-1}\circ R=R\circ G
\]
$\left(\boldsymbol{\Rightarrow}\right)$ Conversely, assume that the map $G$ is reversible, then for any elemental subinterval $J_{i}$ we
have 
\[
G^{-1}\left(R\left(x\right)\right)=R\left(G\left(x\right)\right)\quad\forall x\in J_{i}
\]
Then $G^{-1}$ maps by translation the interval $R\left(J_{i}\right)$ to the interval $R\left(G\left(J_{i}\right)\right)$. Assuming there are no degenerate discontinuities in $G$, $R\left(J_{i}\right)$ must be contained in some elemental subinterval of $G^{-1}$, i.e. there is a subinterval $J_{i'}$ such that 
\[
R\left(J_{i}\right)\subseteq G\left(J_{i'}\right)
\]
We view this relation between $i$ and $i'$ as an arrow in a directed graph of vertices representing the elemental subintervals $1\cds d$, with only one outgoing edge from every vertex. Such a graph is a disjoint collection of cycles, with possibly some trails entering the cycles. We now show that all cycles of the graph are of length 1 or 2, with no trails entering the cycles. 

A cycle of vertices $i_{1}\cds i_{m}$ implies the following relations
\[
R\left(J_{i_{1}}\right)\subseteq G\left(J_{i_{2}}\right),R\left(J_{i_{2}}\right)\subseteq G\left(J_{i_{3}}\right)\cds R\left(J_{i_{m}}\right)\subseteq G\left(J_{i_{1}}\right)
\]
applying $G\circ R$ and using the reversibility we get that also
\[
G\left(J_{i_{1}}\right)\subseteq R\left(J_{i_{2}}\right),G\left(J_{i_{2}}\right)\subseteq R\left(J_{i_{3}}\right)\cds G\left(J_{i_{m}}\right)\subseteq R\left(J_{i_{1}}\right)
\]
Combining the obtained inclusions, we find that when $m$ is even, the following chain of inclusions hold:
\begin{align*}
R\left(J_{i_{1}}\right) & \subseteq R\left(J_{i_{3}}\right)\subseteq\cdots\subseteq R\left(J_{i_{m-1}}\right)\subseteq R\left(J_{i_{1}}\right)\\
R\left(J_{i_{2}}\right) & \subseteq R\left(J_{i_{4}}\right)\subseteq\cdots\subseteq R\left(J_{i_{m}}\right)\subseteq R\left(J_{i_{2}}\right)
\end{align*}
Each chain of inclusions forms a closed loop, implying equality of the sets:
\begin{align*}
J_{i_{1}} & =J_{i_{3}}=\cdots=J_{i_{m-1}}=J_{i_{1}} & \Rightarrow &  & i_{1} & =i_{3}=\cdots=i_{m-1}\\
J_{i_{2}} & =J_{i_{4}}=\cdots=J_{i_{m}}=J_{i_{2}} & \Rightarrow &  & i_{2} & =i_{4}=\cdots=i_{m}
\end{align*}
Thus $m$ must be $2$, and we have $R\left(J_{i_{1}}\right)=G\left(J_{i_{2}}\right)$ and $R\left(J_{i_{2}}\right)=G\left(J_{i_{1}}\right)$. There are no trails entering the two vertices of the cycle as $R\left(J_{i_{1}}\right)$ and $R\left(J_{i_{2}}\right)$ are the only reflections of subintervals contained in $G\left(J_{i_{2}}\right)$ and $G\left(J_{i_{1}}\right)$ respectively. We also note that $\lambda_{i_{1}}=\lambda_{i_{2}}$.

When $m$ is odd, we get a single loop of inclusions
\[
R\left(J_{i_{1}}\right)\subseteq R\left(J_{i_{3}}\right)\subseteq\cdots\subseteq R\left(J_{i_{m}}\right)\subseteq R\left(J_{i_{2}}\right)\subseteq R\left(J_{i_{4}}\right)\subseteq\cdots\subseteq R\left(J_{i_{m-1}}\right)\subseteq R\left(J_{1}\right)
\]
so similarly, $m$ must be $1$, and $R\left(J_{i_{1}}\right)=G\left(J_{i_{1}}\right)$ and there are no trails entering the cycle like before.

Let $W$ be the IEM defined by $W\left(J_{i}\right)\coloneqq J_{i'}$, it is a swap, swapping between intervals participating in the cycles of length $2$, and identity on the other intervals of cycles of length $1$.

Then the IEM $G\circ W^{-1}$ has the following property 
\[
G\circ W^{-1}\left(J_{i}\right)=G\left(J_{i'}\right)=R\left(J_{i}\right)
\]
By Theorem \ref{thm:iem-symm-iff-R-F}, it is a symmetric IEM, thus $G=F\circ W$, where $F$ is a symmetric IEM and $W$ is a swap map of some of its equal-length subintervals.

The proof in the case of a CEM follows the same argument.
\end{proof}

\begin{rem}
The swap map W in the proof of the last theorem is a swap defined on the intervals $J_{1}\cds J_{n}$, but we could have an equivalent formulation of $G$ as $G=\tilde W \circ F$ where $\tilde W$ is a swap defined on the intervals $F\left(J_{1}\right)\cds F\left(J_{n}\right)$.
\end{rem}

\begin{cor}
Symmetric IEM are reversible, but the converse does not hold in general.
\end{cor}
The corollary refers to cases in which there are equi-length elemental subintervals, where the swap component of Theorem \ref{thm:IEM-reversible-iff-symm-swap} is not necessarily the identity.
\begin{cor}
Let $G$ be a symmetric IEM, then for any $n$, $G^{n}$ is a composition of a swap and a symmetric IEM, thus periodic intervals of a symmetric IEM are points where the corresponding swap is according to a reflection.
\end{cor}

\begin{proof}
$G^{n}$ is reversible by the same symmetry of $G$. By the previous theorem $G^{n}$ is a composition of a swap $W_{n}$ and a symmetric IEM $F_{n}$. Let $J_{1}^{\left(n\right)}\cds J_{k_{n}}^{\left(n\right)}$ be the elemental subintervals of $G^{n}=F_{n}\circ W_{n}$ then the $i$'th subinterval is fixed when
\begin{align*}
G^{n}\left(J_{i}^{\left(n\right)}\right) & =J_{i}^{\left(n\right)},
\end{align*}
namely,
\begin{align*}
F_{n}\left(W_{n}\left(J_{i}^{\left(n\right)}\right)\right) & =J_{i}^{\left(n\right)}.
\end{align*}
Hence, by Corollary \ref{cor:RF-local-reflection}
\begin{align*}
W_{n}\left(J_{i}^{\left(n\right)}\right) & =R\left(J_{i}^{\left(n\right)}\right)
\end{align*}
\end{proof}

\begin{defn}
A \emph{symmetric periodic interval} of an IEM G is a periodic interval $J$ such that its reflection belongs to its orbit; that is, there exists $k$ such that $R\left(J\right)=G^{k}\left(J\right)$.
\end{defn}

\begin{prop}
Let $G$ be a reversible IEM, any $q$-periodic interval is either symmetric, or its reflection forms another $q$-periodic interval, and then their orbits are disjoint and exchanged by the reflection.
\end{prop}

\begin{proof}
Let $J$ be a $q$-periodic interval $J$ of a reversible IEM, either $R\left(J\right)$ is on its orbit, in which case $J$ is symmetric.
If $R\left(J\right)$ forms another orbit, then
\[
G^{q}\circ R\left(J\right)=R\circ G^{-q}\left(J\right)=R\left(J\right)
\]
thus $R\left(J\right)$ is $q$-periodic. The two orbits are disjoint like in Proposition \ref{prop:periodic-orbit-disjoint}.
\end{proof}

\begin{rem}
\label{rem:midpoint-of-symm-periodic-interval-is-symm}Only the midpoint of a symmetric periodic interval is a symmetric periodic point, as $R$ sends the periodic interval to another interval along its orbit, with reverse orientation.
\end{rem}

Fixed intervals of a symmetric IEM must be a centrally symmetric interval. In particular, a symmetric IEM cannot have a pair of non-symmetric fixed intervals. We show below that this property generalizes to periodic intervals of any period.
\begin{thm}
\label{thm:symm-iem-no-non-symm}A symmetric IEM (or a symmetric CEM) cannot have a non-symmetric $q$-periodic interval.
\end{thm}

\begin{proof}
Suppose that a symmetric IEM $F$ has a non-symmetric $q$-periodic interval $A_{0}$, then its reflection, $B_{0}\coloneqq R\left(A_{0}\right)$, is another $q$-periodic interval. Denote their orbits
\begin{align*}
 & A_{0},A_{1}\cds A_{q-1},A_{q}=A_{0} & \text{where }A_{i}=F^{i}\left(A_{0}\right)\\
 & B_{0},B_{1}\cds B_{q-1},B_{q}=B_{0} & \text{where }B_{i}=F^{i}\left(B_{0}\right)
\end{align*}
These two periodic orbits are disjoint. According to the reversibility, the reflection $R$ relates the two orbits:
\begin{align*}
R\left(A_{i}\right) & =R\circ F^{i}\left(A_{0}\right)=F^{-i}\circ R\left(A_{0}\right)=F^{-i}\left(B_{0}\right)=B_{q-i}\\
R\left(B_{i}\right) & =A_{q-i}
\end{align*}
Denote $L\coloneqq F^{-1}\circ R=R\circ F$, according to Corollary \ref{cor:RF-local-reflection}, $L\left(A_{i}\right)$ is the local reflection of $A_{i}$ through the center of the elemental subinterval containing it. $L$ gives an additional relation between the two orbits:
\begin{align*}
L\left(A_{i}\right) & =F^{-1}\circ R\circ F^{i}\left(A_{0}\right)=F^{-i}\circ F^{-1}\circ R\left(A_{0}\right)=F^{-i-1}\left(B_{0}\right)=B_{q-i-1}\\
L\left(B_{i}\right) & =A_{q-i-1}
\end{align*}

We will reach a contradiction by showing that both orbits are contained within a single, larger periodic orbit; Figure \ref{fig:hypothetical-non-symm} illustrates this situation.

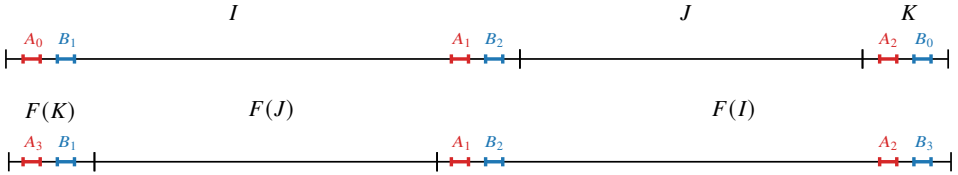
\begin{figure}[h]
\centering
\centering{}\resizebox{\columnwidth}{!}{%

\tikzset{every picture/.style={line width=0.75pt}} %set default line width to 0.75pt        

\begin{tikzpicture}[x=0.75pt,y=0.75pt,yscale=-1,xscale=1]
%uncomment if require: \path (0,300); %set diagram left start at 0, and has height of 300

%Straight Lines [id:da0902353133535605] 
\draw    (30,50) -- (330,50) ;
\draw [shift={(330,50)}, rotate = 180] [color={rgb, 255:red, 0; green, 0; blue, 0 }  ][line width=0.75]    (0,5.59) -- (0,-5.59)   ;
\draw [shift={(30,50)}, rotate = 180] [color={rgb, 255:red, 0; green, 0; blue, 0 }  ][line width=0.75]    (0,5.59) -- (0,-5.59)   ;
%Straight Lines [id:da11051178130462058] 
\draw    (330,50) -- (530,50) ;
\draw [shift={(530,50)}, rotate = 180] [color={rgb, 255:red, 0; green, 0; blue, 0 }  ][line width=0.75]    (0,5.59) -- (0,-5.59)   ;
\draw [shift={(330,50)}, rotate = 180] [color={rgb, 255:red, 0; green, 0; blue, 0 }  ][line width=0.75]    (0,5.59) -- (0,-5.59)   ;
%Straight Lines [id:da21861921771278725] 
\draw    (530,50) -- (548.33,50) -- (580,50) ;
\draw [shift={(580,50)}, rotate = 180] [color={rgb, 255:red, 0; green, 0; blue, 0 }  ][line width=0.75]    (0,5.59) -- (0,-5.59)   ;
\draw [shift={(530,50)}, rotate = 180] [color={rgb, 255:red, 0; green, 0; blue, 0 }  ][line width=0.75]    (0,5.59) -- (0,-5.59)   ;
%Straight Lines [id:da49322850830035603] 
\draw    (31.67,110) -- (50,110) -- (81.67,110) ;
\draw [shift={(81.67,110)}, rotate = 180] [color={rgb, 255:red, 0; green, 0; blue, 0 }  ][line width=0.75]    (0,5.59) -- (0,-5.59)   ;
\draw [shift={(31.67,110)}, rotate = 180] [color={rgb, 255:red, 0; green, 0; blue, 0 }  ][line width=0.75]    (0,5.59) -- (0,-5.59)   ;
%Straight Lines [id:da6660082340481395] 
\draw    (81.67,110) -- (281.67,110) ;
\draw [shift={(281.67,110)}, rotate = 180] [color={rgb, 255:red, 0; green, 0; blue, 0 }  ][line width=0.75]    (0,5.59) -- (0,-5.59)   ;
\draw [shift={(81.67,110)}, rotate = 180] [color={rgb, 255:red, 0; green, 0; blue, 0 }  ][line width=0.75]    (0,5.59) -- (0,-5.59)   ;
%Straight Lines [id:da49168813102895514] 
\draw    (281.67,110) -- (581.67,110) ;
\draw [shift={(581.67,110)}, rotate = 180] [color={rgb, 255:red, 0; green, 0; blue, 0 }  ][line width=0.75]    (0,5.59) -- (0,-5.59)   ;
\draw [shift={(281.67,110)}, rotate = 180] [color={rgb, 255:red, 0; green, 0; blue, 0 }  ][line width=0.75]    (0,5.59) -- (0,-5.59)   ;
%Straight Lines [id:da01775208566918307] 
\draw [color={rgb, 255:red, 31; green, 119; blue, 180 }  ,draw opacity=1 ][line width=1.5]    (560,50) -- (570,50) ;
\draw [shift={(570,50)}, rotate = 180] [color={rgb, 255:red, 31; green, 119; blue, 180 }  ,draw opacity=1 ][line width=1.5]    (0,2.68) -- (0,-2.68)   ;
\draw [shift={(560,50)}, rotate = 180] [color={rgb, 255:red, 31; green, 119; blue, 180 }  ,draw opacity=1 ][line width=1.5]    (0,2.68) -- (0,-2.68)   ;
%Straight Lines [id:da000026900799435347622] 
\draw [color={rgb, 255:red, 214; green, 39; blue, 40 }  ,draw opacity=1 ][line width=1.5]    (290,110) -- (300,110) ;
\draw [shift={(300,110)}, rotate = 180] [color={rgb, 255:red, 214; green, 39; blue, 40 }  ,draw opacity=1 ][line width=1.5]    (0,2.68) -- (0,-2.68)   ;
\draw [shift={(290,110)}, rotate = 180] [color={rgb, 255:red, 214; green, 39; blue, 40 }  ,draw opacity=1 ][line width=1.5]    (0,2.68) -- (0,-2.68)   ;
%Straight Lines [id:da005827298526109637] 
\draw [color={rgb, 255:red, 214; green, 39; blue, 40 }  ,draw opacity=1 ][line width=1.5]    (540,110) -- (550,110) ;
\draw [shift={(550,110)}, rotate = 180] [color={rgb, 255:red, 214; green, 39; blue, 40 }  ,draw opacity=1 ][line width=1.5]    (0,2.68) -- (0,-2.68)   ;
\draw [shift={(540,110)}, rotate = 180] [color={rgb, 255:red, 214; green, 39; blue, 40 }  ,draw opacity=1 ][line width=1.5]    (0,2.68) -- (0,-2.68)   ;
%Straight Lines [id:da4831425455476872] 
\draw [color={rgb, 255:red, 214; green, 39; blue, 40 }  ,draw opacity=1 ][line width=1.5]    (40,110) -- (50,110) ;
\draw [shift={(50,110)}, rotate = 180] [color={rgb, 255:red, 214; green, 39; blue, 40 }  ,draw opacity=1 ][line width=1.5]    (0,2.68) -- (0,-2.68)   ;
\draw [shift={(40,110)}, rotate = 180] [color={rgb, 255:red, 214; green, 39; blue, 40 }  ,draw opacity=1 ][line width=1.5]    (0,2.68) -- (0,-2.68)   ;
%Straight Lines [id:da975980150697149] 
\draw [color={rgb, 255:red, 31; green, 119; blue, 180 }  ,draw opacity=1 ][line width=1.5]    (60,110) -- (70,110) ;
\draw [shift={(70,110)}, rotate = 180] [color={rgb, 255:red, 31; green, 119; blue, 180 }  ,draw opacity=1 ][line width=1.5]    (0,2.68) -- (0,-2.68)   ;
\draw [shift={(60,110)}, rotate = 180] [color={rgb, 255:red, 31; green, 119; blue, 180 }  ,draw opacity=1 ][line width=1.5]    (0,2.68) -- (0,-2.68)   ;
%Straight Lines [id:da0033839545960787643] 
\draw [color={rgb, 255:red, 31; green, 119; blue, 180 }  ,draw opacity=1 ][line width=1.5]    (310,110) -- (320,110) ;
\draw [shift={(320,110)}, rotate = 180] [color={rgb, 255:red, 31; green, 119; blue, 180 }  ,draw opacity=1 ][line width=1.5]    (0,2.68) -- (0,-2.68)   ;
\draw [shift={(310,110)}, rotate = 180] [color={rgb, 255:red, 31; green, 119; blue, 180 }  ,draw opacity=1 ][line width=1.5]    (0,2.68) -- (0,-2.68)   ;
%Straight Lines [id:da3095655860270874] 
\draw [color={rgb, 255:red, 31; green, 119; blue, 180 }  ,draw opacity=1 ][line width=1.5]    (560,110) -- (570,110) ;
\draw [shift={(570,110)}, rotate = 180] [color={rgb, 255:red, 31; green, 119; blue, 180 }  ,draw opacity=1 ][line width=1.5]    (0,2.68) -- (0,-2.68)   ;
\draw [shift={(560,110)}, rotate = 180] [color={rgb, 255:red, 31; green, 119; blue, 180 }  ,draw opacity=1 ][line width=1.5]    (0,2.68) -- (0,-2.68)   ;
%Straight Lines [id:da8070519437848598] 
\draw [color={rgb, 255:red, 214; green, 39; blue, 40 }  ,draw opacity=1 ][line width=1.5]    (40,50) -- (50,50) ;
\draw [shift={(50,50)}, rotate = 180] [color={rgb, 255:red, 214; green, 39; blue, 40 }  ,draw opacity=1 ][line width=1.5]    (0,2.68) -- (0,-2.68)   ;
\draw [shift={(40,50)}, rotate = 180] [color={rgb, 255:red, 214; green, 39; blue, 40 }  ,draw opacity=1 ][line width=1.5]    (0,2.68) -- (0,-2.68)   ;
%Straight Lines [id:da9201849043492092] 
\draw [color={rgb, 255:red, 31; green, 119; blue, 180 }  ,draw opacity=1 ][line width=1.5]    (60,50) -- (70,50) ;
\draw [shift={(70,50)}, rotate = 180] [color={rgb, 255:red, 31; green, 119; blue, 180 }  ,draw opacity=1 ][line width=1.5]    (0,2.68) -- (0,-2.68)   ;
\draw [shift={(60,50)}, rotate = 180] [color={rgb, 255:red, 31; green, 119; blue, 180 }  ,draw opacity=1 ][line width=1.5]    (0,2.68) -- (0,-2.68)   ;
%Straight Lines [id:da6531075000458965] 
\draw [color={rgb, 255:red, 214; green, 39; blue, 40 }  ,draw opacity=1 ][line width=1.5]    (290,50) -- (300,50) ;
\draw [shift={(300,50)}, rotate = 180] [color={rgb, 255:red, 214; green, 39; blue, 40 }  ,draw opacity=1 ][line width=1.5]    (0,2.68) -- (0,-2.68)   ;
\draw [shift={(290,50)}, rotate = 180] [color={rgb, 255:red, 214; green, 39; blue, 40 }  ,draw opacity=1 ][line width=1.5]    (0,2.68) -- (0,-2.68)   ;
%Straight Lines [id:da16685931417698285] 
\draw [color={rgb, 255:red, 31; green, 119; blue, 180 }  ,draw opacity=1 ][line width=1.5]    (310,50) -- (320,50) ;
\draw [shift={(320,50)}, rotate = 180] [color={rgb, 255:red, 31; green, 119; blue, 180 }  ,draw opacity=1 ][line width=1.5]    (0,2.68) -- (0,-2.68)   ;
\draw [shift={(310,50)}, rotate = 180] [color={rgb, 255:red, 31; green, 119; blue, 180 }  ,draw opacity=1 ][line width=1.5]    (0,2.68) -- (0,-2.68)   ;
%Straight Lines [id:da3538535362658023] 
\draw [color={rgb, 255:red, 214; green, 39; blue, 40 }  ,draw opacity=1 ][line width=1.5]    (540,50) -- (550,50) ;
\draw [shift={(550,50)}, rotate = 180] [color={rgb, 255:red, 214; green, 39; blue, 40 }  ,draw opacity=1 ][line width=1.5]    (0,2.68) -- (0,-2.68)   ;
\draw [shift={(540,50)}, rotate = 180] [color={rgb, 255:red, 214; green, 39; blue, 40 }  ,draw opacity=1 ][line width=1.5]    (0,2.68) -- (0,-2.68)   ;

% Text Node
\draw (163,23) node    {$I$};
% Text Node
\draw (46,39) node  [font=\footnotesize,color={rgb, 255:red, 214; green, 39; blue, 40 }  ,opacity=1 ]  {$A_{0}$};
% Text Node
\draw (455,80) node    {$F( I)$};
% Text Node
\draw (566,39) node  [font=\footnotesize,color={rgb, 255:red, 31; green, 119; blue, 180 }  ,opacity=1 ]  {$B_{0}$};
% Text Node
\draw (185,80) node    {$F( J)$};
% Text Node
\draw (56,81) node    {$F( K)$};
% Text Node
\draw (296,99) node  [font=\footnotesize,color={rgb, 255:red, 214; green, 39; blue, 40 }  ,opacity=1 ]  {$A_{1}$};
% Text Node
\draw (545,99) node  [font=\footnotesize,color={rgb, 255:red, 214; green, 39; blue, 40 }  ,opacity=1 ]  {$A_{2}$};
% Text Node
\draw (46,99) node  [font=\footnotesize,color={rgb, 255:red, 214; green, 39; blue, 40 }  ,opacity=1 ]  {$A_{3}$};
% Text Node
\draw (427,23) node    {$J$};
% Text Node
\draw (557.5,23) node    {$K$};
% Text Node
\draw (66,99) node  [font=\footnotesize,color={rgb, 255:red, 31; green, 119; blue, 180 }  ,opacity=1 ]  {$B_{1}$};
% Text Node
\draw (315,99) node  [font=\footnotesize,color={rgb, 255:red, 31; green, 119; blue, 180 }  ,opacity=1 ]  {$B_{2}$};
% Text Node
\draw (566,99) node  [font=\footnotesize,color={rgb, 255:red, 31; green, 119; blue, 180 }  ,opacity=1 ]  {$B_{3}$};
% Text Node
\draw (66,39) node  [font=\footnotesize,color={rgb, 255:red, 31; green, 119; blue, 180 }  ,opacity=1 ]  {$B_{1}$};
% Text Node
\draw (296,39) node  [font=\footnotesize,color={rgb, 255:red, 214; green, 39; blue, 40 }  ,opacity=1 ]  {$A_{1}$};
% Text Node
\draw (315,39) node  [font=\footnotesize,color={rgb, 255:red, 31; green, 119; blue, 180 }  ,opacity=1 ]  {$B_{2}$};
% Text Node
\draw (545,39) node  [font=\footnotesize,color={rgb, 255:red, 214; green, 39; blue, 40 }  ,opacity=1 ]  {$A_{2}$};

\end{tikzpicture}
}\caption{\label{fig:hypothetical-non-symm}Hypothetical pair of non-symmetric $3$-periodic intervals ($A_{i}$ and $B_{i+1}$ are actually part of the same periodic interval). Note that $B_{1}=R\left(A_{2}\right)=L\left(A_{1}\right)$ and $B_{2}=R\left(A_{1}\right)=L\left(A_{0}\right)$. The orbits are actually contained in the periodic orbit $K,F\left(K\right),F^{2}\left(K\right),F^{3}\left(K\right)=K$.}
\end{figure}

The map $L$ enforces each elemental subinterval to contain an even number of periodic intervals. 

Let $C_{0},C_{1}\cds C_{2q-1}$ denote the periodic intervals, ordered as they appear along the interval (in Figure \ref{fig:hypothetical-non-symm}, $C_{0}\coloneqq A_{0},C_{1}\coloneqq B_{1},C_{2}\coloneqq A_{1},C_{3}\coloneqq B_{2},C_{4}\coloneqq A_{2},C_{5}\coloneqq B_{0}$).
As the discontinuity points of the IEM are only between even number of periodic intervals we conclude
\begin{enumerate}
\item[a.] For $i=0\cds q-1$, the intervals $C_{2i},C_{2i+1}$ are contained in the same elemental subinterval.
\item[b.] The images of the intervals advance by an even number of steps that depends on their elemental subinterval. Namely, for each $\alpha\in\mathcal{A}$ exists an integer $k_{\alpha}$ such that $F\left(C_{i}\right)=C_{i+2k_{\alpha}}$ for any $C_{i}\in J_{\alpha}$.
\end{enumerate}
Combining these two observations, we get that each pair of periodic intervals $C_{2i},C_{2i+1}$ are contained in the same elemental subinterval, and mapped to two periodic intervals $F\left(C_{2i}\right)=C_{2i+2k_{\alpha}},F\left(C_{2i+1}\right)=C_{2i+2k_{\alpha}+1}$ that are part of the same elemental subinterval. It follows that all along their orbit they evolve together staying in the same elemental subintervals. Thus, $C_{2i},C_{2i+1}$ must be part of the same periodic interval, contradicting the assumption that the periodic intervals we started with are maximal. 

The proof in the case of a CEM is similar. We index the periodic intervals with cyclic indices $C_{i\mod{2q}}$, starting at the first periodic interval within some elemental arc. The map $L$ is a local reflection and property a. follows. It is not immediate that property b. holds, as beside arcs rearrangement, there is a global rotation of $\theta_{1}-\theta_{0}$. 
Using the fact that for any elemental arc $J_{\alpha}$, $F\left(J_{\alpha}\right)=R\left(J_{\alpha}\right)$, where $R$ is the reflection about the $\frac{\theta_{0}+\theta_{1}}{2}$ angle, it follows that the map translates the periodic interval index in even steps, and property b. holds.
\end{proof}

\begin{example}
There exists a reversible IEM with a pair of non-symmetric periodic intervals. This is possible when the swap component $S$ swaps two pairs of equal-length intervals, where the symmetric component $F$ maps one pair to the other. An illustration of this configuration is given in Figure \ref{fig:non-symm-pair-TRS}. 
\begin{figure}[h]
\centering{}\resizebox{0.9\columnwidth}{!}{%

\tikzset{every picture/.style={line width=0.75pt}} %set default line width to 0.75pt        

\begin{tikzpicture}[x=0.75pt,y=0.75pt,yscale=-1,xscale=1]
%uncomment if require: \path (0,300); %set diagram left start at 0, and has height of 300

%Straight Lines [id:da4443117793928324] 
\draw    (120,260) -- (130,260) ;
%Straight Lines [id:da12674321741227668] 
\draw    (120,110) -- (130,110) ;
%Straight Lines [id:da8757109124860292] 
\draw    (330,160) -- (340,160) ;
%Straight Lines [id:da5406246506177484] 
\draw    (330,210) -- (340,210) ;
%Straight Lines [id:da1668782465562525] 
\draw    (120,60) -- (130,60) ;
%Straight Lines [id:da9779531602751446] 
\draw    (130,60) -- (210,60) ;
\draw [shift={(210,60)}, rotate = 180] [color={rgb, 255:red, 0; green, 0; blue, 0 }  ][line width=0.75]    (0,5.59) -- (0,-5.59)   ;
\draw [shift={(130,60)}, rotate = 180] [color={rgb, 255:red, 0; green, 0; blue, 0 }  ][line width=0.75]    (0,5.59) -- (0,-5.59)   ;

%Straight Lines [id:da3938833716233947] 
\draw    (210,60.5) -- (290,60.5) ;
\draw [shift={(290,60.5)}, rotate = 180] [color={rgb, 255:red, 0; green, 0; blue, 0 }  ][line width=0.75]    (0,5.59) -- (0,-5.59)   ;
\draw [shift={(210,60.5)}, rotate = 180] [color={rgb, 255:red, 0; green, 0; blue, 0 }  ][line width=0.75]    (0,5.59) -- (0,-5.59)   ;

%Curve Lines [id:da1862338013595688] 
\draw    (206,70) .. controls (211.95,74.79) and (211.52,87.66) .. (206.78,98.33) ;
\draw [shift={(206,100)}, rotate = 290.27] [color={rgb, 255:red, 0; green, 0; blue, 0 }  ][line width=0.75]    (10.93,-3.29) .. controls (6.95,-1.4) and (3.31,-0.3) .. (0,0) .. controls (3.31,0.3) and (6.95,1.4) .. (10.93,3.29)   ;
%Curve Lines [id:da22394056627295367] 
\draw    (210,120) .. controls (216.14,124.94) and (299.62,131.95) .. (328.31,139.54) ;
\draw [shift={(330,140)}, rotate = 196.07] [color={rgb, 255:red, 0; green, 0; blue, 0 }  ][line width=0.75]    (10.93,-3.29) .. controls (6.95,-1.4) and (3.31,-0.3) .. (0,0) .. controls (3.31,0.3) and (6.95,1.4) .. (10.93,3.29)   ;
%Curve Lines [id:da5184350081919751] 
\draw    (330,220) .. controls (321.63,227.07) and (245.79,225.68) .. (221.43,239.16) ;
\draw [shift={(220,240)}, rotate = 327.6] [color={rgb, 255:red, 0; green, 0; blue, 0 }  ][line width=0.75]    (10.93,-3.29) .. controls (6.95,-1.4) and (3.31,-0.3) .. (0,0) .. controls (3.31,0.3) and (6.95,1.4) .. (10.93,3.29)   ;
%Straight Lines [id:da6419223039815611] 
\draw [color={rgb, 255:red, 31; green, 119; blue, 180 }  ,draw opacity=1 ][line width=1.5]    (130,60.5) -- (210,60.5) ;
\draw [shift={(130,60.5)}, rotate = 180] [color={rgb, 255:red, 31; green, 119; blue, 180 }  ,draw opacity=1 ][line width=1.5]    (0,6.71) -- (0,-6.71)   ;
%Straight Lines [id:da48375593520514526] 
\draw [color={rgb, 255:red, 214; green, 39; blue, 40 }  ,draw opacity=1 ][line width=1.5]    (210,60.5) -- (290,60.5) ;
\draw [shift={(210,60.5)}, rotate = 180] [color={rgb, 255:red, 214; green, 39; blue, 40 }  ,draw opacity=1 ][line width=1.5]    (0,6.71) -- (0,-6.71)   ;
%Straight Lines [id:da3111356900706901] 
\draw [color={rgb, 255:red, 31; green, 119; blue, 180 }  ,draw opacity=1 ][line width=1.5]    (130,260) -- (210,260) ;
\draw [shift={(130,260)}, rotate = 180] [color={rgb, 255:red, 31; green, 119; blue, 180 }  ,draw opacity=1 ][line width=1.5]    (0,6.71) -- (0,-6.71)   ;
%Straight Lines [id:da2842380759583689] 
\draw [color={rgb, 255:red, 214; green, 39; blue, 40 }  ,draw opacity=1 ][line width=1.5]    (210,260) -- (290,260) ;
\draw [shift={(210,260)}, rotate = 180] [color={rgb, 255:red, 214; green, 39; blue, 40 }  ,draw opacity=1 ][line width=1.5]    (0,6.71) -- (0,-6.71)   ;
%Straight Lines [id:da5138518731643401] 
\draw [color={rgb, 255:red, 214; green, 39; blue, 40 }  ,draw opacity=1 ][line width=1.5]    (130,110) -- (131.04,110) -- (203.93,110) -- (210,110) ;
\draw [shift={(130,110)}, rotate = 180] [color={rgb, 255:red, 214; green, 39; blue, 40 }  ,draw opacity=1 ][line width=1.5]    (0,6.71) -- (0,-6.71)   ;
%Straight Lines [id:da21144459716007602] 
\draw [color={rgb, 255:red, 31; green, 119; blue, 180 }  ,draw opacity=1 ][line width=1.5]    (340,160) -- (419.73,160) ;
\draw [shift={(340,160)}, rotate = 180] [color={rgb, 255:red, 31; green, 119; blue, 180 }  ,draw opacity=1 ][line width=1.5]    (0,6.71) -- (0,-6.71)   ;
%Straight Lines [id:da06290252691603448] 
\draw [color={rgb, 255:red, 214; green, 39; blue, 40 }  ,draw opacity=1 ][line width=1.5]    (419.73,160) -- (500,160) ;
\draw [shift={(419.73,160)}, rotate = 180] [color={rgb, 255:red, 214; green, 39; blue, 40 }  ,draw opacity=1 ][line width=1.5]    (0,6.71) -- (0,-6.71)   ;
%Straight Lines [id:da7401013585980188] 
\draw [color={rgb, 255:red, 214; green, 39; blue, 40 }  ,draw opacity=1 ][line width=1.5]    (340,210) -- (420,210) ;
\draw [shift={(340,210)}, rotate = 180] [color={rgb, 255:red, 214; green, 39; blue, 40 }  ,draw opacity=1 ][line width=1.5]    (0,6.71) -- (0,-6.71)   ;
%Straight Lines [id:da3527890921641573] 
\draw    (40,60) -- (80,60) ;
\draw [shift={(40,60)}, rotate = 180] [color={rgb, 255:red, 0; green, 0; blue, 0 }  ][line width=0.75]    (0,5.59) -- (0,-5.59)   ;
%Straight Lines [id:da8573772488605031] 
\draw  [dash pattern={on 0.84pt off 2.51pt}]  (80,60) -- (90,60) ;
%Straight Lines [id:da9061588689241006] 
\draw    (40,260) -- (80,260) ;
\draw [shift={(40,260)}, rotate = 180] [color={rgb, 255:red, 0; green, 0; blue, 0 }  ][line width=0.75]    (0,5.59) -- (0,-5.59)   ;
%Straight Lines [id:da29973937383651383] 
\draw  [dash pattern={on 0.84pt off 2.51pt}]  (80,260) -- (90,260) ;
%Straight Lines [id:da4409974686935546] 
\draw    (550,60) -- (570,60) ;
\draw [shift={(570,60)}, rotate = 180] [color={rgb, 255:red, 0; green, 0; blue, 0 }  ][line width=0.75]    (0,5.59) -- (0,-5.59)   ;
%Straight Lines [id:da6494808161011648] 
\draw  [dash pattern={on 0.84pt off 2.51pt}]  (300,60) -- (310,60) ;
%Straight Lines [id:da8830852753906211] 
\draw  [dash pattern={on 0.84pt off 2.51pt}]  (540,60) -- (550,60) ;
%Straight Lines [id:da9483587121123516] 
\draw    (290,60) -- (300,60) ;
%Straight Lines [id:da47107703424822867] 
\draw  [dash pattern={on 0.84pt off 2.51pt}]  (110,260) -- (120,260) ;
%Straight Lines [id:da6842481342713352] 
\draw  [dash pattern={on 0.84pt off 2.51pt}]  (300,260) -- (310,260) ;
%Straight Lines [id:da23779103268603008] 
\draw    (550,260) -- (570,260) ;
\draw [shift={(570,260)}, rotate = 180] [color={rgb, 255:red, 0; green, 0; blue, 0 }  ][line width=0.75]    (0,5.59) -- (0,-5.59)   ;
%Straight Lines [id:da9227780662683113] 
\draw  [dash pattern={on 0.84pt off 2.51pt}]  (540,260) -- (550,260) ;
%Straight Lines [id:da9476304936944445] 
\draw  [dash pattern={on 0.84pt off 2.51pt}]  (110,60) -- (120,60) ;
%Straight Lines [id:da1815005094851605] 
\draw [color={rgb, 255:red, 31; green, 119; blue, 180 }  ,draw opacity=1 ][line width=1.5]    (210,110) -- (290,110) ;
\draw [shift={(210,110)}, rotate = 180] [color={rgb, 255:red, 31; green, 119; blue, 180 }  ,draw opacity=1 ][line width=1.5]    (0,6.71) -- (0,-6.71)   ;
%Straight Lines [id:da0011367424306091278] 
\draw [color={rgb, 255:red, 31; green, 119; blue, 180 }  ,draw opacity=1 ][line width=1.5]    (420,210) -- (500,210) ;
\draw [shift={(420,210)}, rotate = 180] [color={rgb, 255:red, 31; green, 119; blue, 180 }  ,draw opacity=1 ][line width=1.5]    (0,6.71) -- (0,-6.71)   ;
%Straight Lines [id:da7306648398808531] 
\draw    (40,110) -- (80,110) ;
\draw [shift={(40,110)}, rotate = 180] [color={rgb, 255:red, 0; green, 0; blue, 0 }  ][line width=0.75]    (0,5.59) -- (0,-5.59)   ;
%Straight Lines [id:da5178573015913163] 
\draw  [dash pattern={on 0.84pt off 2.51pt}]  (80,110) -- (90,110) ;
%Straight Lines [id:da6467380048963118] 
\draw    (550,110) -- (570,110) ;
\draw [shift={(570,110)}, rotate = 180] [color={rgb, 255:red, 0; green, 0; blue, 0 }  ][line width=0.75]    (0,5.59) -- (0,-5.59)   ;
%Straight Lines [id:da41347483212979874] 
\draw  [dash pattern={on 0.84pt off 2.51pt}]  (540,110) -- (550,110) ;
%Straight Lines [id:da3838502395344765] 
\draw    (40,160) -- (80,160) ;
\draw [shift={(40,160)}, rotate = 180] [color={rgb, 255:red, 0; green, 0; blue, 0 }  ][line width=0.75]    (0,5.59) -- (0,-5.59)   ;
%Straight Lines [id:da39256328234365845] 
\draw  [dash pattern={on 0.84pt off 2.51pt}]  (80,160) -- (90,160) ;
%Straight Lines [id:da16920473306168848] 
\draw    (550,160) -- (570,160) ;
\draw [shift={(570,160)}, rotate = 180] [color={rgb, 255:red, 0; green, 0; blue, 0 }  ][line width=0.75]    (0,5.59) -- (0,-5.59)   ;
%Straight Lines [id:da25296641137590203] 
\draw  [dash pattern={on 0.84pt off 2.51pt}]  (540,160) -- (550,160) ;
%Straight Lines [id:da6894000287482346] 
\draw    (40,210) -- (80,210) ;
\draw [shift={(40,210)}, rotate = 180] [color={rgb, 255:red, 0; green, 0; blue, 0 }  ][line width=0.75]    (0,5.59) -- (0,-5.59)   ;
%Straight Lines [id:da9076058178111702] 
\draw  [dash pattern={on 0.84pt off 2.51pt}]  (80,210) -- (90,210) ;
%Straight Lines [id:da4021414810973577] 
\draw    (550,210) -- (570,210) ;
\draw [shift={(570,210)}, rotate = 180] [color={rgb, 255:red, 0; green, 0; blue, 0 }  ][line width=0.75]    (0,5.59) -- (0,-5.59)   ;
%Straight Lines [id:da523181920531698] 
\draw  [dash pattern={on 0.84pt off 2.51pt}]  (540,210) -- (550,210) ;
%Curve Lines [id:da41376976269336974] 
\draw    (420,170) .. controls (425.95,174.79) and (425.52,187.66) .. (420.78,198.33) ;
\draw [shift={(420,200)}, rotate = 290.27] [color={rgb, 255:red, 0; green, 0; blue, 0 }  ][line width=0.75]    (10.93,-3.29) .. controls (6.95,-1.4) and (3.31,-0.3) .. (0,0) .. controls (3.31,0.3) and (6.95,1.4) .. (10.93,3.29)   ;
%Straight Lines [id:da487876616336574] 
\draw  [dash pattern={on 0.84pt off 2.51pt}]  (510,160) -- (520,160) ;
%Straight Lines [id:da612212470929275] 
\draw    (500,160) -- (510,160) ;
\draw [shift={(500,160)}, rotate = 180] [color={rgb, 255:red, 0; green, 0; blue, 0 }  ][line width=0.75]    (0,5.59) -- (0,-5.59)   ;
%Straight Lines [id:da40392375563273497] 
\draw  [dash pattern={on 0.84pt off 2.51pt}]  (300,110) -- (310,110) ;
%Straight Lines [id:da040107169615164584] 
\draw    (290,110) -- (300,110) ;
\draw [shift={(290,110)}, rotate = 180] [color={rgb, 255:red, 0; green, 0; blue, 0 }  ][line width=0.75]    (0,5.59) -- (0,-5.59)   ;
%Straight Lines [id:da24587578579744007] 
\draw  [dash pattern={on 0.84pt off 2.51pt}]  (510,210) -- (520,210) ;
%Straight Lines [id:da7099793805647968] 
\draw    (500,210) -- (510,210) ;
\draw [shift={(500,210)}, rotate = 180] [color={rgb, 255:red, 0; green, 0; blue, 0 }  ][line width=0.75]    (0,5.59) -- (0,-5.59)   ;
%Straight Lines [id:da4981120632224325] 
\draw  [dash pattern={on 0.84pt off 2.51pt}]  (320,160) -- (330,160) ;
%Straight Lines [id:da47534643400141596] 
\draw  [dash pattern={on 0.84pt off 2.51pt}]  (110,110) -- (120,110) ;
%Straight Lines [id:da8164115679498024] 
\draw  [dash pattern={on 0.84pt off 2.51pt}]  (320,210) -- (330,210) ;
%Straight Lines [id:da23671125822884242] 
\draw    (290,260) -- (300,260) ;
\draw [shift={(290,260)}, rotate = 180] [color={rgb, 255:red, 0; green, 0; blue, 0 }  ][line width=0.75]    (0,5.59) -- (0,-5.59)   ;

% Text Node
\draw (168.5,50) node    {$A$};
% Text Node
\draw (248.5,50.5) node    {$B$};
% Text Node
\draw (383,199.5) node  [font=\small]  {$W\!\circ\! F\!\circ\! W( B)$};
% Text Node
\draw (463.71,200) node  [font=\small]  {$W\!\circ\! F\!\circ\! W( A)$};
% Text Node
\draw (169.83,249.25) node  [font=\normalsize]  {$G^{2}( A)$};
% Text Node
\draw (250.33,249.25) node  [font=\normalsize]  {$G^{2}( B)$};
% Text Node
\draw (218.5,80.5) node    {$W$};
% Text Node
\draw (263,140) node    {$F$};
% Text Node
\draw (274,219) node    {$F$};
% Text Node
\draw (457.58,150) node    {$F\!\circ\! W( B)$};
% Text Node
\draw (379.1,149.5) node    {$F\!\circ\! W( A)$};
% Text Node
\draw (168.5,99.5) node    {$W( B)$};
% Text Node
\draw (248.5,99.5) node    {$W( A)$};
% Text Node
\draw (432.5,180.5) node    {$W$};

\end{tikzpicture}
}\caption{\label{fig:non-symm-pair-TRS}Non-symmetric pair of $2$-periodic
intervals of reversible IEM $G=F\circ S$, where $F$ is symmetric
and $S$ is the swap map from Theorem \ref{thm:IEM-reversible-iff-symm-swap}. }
\end{figure}
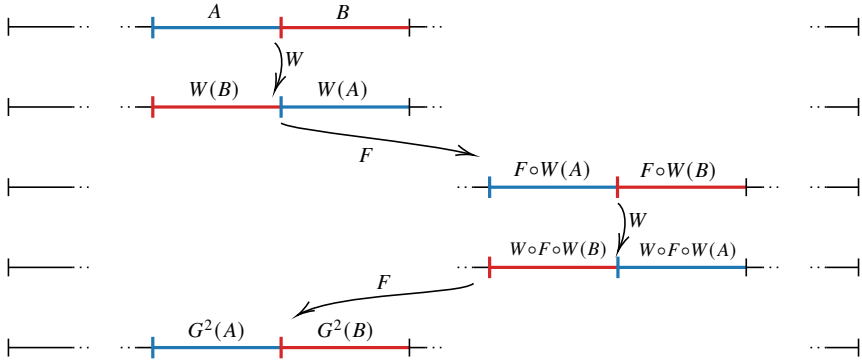
\end{example}

Summarizing, we established that the time reversal symmetry of the IEM and the CEM implies that all their periodic intervals must be either symmetric or come in symmetric pairs, and that only the center of the symmetric intervals corresponds to symmetric periodic orbits. Importantly, we established that \textbf{symmetric} IEM and CEM can have \textbf{only} symmetric periodic intervals.

\section{Perturbation of a FIEM} \label{sec:pertform}

Given a family of IEMs, parameterized by $y$, namely $F_{y}:x\mapsto x+\omega_{\alpha}\left(y\right)\quad x\in I_{\alpha}$,
we consider the setting in which, at each iteration of the map, the variable $y$ undergoes an $\eps$-small perturbation, given by a smooth function on the   $(x,y)$ cylinder $f$. This leads to the piecewise continuous $2$-dimensional map:
\[
T\left(\begin{matrix}x\\
y
\end{matrix}\right)=\left(\begin{matrix}F_{y}\left(x\right)\\
y+\eps f\left(x,y\right)
\end{matrix}\right)=\left(\begin{matrix}x+\omega\left(y\right)\\
y+\eps f\left(x,y\right)
\end{matrix}\right).
\]
We use the shorthand $\omega\left(y\right)$ to denote the appropriate coordinate of the translation vector according to $x$, which should cause no ambiguity in context. Locally, the translation vector is only a function of $y$. To simplify the analysis, we will assume $y$ is away from the boundary of the variable space, so $y+\eps f\left(x,y\right)$ will always be within this space. Notice that we do not allow perturbations in the horizontal direction.

A natural requirement is that $T$ would be area-preserving. This imposes (where $f_{x},f_{y}$ are the partial derivatives of $f$):
\begin{gather*}
\det\left(DT\right)=\det\left(\begin{matrix}1 & \omega'\left(y\right)\\
\eps f_{x} & 1+\eps f_{y}
\end{matrix}\right)=1+\eps f_{y}-\omega'\left(y\right)\eps f_{x}=1\\
\omega'\left(y\right)f_{x}-f_{y}=0
\end{gather*}
The characteristic curves of the PDE are integral curves of the vector
field $\left(\omega'\left(y\right),-1\right)$ \cite{strauss2007partial},
given by:
\begin{align*}
\frac{dx}{dy} & =-\omega'\left(y\right) & \Rightarrow &  & x & =-\omega\left(y\right)+c
\end{align*}
As $f$ is constant along the characteristic curves, it is a function on the circle
of the following form
\[
f\left(x,y\right)=\tilde f\left(x+\omega\left(y\right)\right)=\tilde f\left(F_{y}\left(x\right)\right).
\]
This allows us to write any perturbed area-preserving FIEM $T$ in the form:
\begin{equation}
T:\begin{cases}
x'=x+\omega\left(y\right)\\
y'=y+\eps \tilde f\left(x'\right)
\end{cases}\label{eq:fiem-T}
\end{equation}
where $\tilde f$ is a smooth function on the circle.
\begin{cor}\label{reversfiem}
When each of the IEM is reversible, $R\circ F_{y}\circ R=F_{y}^{-1}$, and
$f$ is an antisymmetric function on the circle, $f\left(R\left(x\right)\right)=-f\left(x\right)$, the
map $T$ is reversible with the symmetry 
\begin{equation}
S\left(\begin{matrix}x\\
y
\end{matrix}\right)=\left(\begin{matrix}R\left(x\right)\\
y-\eps f\left(x\right)
\end{matrix}\right)\label{eq:S-symm-def}
\end{equation}
i.e. $S\circ T\circ S=T^{-1}$.
\end{cor}
The standard map (Example \ref{exa:std-map}) is now a special case
of this construction, it is the perturbed symmetric linear $2$-FIEM
given by $\lambda_{y}=\left(y,1-y\right)$ with $f\left(x\right)\coloneqq\sin\left(2\pi x\right)$, which is antisymmetric with respect to $x=1/2$.

\begin{rem}
The other symmetry of the standard map, $\tilde S$, is a symmetry
for the perturbed FIEM only in the specific case of $F_{-y}=F_{y}^{-1}$.
Linear FIEM that arise from varying by $y$ the direction of flow
inside a rectilinear translation surface fulfill this condition when
$y=0$ stands for a flow perpendicular to the section.

\end{rem}

\begin{rem}
For a family of CEMs, the symmetry $R$ is the reflection of the circle
about $\frac{\theta_{0}+\theta_{1}}{2}$. In general, this point depends
on $y$ and the reversibility above would not hold, as $f$ needs
to be antisymmetric with respect to each and every $R$ along the family. In
natural mechanical settings (and thus in the standard map), in which the system is reversible, the
symmetry $R$ is coming from the same central symmetry of the family
of translation surfaces, so $\frac{\theta_{0}+\theta_{1}}{2}$ is
constant and $R$ is the same along the family, see the example in
Section \ref{sec:HIS}. 
\end{rem}

\section{Periodic Points of Perturbed Symmetric FIEMs}\label{sec:pertperiodicpoints}

We characterize the periodic points of perturbed FIEMs, families where the (constant) combinatorial data is that of a symmetric IEM (or CEM).

In what follows we will use the following notation for trajectories of the map $T$:
\begin{equation}
\begin{cases}
x_{n+1}= & x_{n}+\omega_{\alpha_{n}}\left(y_{n}\right), \ \ x_{n}\in I_{\alpha_{n}}\\
y_{n+1}= & y_{n}+\eps f\left(x_{n+1}\right)
\end{cases}\label{eq:recursive_fiem}
\end{equation}
when no ambiguity arises we write $\omega\left(y_{n}\right)$ instead of $\omega_{\alpha_{n}}\left(y_{n}\right)$. We assume both the translation vector $\omega_{\alpha_{n}}\left(y_{n}\right)$ and $f$ are smooth. The $n$'th iteration of the map is given by:
\begin{align*}
x_{n} & =x_{0}+\sum_{k=0}^{n-1}\omega_{\alpha_{k}}\left(y_{k}\right) & = &  & x_{0}+\sum_{k=0}^{n-1}\omega_{\alpha_{k}}\left(y_{0}+\eps\sum_{\ell=1}^{k}f\left(x_{\ell}\right)\right)\\
y_{n} & =y_{0}+\eps\sum_{k=1}^{n}f\left(x_{k}\right) & = &  & y_{0}+\eps\sum_{k=1}^{n}f\left(x_{0}+\sum_{\ell=0}^{k-1}\omega_{\alpha_{\ell}}\left(y_{\ell}\right)\right)
\end{align*}

Then $x_{0},y_{0}$ is a $q$-periodic point when ${\displaystyle \sum_{k=0}^{q-1}\omega\left(y_{k}\right)}=0$ (for a CEM $\mod 1$) and ${\displaystyle \eps\sum_{k=1}^{q}f\left(x_{k}\right)}=0$.

\subsection{Symmetric Periodic Orbits}

Recall that symmetric periodic orbits lie at intersections of symmetry lines. We thus begin by explicitly identifying the symmetry lines of a symmetric FIEM with respect to the symmetry $S$ \eqref{eq:S-symm-def}. 

The following lemma provides important insight regarding the symmetry lines of a perturbed symmetric FIEM:
\begin{lem}
Let $F$ be a symmetric $d$--IEM, the only solutions to  $F\left(x\right)=R\left(x\right)$, are the midpoints of the elemental subintervals of $F$.
\end{lem}

\begin{proof}
$R\circ F$ is the local reflection map on each elemental subinterval (Theorem \ref{cor:RF-local-reflection}), the only fixed points of $R\circ F$ are the midpoints.
\end{proof}

\begin{prop}
\label{prop:fiem-sym-lines-0-min-1}Let $T$ be a perturbed symmetric $d$-FIEM with the symmetry $S$ \eqref{eq:S-symm-def}, with $f(x)$ an antisymmetric function on the circle. Then, $\Gamma_{0}$ is the two vertical lines $\{0,\frac{1}{2}\}\times P $ and $\Gamma_{-1}$ is the union of the midpoints curves of the elemental subintervals of the unperturbed FIEM. 
\end{prop}

\begin{proof}
Since $f(x)$ is an antisymmetric function on the circle $f(0)=f(\frac{1}{2})=0$  so $\Gamma_{0}=\{0,\frac{1}{2}\}\times P$ are the only  solutions of $S\left(x,y\right)=\left(x,y\right)$ (recall that by Remark  \ref{rem:identifyboundaries} $R(0)=0$).

Let the map $T$ be 
\[
T\left(\begin{matrix}x\\
y
\end{matrix}\right)=\left(\begin{matrix}F_{y}\left(x\right)\\
y+\eps f\left(F_{y}\left(x\right)\right)
\end{matrix}\right).
\]

Let $m_{i}\left(y\right)$ denote the midpoint of the $i$'th elemental subinterval of $F_{y}$, namely $m_{i}\left(y\right)\coloneqq\frac{\lambda_{i}\left(y\right)}{2}+\sum_{j<i}\lambda_{j}\left(y\right)$. Need to show that 
\begin{equation} \label{eq:gamma_minus_1}
\Gamma_{-1}=\left\{ \left(\begin{matrix}m_{i}\left(y\right)\\
y
\end{matrix}\right):i=1\cds d,\quad y\in\left[0,1\right)\right\} 
\end{equation}

The fixed points fulfill 
\begin{gather*}
S_{-1}\left(\begin{matrix}x\\
y
\end{matrix}\right)=T^{-1}\circ S\left(\begin{matrix}x\\
y
\end{matrix}\right)=\left(\begin{matrix}x\\
y
\end{matrix}\right)\\
S\left(\begin{matrix}x\\
y
\end{matrix}\right)=T\left(\begin{matrix}x\\
y
\end{matrix}\right)\\
\left(\begin{matrix}R\left(x\right)\\
y-\eps f\left(x\right)
\end{matrix}\right)=\left(\begin{matrix}F_{y}\left(x\right)\\
y+\eps f\left(F_{y}\left(x\right)\right)
\end{matrix}\right)
\end{gather*}
 The equation for the second coordinates follows from the first equation,
as $f$ is antisymmetric $f\left(R\left(x\right)\right)=-f\left(x\right)$.
The solutions for the first coordinate are the midpoints according
to the previous lemma.
\end{proof}

The last proposition shows that adding an interval to the FIEM increases the number of symmetry lines. Heuristically, increasing the number of symmetry lines would increase the amount of (symmetric) periodic points.

For the unperturbed FIEM ($\eps=0$), all the periodic points are contained in periodic intervals of an IEM, at the corresponding $y$ value (Proposition \ref{prop:periodic-pts-iem-are-periodic-intervals}), and all of the periodic intervals are symmetric (Theorem \ref{thm:symm-iem-no-non-symm}).
Each symmetric periodic interval has a single symmetric periodic point at its midpoint (see Remark \ref{rem:midpoint-of-symm-periodic-interval-is-symm}).  By Proposition \ref{prop:symperiodicpointsonsymlines}, for the two-dimensional unperturbed map, these midpoint periodic orbits must belong to the intersection of symmetry lines, namely to images under $T$ of $\Gamma_0,\Gamma_{-1}$ of \ref{prop:fiem-sym-lines-0-min-1}.  

The following theorem shows that for sufficiently small perturbations these periodic points usually persist.
\begin{thm}
\label{thm:symm-intersect-transv-persist}Let $X^*=\left(x_{0}^{*},y_{0}^{*}\right)\cds\left(x_{q-1}^{*},y_{0}^{*}\right)$ be a symmetric $q$-periodic orbit of the unperturbed FIEM, where the symmetry lines intersect transversely, then, in a small neighborhood of $X^*$, for any sufficiently small $\eps$, the perturbed map has a unique symmetric $q$-periodic orbit $ X^\eps$.
\end{thm}

\begin{proof}
As follows from the previous proposition and Proposition \ref{prop:symm-lines-are-images-of-primary}, even and odd symmetry lines are given by 
\begin{align}
\Gamma_{2n} & =T^{n}\Gamma_{0}=\left\{ T^{n}\left(\begin{matrix}x\\
y
\end{matrix}\right):x\in\left\{ 0,\frac{1}{2}\right\} ,y\in P\right\} \label{eq:gamma_even}\\
\Gamma_{2n+1} & =T^{n+1}\Gamma_{-1}=\left\{ T^{n+1}\left(\begin{matrix}m_{i}\left(y\right)\\
y
\end{matrix}\right):i=1\cds d,\quad y\in P\right\} \label{eq:gamma_odd}
\end{align}
The symmetry lines depend smoothly on $\eps$. Given they intersect transversely at $\eps=0$, the theorem follows from the implicit function theorem.
\end{proof}

Next, we show that usually, symmetry lines do intersect transversely.
We differentiate the above expressions with respect to $y$, and using the chain rule we find the tangent vectors at height $y$ of the symmetry lines.  

\begin{prop} \label{prop:gamma_tangent}
The tangent to the odd and even unperturbed symmetry lines at the point $(x,y)$ along the line is given by 
\begin{align}
D_{(x,y)} \Gamma_{2n} & = \left(\begin{matrix}\sum_{k=1}^{n}\omega'_{\alpha_{k}}\left(y\right)\\
1
\end{matrix}\right)\\
D_{(x,y)} \Gamma_{2n+1} & = \left(\begin{matrix} -\frac12\omega_{i}'(y) +\sum_{k=0}^{n-1}\omega'_{\alpha_{k}}\left(y\right)  \\
1
\end{matrix}\right)
\end{align}
where $\alpha_{k}$ denotes the interval at which $T^{k}(x,y)$ resides
\end{prop}

\begin{proof}
By Proposition \ref{prop:fiem-sym-lines-0-min-1}, the tangent of the 2 components of $\Gamma_0$ is the vertical line $\left(\begin{smallmatrix} 0 \\ 1 \end{smallmatrix}\right)$, then by \eqref{eq:gamma_even} the tangent to the even symmetry lines is given by
\begin{equation*} 
D_{(x,y)}\left(T^{n}\right)\left(\begin{matrix} 0\\
1
\end{matrix}\right)=\left(\begin{matrix}1 & \omega'_{\alpha_{n}}\left(y\right)\\
0 & 1
\end{matrix}\right)\cdots\left(\begin{matrix}1 & \omega'_{\alpha_{1}}\left(y\right)\\
0 & 1
\end{matrix}\right)\left(\begin{matrix}0\\
1
\end{matrix}\right)=\left(\begin{matrix}\sum_{k=1}^{n}\omega'_{\alpha_{k}}\left(y\right)\\
1
\end{matrix}\right).
\end{equation*}
By Proposition \ref{prop:fiem-sym-lines-0-min-1}, the tangent to the  $i^{\text{th}}$ component of the $\Gamma_{-1}$ symmetry line is given by $\left(\begin{smallmatrix}m_{i}'\left(y\right)\\ 1 \end{smallmatrix}\right)$ then by \eqref{eq:gamma_odd} the tangent to the odd symmetry lines is given by
\begin{equation*} 
D_{(x,y)}\left(T^{n}\right)\left(\begin{matrix}m_{i}'\left(y\right)\\
1
\end{matrix}\right)=\left(\begin{matrix}m_{i}'\left(y\right)+\sum_{k=0}^{n-1}\omega'_{\alpha_{k}}\left(y\right)  \\
1
\end{matrix}\right)=\left(\begin{matrix} -\frac12\omega_{i}'(y) +\sum_{k=0}^{n-1}\omega'_{\alpha_{k}}\left(y\right)  \\
1
\end{matrix}\right).
\end{equation*}
where $\alpha_0=i$. The last equality follows from the fact that for reversing permutations
\begin{equation}
    \omega_{i}(y) = \sum_{j>i}\lambda_{j}\left(y\right)-\sum_{j<i}\lambda_{j}\left(y\right)=1- \lambda_i(y) -2\sum_{j<i}\lambda_{j}\left(y\right) = 1-2 m_i(y), \label{eq:omegaireversing}
\end{equation} namely  $m_i(y) = \frac12 (1-\omega_{i}(y) ) $. 
\end{proof}

The sums are integer linear combinations of $\frac12\omega'_{1}\cds \frac12\omega'_{d}$,  so, generically, their intersection is expected to be transverse for $d>2$. Indeed, the transversality conditions amount to a natural requirement on $\lambda'(y)$:

\begin{cor}\label{cor:transversesym}
For any $y\in P$ for which the unperturbed symmetric FIEM has a $q$-periodic interval and for which the vector $\lambda'(y)=\{\lambda_i'(y)\}_{i=1}^{d}$  has a  $(d-1)$ subset of values that are incommensurate over the integers, the unperturbed symmetric $q$-periodic orbit persists under sufficiently small perturbations.
\end{cor}
\begin{proof}
Recall that since $|I|=1$ (see  \eqref{eq:sumlambdais1})  $\sum_{i=1}^d \lambda_i'(y) =0$, so, for any FIEM there is always one linear  integer dependency of the $\omega_i'$.  For reversing permutations $\sum_{i=1}^d (-1)^i\,\omega_i'(y) =(\hat k_{alt}, \omega')=0$, where $\hat k_{alt} =(-1,1,\dots)$. By the assumption on $\lambda'(y)$, for any integer vector  $\hat k$  which is not in the span of $\hat k_{alt}$, the product sum with $\omega'(y)$ does not vanish $(\hat k, \omega'(y)) \ne 0$. Since symmetric periodic orbits reside in intersections of symmetry lines, to show that the tangent vectors intersect transversely we need to show that the difference of the sums of $\omega_i'(y)$ that appear in  formulae for the tangent vectors correspond to a vector $\hat k$ which is not in the span of the alternating vector $\hat k_{alt}$.

Recall that Proposition \ref{prop:symperiodicpointsonsymlines} provides the complete description of the symmetry lines on which symmetric periodic orbits lie. Next, we check that the tangent vectors to the symmetry lines of the unperturbed FIEM for each of the three cases listed in the proposition intersect transversely. Since the symmetric periodic orbits of the unperturbed FIEM lie at the center of the periodic intervals, and each of these intervals have a positive length (otherwise they correspond to saddle connections), they are bounded away from the singularities. Thus,  the tangent vectors to the symmetry lines are well defined along their orbits.

Consider a symmetric $q$-periodic orbit, $(x_0^*,y_0^*) = (x_q^*,y_q^*)$ with $(x_i^*,y_i^*) \in I_{\alpha_i^*}$.  
If the symmetric $q$ periodic orbit has an odd period $q=2n+1$ it must have a point, denoted by $(x_0^*,y_0^*)$ in the intersection of $\Gamma_0 \cap \Gamma_{q=2n+1}$. The tangent to $\Gamma_0$ at this point is $(0,1)$. Since $(x_0^*,y_0^*)\in \Gamma_{2n+1}$,  we conclude, by Proposition \ref{prop:image-of-symm-line} that  $T^{-n-1}(x_0^*,y_0^*)\in \Gamma_{-1}$,  hence there exists an $i$ such that $T^{-n-1}(x_0^*,y_0^*) = (m_i(y_0^*),y_0^*)$ and by periodicity $(m_i(y_0^*),y_0^*)=T^{n}(x_0^*,y_0^*)=(x^*_n,y^*_n) $, so the branch of $\Gamma_{-1}$ to which the orbit belongs is  $i=\alpha_n^*$. Thus, by Proposition \ref{prop:gamma_tangent}, the tangent vector at  $(x_0^*,y_0^*)\in \Gamma_{2n+1}$ is
 $(-\frac12\omega_{\alpha_n^*}'(y) +\sum_{k=0}^{n}\omega'_{\alpha^{*}_{n+k}}\left(y\right),1)$. Thus, the tangent vectors intersect transversly if $-\omega_{\alpha_n^*}'(y) + 2 \sum_{k=0}^{n}\omega'_{\alpha^{*}_{n+k}}\left(y\right)= -\omega_{\alpha_n^*}'(y) +2\sum_{i=1}^{d} k_i\omega'_{i}\left(y\right)\ne 0$ where the non-negative integer vector $\tilde k=(k_1,\dots,k_n)$ counts the number of visits of half the orbit in the elemental interval $i$, so in particular $k_{\alpha_n^*} \ge 1$. The resulting integer vector, $\hat k =2 \tilde k - \hat e_{\alpha_n^*} $ (where $\hat e_{j} $ 
is the unit vector with the entry $1$ at the entry $j$) is non-negative and nonzero so it is clearly not in the span of the alternating vector $\hat k_{alt}$, as claimed.

Now consider an even $q=2n$ for the case in which the orbit lies on two even symmetry lines. Then, there exists a point,  denoted by $(x_0^*,y_0^*)$, that belongs to  $\Gamma_0 \cap \Gamma_{2n}=T^{n}\Gamma_0$. The tangent to $\Gamma_0$ at this point is $(0,1)$. Since $(x_0^*,y_0^*)\in \Gamma_{2n}$,  we conclude that there exists an $s\in \{0,\frac12\}$ such that $T^{-n}(x_0^*,y_0^*) = (s,y_0^*)=T^{n}(x_0^*,y_0^*)=(x^*_n,y^*_n) \in \Gamma_{0}$, so, by Proposition \ref{prop:gamma_tangent},  the tangent to $\Gamma_{2n}$ at  $(x_0^*,y_0^*)$,  is $(\sum_{k=1}^{n}\omega'_{\alpha^{*}_{n+k}}\left(y\right),1)=(\sum_{i=1}^{d} k_i\omega'_{i}\left(y\right),1)$ where  
where the non-negative integer vector $\hat k=(k_1,\dots,k_n)$ counts the number of visits of half the orbit in the elemental interval $i$. In particular, for $n \ge 1$, $|\hat K| \ne 0$ and since it has only positive entries it is not in the span of the alternating vector $\hat k_{alt}$.

Finally,  consider an even $q=2n$ periodic orbit for the case in which the orbit lies on two odd symmetry lines. Choose a point  $(x_0^*,y_0^*)\in \Gamma_{-1} \cap \Gamma_{2n-1}=T^{n}\Gamma_{-1}$, so there exists a branch $i$ such that  $(x_0^*,y_0^*) = (m_i(y_0^*),y_0^*) \in \Gamma_{-1}$ and a branch $j$ of $\Gamma_{-1}$ such that  $T^{-n}(x_0^*,y_0^*) = (m_j(y_0^*),y_0^*)=T^{n}(x_0^*,y_0^*)=(x^*_n,y^*_n) \in \Gamma_{-1}$.  Hence, $i=\alpha_0^*=\alpha_{2n}^*$ and $j=\alpha^*_n$. Then, the tangent vector to $\Gamma_{-1}$ is $(m_i'(y_0^*),1)= ( -\frac12 \omega_{\alpha_0^*}'(y) ,1)  $ and the tangent vector to $\Gamma_{2n-1}$ at 
 $(x_0^*,y_0^*)$  is 
 $(-\frac12\omega_{\alpha_n^*}'(y) +\sum_{k=0}^{n-1}\omega'_{\alpha^{*}_{n+k}}\left(y\right),1)$. Thus, the tangent vectors intersect transversly if $ \omega_{\alpha_0^*}'(y)-\omega_{\alpha_n^*}'(y) + 2 \sum_{k=0}^{n-1}\omega'_{\alpha^{*}_{n+k}}\left(y\right)=\omega_{\alpha_0^*}'(y) -\omega_{\alpha_n^*}'(y) +2\sum_{i=1}^{d} k_i\omega'_{i}\left(y\right)\ne 0$ where, again, the non-negative integer vector $\tilde k=(k_1,\dots,k_n)$ counts the number of visits of half the orbit in the elemental intervals, and thus  $k_{\alpha_n^*} \ge 1$ and  $k_{\alpha_0^*} \ge 1$. We conclude that the resulting nonzero integer vector, $\hat k =2 \tilde k  - \hat e_{\alpha_n^*}+\hat e_{\alpha_0^*} $  is non-negative, so it is not spanned by $\hat k_{alt}$.
\end{proof}

\begin{cor}\label{cor:transversesym-linear}
For a symmetric linear family  $\lambda'(y)=\lambda_1-\lambda_0$, hence, provided this $d$ dimensional vector has a  $(d-1)$ subset of values that are incommensurate over the integers, all the symmetric periodic orbits persist.
\end{cor}
We conjecture that such linear families have a dense set of $y$ values at which symmetric periodic intervals emerge, with their centers corresponding to persisting symmetric periodic orbits.
\begin{example}
Consider the symmetric linear $4$-FIEM in Figure \ref{fig:pert-FIEM},
its symmetry lines are scattered curves, as opposed to the continuous
curves of the standard map (Figure \ref{fig:std-map}), elliptic and
hyperbolic periodic points appear at the intersection of the symmetry
lines. 
\begin{figure}[h]
\begin{centering}
\includegraphics[width=0.5\textwidth]{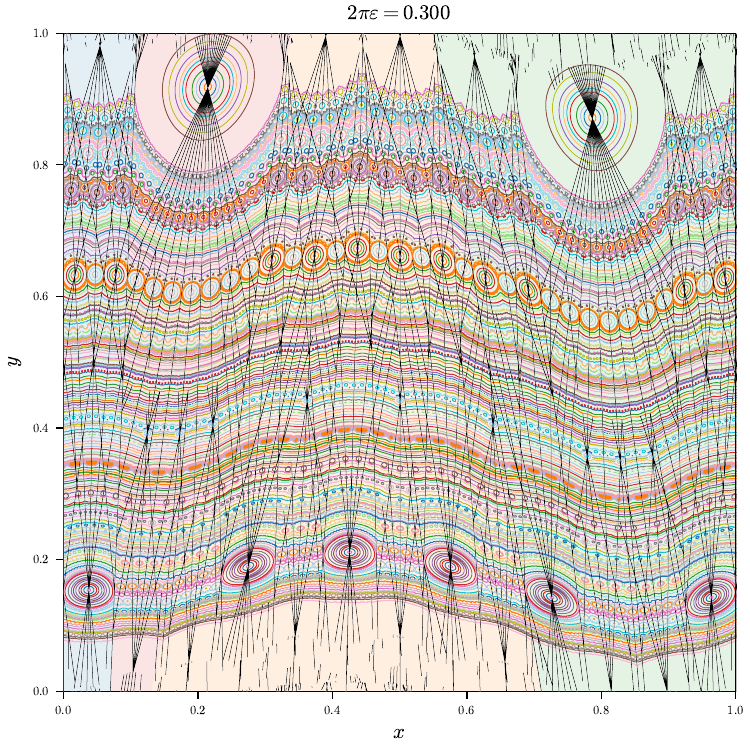}\includegraphics[width=0.5\textwidth]{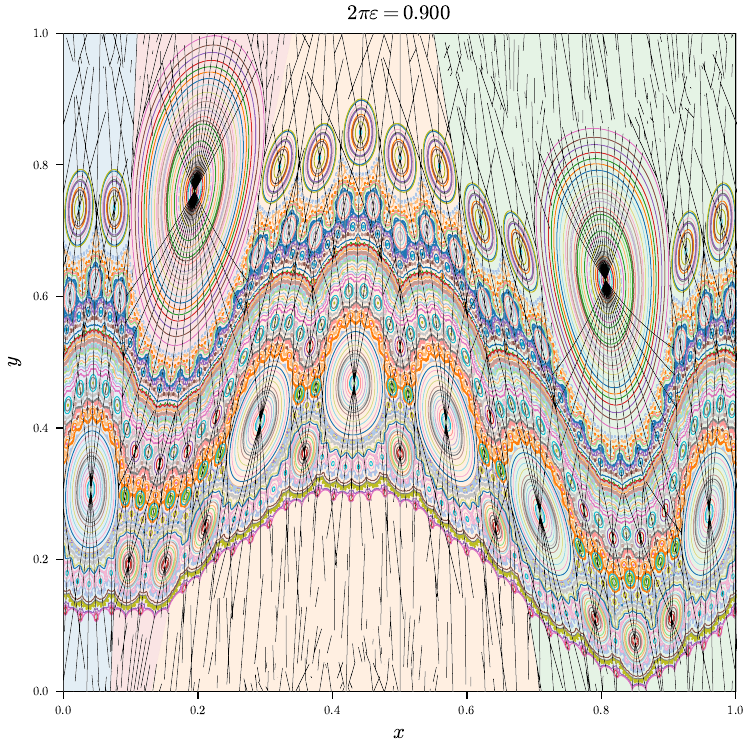}
\par\end{centering}
\centering{}\caption{\label{fig:pert-FIEM}The perturbed linear FIEM with $\lambda_{0}=\left(0.07,0.06,0.13,0.29\right),\lambda_{1}=\left(0.12,0.23,0.29,0.45\right)$.
The symmetry lines $\Gamma_{-7}\cds\Gamma_{7}$ are shown in black and trajectories are plotted in different colors. The shaded background colors indicate the elemental subregions of the unperturbed FIEM. }
\end{figure}
\end{example}

As $\eps$ increases, the symmetry lines undergo deformation and displacement and new intersection points may arise, as demonstrated in Figure \ref{fig:new-symm-pert-FIEM}.

\begin{figure}[h]
\begin{centering}
\includegraphics[width=0.5\textwidth]{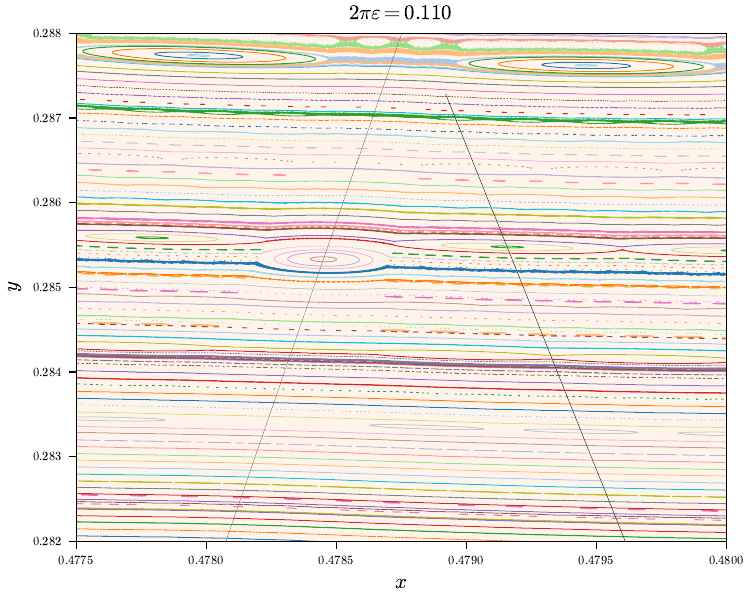}\includegraphics[width=0.5\textwidth]{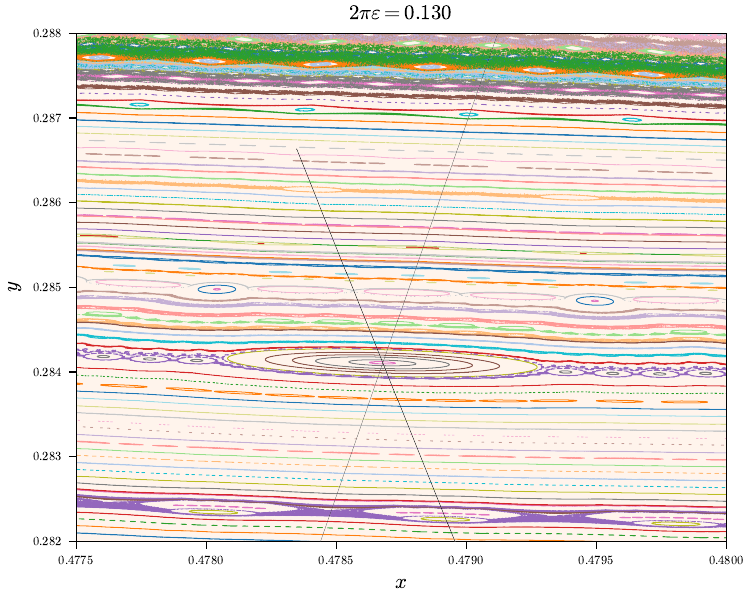}
\par\end{centering}
\centering{}\caption{\label{fig:new-symm-pert-FIEM}Bifurcation of periodic point by new intersection points of symmetry lines. The same perturbed linear FIEM from Figure \ref{fig:pert-FIEM} is shown, displaying only the symmetry lines $\Gamma_{14}$ and $\Gamma_{-9}$ in black. In the left panel the symmetry lines do not intersect. On the right panel, after increasing $\eps$, the symmetry lines intersect, with a resonance island around the elliptic periodic point of period $23$. The other resonance islands lie at intersections with other symmetry lines that do not appear in this figure.}
\end{figure}

In summary, we have described all the symmetric periodic orbits of a perturbed symmetric FIEM. Generically, by Theorem \ref{thm:symm-intersect-transv-persist} and Corollary  \ref{cor:transversesym} all the unperturbed symmetric periodic orbits persist under sufficiently small perturbation.
Additionally, there are symmetric periodic orbits that arise from new intersections of the symmetry lines, appearing at some critical nonzero $\eps$.

\subsection{Non-symmetric Periodic Orbits}

Non-symmetric periodic orbits are less discussed in the literature. As for the symmetric case, we first establish conditions under which non-symmetric periodic orbits may persist (exist for all sufficiently small $\eps$) and then discuss and demonstrate their possible appearance via bifurcations at finite $\eps$ values.

Recall that the unperturbed symmetric FIEM has periodic intervals, all of them symmetric. The midpoints of the periodic intervals are symmetric periodic points, which persist upon perturbation by Theorem \ref{thm:symm-intersect-transv-persist} and \ref{cor:transversesym}.
The remaining points of the periodic intervals are non-symmetric periodic points. First, we show in Theorem \ref{thm:periodic-pts-persist} that the persistence of general periodic orbits (not necessarily symmetric) under small perturbations, amounts, as usual, to proving, as a necessary condition, that the sums of the forcing function $f$ along the orbit vanish, and, as a sufficient condition, that  additionally the linearization of the return map at a point on the orbit is non-degenerate (the product of the sums of the derivatives of the forcing function and the twists does not vanish). 

 Then, in Theorem \ref{thm:balancedorbits} we show that for $f=\sin(2 \pi x)$ there are no persisting non-symmetric fixed points, and, if there exists a persisting non-symmetric periodic orbit then all persisting periodic orbits in this interval, including the symmetric one, must violate the transversality conditions of  Theorem \ref{thm:periodic-pts-persist}. On the other hand, in Theorem \ref{thm:balancedorbitslargeell}, we show that for  $f=\sin(2 \pi \ell x)$  and $\ell$ sufficiently large, persisting non-symmetric fixed points and non-symmetric periodic orbits exist.

For any $q$-periodic orbit $X $, let  \begin{equation}\label{eq:Mres}
    M (X):= \left( \sum_{k=0}^{q-1}f'\left(x_{k}^{*}\right) \right)\left(\sum_{k=0}^{q-1}\omega_{\alpha_{k}}'\left(y_{0}^{*}\right)\right).
\end{equation}
\begin{thm}
\label{thm:periodic-pts-persist}If the unperturbed map has a $q$-periodic orbit $X^*=\{(x_{0}^{*},y_{0}^{*})\cds(x_{q-1}^{*},y_{0}^{*})\}$ satisfying  $\sum_{k=0}^{q-1}f\left(x_{k}^{*}\right)=0$ and $M(X^*)\neq0$, then, for any sufficiently small $\eps$ the perturbed map has a unique $q$-periodic orbit $X^\eps = \{(x_{0}^{\eps},y_{0}^{\eps})\cds(x_{q-1}^{\eps},y_{q-1}^{\eps})\}$ in a small neighborhood of $X^*$. Moreover, the condition $\sum_{k=0}^{q-1}f\left(x_{k}^{*}\right)=0$ is necessary.
\end{thm}

\begin{proof}
Define the function
\[
G\left(x_{0},y_{0},\eps\right)=\left(\begin{matrix}\sum_{k=0}^{q-1}\omega_{\alpha_{k}}\left(y_{k}\right)\\
\sum_{k=0}^{q-1}f\left(x_{k}\right)
\end{matrix}\right)
\]
For $\left(x_{0},y_{0},\eps\right)$ in a small neighborhood of $\left(x_{0}^{*},y_{0}^{*},0\right)$ we would have that $x_{j},x_{j}^{*}\in\alpha_{j}$, then $\left(x_{0},y_{0},\eps\right)$ is a $q$-periodic orbit if $G\left(x_{0},y_{0},\eps\right)=\left(\begin{smallmatrix}0\\
0
\end{smallmatrix}\right)$.

We have that $G\left(x_{0}^{*},y_{0}^{*},0\right)=\left(\begin{smallmatrix}0\\
0
\end{smallmatrix}\right)$ and that 
\begin{align*}
DG\left(x_{0},y_{0},\eps\right) & =\left(\begin{matrix}\frac{\partial}{\partial x_{0}}\sum_{k=0}^{q-1}\omega_{\alpha_{k}}\left(y_{k}\right) & \frac{\partial}{\partial y_{0}}\sum_{k=0}^{q-1}\omega_{\alpha_{k}}\left(y_{k}\right)\\
\frac{\partial}{\partial x_{0}}\sum_{k=0}^{q-1}f\left(x_{k}\right) & \frac{\partial}{\partial y_{0}}\sum_{k=0}^{q-1}f\left(x_{k}\right)
\end{matrix}\right)\\
 & =\left(\begin{matrix}\sum_{k=0}^{q-1}\omega_{\alpha_{k}}'\left(y_{k}\right)\frac{\partial y_{k}}{\partial x_{0}} & \sum_{k=0}^{q-1}\omega_{\alpha_{k}}'\left(y_{k}\right)\frac{\partial y_{k}}{\partial y_{0}}\\
\sum_{k=0}^{q-1}f'\left(x_{k}\right)\frac{\partial x_{k}}{\partial x_{0}} & \sum_{k=0}^{q-1}f'\left(x_{k}\right)\frac{\partial x_{k}}{\partial y_{0}}
\end{matrix}\right)
\end{align*}
using the following
\begin{align*}
\frac{\partial y_{k}}{\partial y_{0}} & =\frac{\partial}{\partial y_{0}}\left[y_{0}+\eps\sum_{d=1}^{k}f\left(x_{d}\right)\right]=1+O\left(\eps\right)\\
\frac{\partial y_{k}}{\partial x_{0}} & =\frac{\partial}{\partial x_{0}}\left[y_{0}+\eps\sum_{d=1}^{k}f\left(x_{d}\right)\right]=0+O\left(\eps\right)\\
\frac{\partial x_{k}}{\partial x_{0}} & =\frac{\partial}{\partial x_{0}}\left[x_{0}+\sum_{d=0}^{k-1}\omega_{\alpha_{d}}\left(y_{d}\right)\right]=1+\sum_{d=1}^{k-1}\omega_{\alpha_{d}}'\left(y_{d}\right)\frac{\partial y_{d}}{\partial x_{0}}=1+\sum_{d=1}^{k-1}\omega_{\alpha_{d}}'\left(y_{d}\right)O\left(\eps\right)\\
\frac{\partial x_{k}}{\partial y_{0}} & =\frac{\partial}{\partial y_{0}}\left[x_{0}+\sum_{d=0}^{k-1}\omega_{\alpha_{d}}\left(y_{d}\right)\right]=0+\sum_{d=0}^{k-1}\omega_{\alpha_{d}}'\left(y_{d}\right)\frac{\partial y_{d}}{\partial y_{0}}=\sum_{d=0}^{k-1}\omega_{\alpha_{d}}'\left(y_{d}\right)\left(1+O\left(\eps\right)\right)
\end{align*}
and substituting $\left(x_{0}^{*},y_{0}^{*},0\right)$, after some algebra, gives
\begin{multline*}
\det\left(DG\left(x_{0}^{*},y_{0}^{*},0\right)\right)=\det\left(\begin{matrix}0 & \sum_{k=0}^{q-1}\omega_{\alpha_{k}}'\left(y_{k}^{*}\right)\\
\sum_{k=0}^{q-1}f'\left(x_{k}^{*}\right) & \sum_{k=0}^{q-1}\left(f'\left(x_{k}^{*}\right)\sum_{d=1}^{k-1}\omega_{\alpha_{d}}'\left(y_{d}^{*}\right)\right)
\end{matrix}\right)\\
=-\left(\sum_{k=0}^{q-1}\omega_{\alpha_{k}}'\left(y_{k}^{*}\right)\right)\left(\sum_{k=0}^{q-1}f'\left(x_{k}^{*}\right)\right)\neq0
\end{multline*}
By the implicit function theorem, for sufficiently small $\eps>0$ we would have a unique $\left(x_{0}^{\eps},y_{0}^{\eps}\right)$ such that $G\left(x_{0}^{\eps},y_{0}^{\eps},\eps\right)=0$, and $x_{0}^{0}=x_{0}^{*},y_{0}^{0}=y_{0}^{*}$ hence the perturbed system has a $q$-periodic orbit.

Moreover, if the second component of $G$ is bounded away from zero, by continuity, any approaching sequence will inherit this property, and then the sequence does not correspond to a periodic orbit. 
\end{proof}

\begin{rem}\label{rem:persistpertorbits}
    Notice that the above theorem applies also to perturbed orbits: if $X^{\eps_0}$ is a periodic orbit of the perturbed system at some $\eps=\eps_0$, and it satisfies the non-degeneracy condition $M(X^{\eps_0}) \ne 0$, then, by the same calculation as above, it persists for nearby $\eps$ values.
\end{rem}
Notice that a symmetric periodic orbit of the unperturbed system  satisfies $G(x,y)=0$ (since $f$ is antisymmetric). On the other hand, interestingly, the requirement of the transversality of the symmetry lines in Theorem \ref{thm:symm-intersect-transv-persist} is independent of the derivatives of $f$, suggesting that symmetric periodic orbits are more robust, and can persist even when the sufficient transversality condition is violated.  For the standard map this is certainly the case - since the $x_k$ values are evenly spaced, for all $q$-periodic orbits $\left(\sum_{k=0}^{q-1}f'\left(x_{k}^{*}\right)\right) = \left(\sum_{k=0}^{q-1}\cos \left(2 \pi (x_{0}^{*}+k/q)\right)\right) = 0$, so the conditions for Theorem \ref{thm:periodic-pts-persist} are not met.

\begin{defn}\label{def:balanced}
We call a set of points on the unit circle \emph{balanced} \textit{with respect to} $f(x)$ (and, for $f(x)= \sin (2 \pi x)$, simply \textit{balanced})
if their center of mass is at the origin, namely 
 $\left(\sum_{k=0}^{q-1}f\left(x_{k}^{*}\right)\right) =\left(\sum_{k=0}^{q-1}f'\left(  x_{k}^{*}\right)\right) =0$.
\end{defn}  
A balanced set of points (i.e. balanced with respect to $f(x)=\sin(2\pi x)$) consists of the vertices of a collection of regular polygons inscribed in the unit circle.  Balanced periodic orbits violate the transversality condition of Theorem \ref{thm:periodic-pts-persist}.  We next discuss the relations between balanced unperturbed periodic orbits and their persistence.

%[[ Can one utilize this property to construct non-trivial examples of persistence of periodic orbits ? ]]

In Section \ref{subsec:area-presv-symmetries}, we observed that non-symmetric orbits of the perturbed map must appear in symmetry-related pairs. Away from the discontinuity lines the map is a two-dimensional area-preserving map so the two periodic orbits have the same linear stability. The following proposition shows that each point on the first orbit is related by local reflection to a point on the other orbit, and, surprisingly, that both points belong  to the same horizontal segment in a given elemental subregion. 
\begin{prop}
\label{prop:local-symmetry-2d}The local symmetry $L\coloneqq S\circ T=T^{-1}\circ S$ takes points to their reflection through the center of the elemental subinterval containing them. $L$ sends symmetric periodic points to other periodic points along their orbit, and sends non-symmetric periodic points to points on their symmetric pair orbit.
\end{prop}

\begin{proof}
\[
L\left(\begin{matrix}x\\
y
\end{matrix}\right)=S\circ T\left(\begin{matrix}x\\
y
\end{matrix}\right)=S\left(\begin{matrix}F_{y}\left(x\right)\\
y+\eps f\left(F_{y}\left(x\right)\right)
\end{matrix}\right)=\left(\begin{matrix}R\circ F_{y}\left(x\right)\\
y+\eps f\left(F_{y}\left(x\right)\right)-\eps f\left(F_{y}\left(x\right)\right)
\end{matrix}\right)=\left(\begin{matrix}R\circ F_{y}\left(x\right)\\
y
\end{matrix}\right)
\]
and the result follows from Corollary \ref{cor:RF-local-reflection}. 

As explained in Section \ref{subsec:area-presv-symmetries},  $L$ is a local symmetry, hence it sends periodic points to periodic points and, by definition, these images belong to the same orbit for the case of symmetric orbits and to the symmetric-pair-orbit otherwise.
\end{proof}

 The following Theorem shows that for the special case $f(x)=\sin\left(2\pi x\right)$, the only fixed points that persist are the symmetric ones, and that if non-symmetric periodic orbits of period larger than one persist they must be balanced;  
\begin{thm}\label{thm:balancedorbits}
Let $f\left(x\right)=\sin\left(2\pi x\right)$. If an unperturbed non-symmetric periodic orbit of period $q > 1$ persists under perturbation, then its set of $x$-coordinates is balanced, and so is the set of $x$-coordinates of the symmetric periodic orbit that is at the center of its periodic interval. There are no persisting unperturbed non-symmetric fixed points.
\end{thm}

\begin{proof}
First, notice that a fixed point of  $\left(\begin{matrix}F_{y}\left(x\right)\\
y+\eps f\left(F_{y}\left(x\right)\right)
\end{matrix}\right)=\left(\begin{matrix}x\\
y
\end{matrix}\right)$ is, in particular, a fixed point of the IEM $F_{y}$.  Thus, by Theorem \ref{thm:symm-iem-no-non-symm} it is a part of a central symmetric elemental subinterval of $F_{y}$, so, provided $F_y$ is not the identity, it excludes the boundaries of the interval. Since it is also a root of the perturbation $f\left(x\right)=\sin\left(2\pi x\right)=0$, the only possible root is $x=\frac{1}{2}$, which is symmetric.

Suppose there is an unperturbed non-symmetric periodic orbit of period $q>1$ that persist upon perturbation, then there is a pair of non-symmetric periodic orbits
\begin{align*}
(x_{0}^{\eps},y_{0}^{\eps})\cds(x_{q-1}^{\eps},y_{q-1}^{\eps})\\
(u_{0}^{\eps},v_{0}^{\eps})\cds(u_{q-1}^{\eps},v_{q-1}^{\eps})
\end{align*}
where $\left(\begin{smallmatrix}u_{0}^{\eps}\\
v_{0}^{\eps}
\end{smallmatrix}\right)=S\circ T\left(\begin{smallmatrix}x_{0}^{\eps}\\
y_{0}^{\eps}
\end{smallmatrix}\right)$, and as $\eps\to0$, the periodic orbits converge to a pair of unperturbed non-symmetric periodic orbits. According to Proposition \ref{prop:local-symmetry-2d} $(x_{0}^{\eps},y_{0}^{\eps})$ and $(u_{0}^{\eps},v_{0}^{\eps})$ are on the same (unperturbed) periodic interval, denote $\delta\coloneqq x_{0}^{0}-u_{0}^{0}$ then the two periodic orbits at $\eps=0$ are
\begin{align*}
 & \left(x_{0}^{0},y_{0}^{0}\right)\cds\left(x_{q-1}^{\eps},y_{0}^{0}\right)\\
 & \left(x_{0}^{0}+\delta,y_{0}^{0}\right)\cds\left(x_{q-1}^{\eps}+\delta,y_{0}^{0}\right)
\end{align*}
Since the two orbits are periodic (for sufficiently small $\eps$), we have that
\begin{equation}   
\label{eq:sin-sum-zero}
\eps\sum_{k=0}^{q-1}f\left(x_{k}^{\eps}\right)=\eps\sum_{k=0}^{q-1}f\left(u_{k}^{\eps}\right)=0
\end{equation}
As $f$ is continuous, we have that $\sum_{k=0}^{q-1}f\left(x_{k}^{0}\right)=\sum_{k=0}^{q-1}f\left(u_{k}^{0}\right)=0$.

Define
\[
S:=\sum_{k=0}^{q-1}\sin(2\pi x_k^0),
\qquad
C:=\sum_{k=0}^{q-1}\cos(2\pi x_k^0).
\]
Using $u_k^0=x_k^0+\delta$ and the angle-addition formula,
\[
\sin(2\pi(x_k^0+\delta))
=\sin(2\pi x_k^0)\cos(2\pi\delta)
+\cos(2\pi x_k^0)\sin(2\pi\delta),
\]
and summing over $k$ gives
\[
\sum_{k=0}^{q-1}\sin(2\pi u_k^0)
=\cos(2\pi\delta)\,S+\sin(2\pi\delta)\,C.
\]
Together with \eqref{eq:sin-sum-zero} and $S=0$, this implies
\[
\sin(2\pi\delta)\,C=0.
\]
Hence either $C=0$ or $\sin(2\pi\delta)=0$. In the first case,
\[
\sum_{k=0}^{q-1}\cos(2\pi x_k^0)=0
\quad\text{and}\quad
\sum_{k=0}^{q-1}\sin(2\pi x_k^0)=0,
\]
so the center of mass of the points $\{(\cos 2\pi x_k^0,\sin 2\pi x_k^0)\}_{k=0}^{q-1}$ lies at the origin, and the set of $x$--coordinates is balanced, and, similarly, by the addition formula, so is the set of points $\{(\cos 2\pi u_k^0,\sin 2\pi u_k^0)\}_{k=0}^{q-1}$. Then, by the addition formula for $x_k^0+\delta/2$, the symmetric periodic orbit belonging to the same periodic interval, the symmetric orbit must be balanced as well.

In the second case, $\sin(2\pi\delta)=0$, so $\delta\in\{0,\tfrac12\}$. The case $\delta=0$ corresponds to a symmetric orbit and is excluded by assumption. The case $\delta=\tfrac12$ implies that two distinct points in the same periodic interval differ by one half, which forces the interval to have length of at least $\tfrac12$. By the structure of periodic intervals, this can occur only for a fixed interval, in which case the return map $F_y$ is the identity, contradicting the assumption that the map has more than one elemental interval. Thus this case is impossible.

Therefore $C=0$, and the set of $x$--coordinates is balanced for both the non-symmetric and symmetric orbits belonging to this interval.
\end{proof}

Applying the above calculations for $f=\sin\left(2\pi\ensuremath{\ell}x\right)$
for some large $\ell$, shows that $\sum_{k=1}^{q}f\left(x_{k}\right)=0$ may occur for a pair
of points on a periodic or fixed interval even if they are not balanced: the condition on $\delta$ becomes  $\sin(2\pi \ell \delta)=0$, so non-symmetric periodic points belonging to unperturbed periodic intervals of length larger than $\delta$ may persist. In fact, in this case, by Theorem \ref{thm:periodic-pts-persist}, the unbalanced periodic orbits do persist:

Let $d_q$ denote the length of the unperturbed $q$-periodic interval
centered at the $q$-periodic symmetric orbit $X^{0,s}$.
An orbit $X$ is said to be $\ell$-balanced if it is balanced with respect to the function $f(x)=\sin(2 \pi \ell x)$.

\begin{thm}
    \label{thm:balancedorbitslargeell}
Let $f(x)=\sin(2\pi \ell x)$ and let $X^{0,s}$ denote the unperturbed
$q$-periodic symmetric orbit of the FIEM, i.e.\ the midpoints of the
$q$-periodic intervals of width $d_q$.
If $X^{0,s}$ is not $\ell$-balanced, then:
\begin{itemize}
\item for $\ell < \lceil 1/d_q\rceil$, there are no persisting non-symmetric
$q$-periodic orbits;
\item for $\ell \ge \lceil 1/d_q\rceil$, there are exactly
$2\big(\lceil \ell d_q\rceil-1\big)$ persisting non-symmetric $q$-periodic
orbits, located at
\[
x_k^{0,s}\pm \frac{m}{2\ell},
\qquad
m=1,\dots,\lceil \ell d_q\rceil-1,\ k=0,\dots,q-1.
\]
\end{itemize}

If $X^{0,s}$ is $\ell$-balanced, then any persisting pair of non-symmetric
$q$-periodic orbits must be $\ell$-balanced, and their displacement from $X^{0,s}$
is not restricted to integer multiples of $1/(2\ell)$.
\end{thm}

\begin{proof}
We begin with the case $q=1$.
The symmetric fixed point $x^{0,s}=\tfrac12$ is not $\ell$-balanced.
The non-central roots of
$f(x)=\sin(2\pi\ell x)=0$ are
\[
x_m=\tfrac12\pm \frac{m}{2\ell}, \qquad m\in\mathbb{N^+}.
\]
Such a root lies inside the central fixed interval
$\big(\tfrac12-\tfrac{d_1}{2},\,\tfrac12+\tfrac{d_1}{2}\big)$
if and only if $m<\ell d_1$, so for $\ell<\lceil 1/d_1\rceil$ there are no persisting fixed points.
For $\ell\ge \lceil 1/d_1\rceil$, the admissible values are
$m=1,\dots,\lceil \ell d_1\rceil-1$, yielding exactly
$2(\lceil \ell d_1\rceil-1)$ additional fixed points.
Moreover,
\[
f'(x_m)=2\pi\ell\cos\big(2\pi\ell x_m\big)
      =2\pi\ell\cos\big(\pi(\ell\pm m)\big)\neq 0,
\]
so these fixed points are unbalanced and persist.

We now consider the case $q>1$.
For an unperturbed $q$-periodic orbit $X^0=(x_0^0,\dots,x_{q-1}^0)$ define
\[
S_\ell(X^0):=\sum_{k=0}^{q-1}\sin(2\pi\ell x_k^0),
\qquad
C_\ell(X^0):=\sum_{k=0}^{q-1}\cos(2\pi\ell x_k^0).
\]
Using the identity
\[
\sin(2\pi\ell(x_k+\delta))
=\sin(2\pi\ell x_k)\cos(2\pi\ell\delta)
+\cos(2\pi\ell x_k)\sin(2\pi\ell\delta),
\]
we obtain
\[
S_\ell(X^0+\delta)
=\cos(2\pi\ell\delta)\,S_\ell(X^0)
+\sin(2\pi\ell\delta)\,C_\ell(X^0).
\]

Since $S_\ell(X^{0,s})=0$, it follows that
\[
S_\ell(X^{0,s}+\delta)
=\sin(2\pi\ell\delta)\,C_\ell(X^{0,s}).
\]
Therefore, if $\delta=\pm m/(2\ell)$, then $\sin(2\pi\ell\delta)=\sin(\pm\pi m)=0$ and $S_\ell(x^{0,s}\pm m/(2\ell))=0$, independently of $C_\ell(x^{0,s})$. Moreover,
\[
C_\ell\!\left(X^{0,s}\pm \tfrac{m}{2\ell}\right)
=(-1)^m C_\ell(X^{0,s}).
\]
Hence, if $x^{0,s}$ is not $\ell$-balanced (so $C_\ell(X^{0,s})\neq 0$), all such non-symmetric orbits persist.

If $\delta\neq m/(2\ell)$ for any integer $m$, then $\sin(2\pi\ell\delta)\neq 0$, and the condition $S_\ell(X^{0,s}+\delta)=0$ forces $C_\ell(X^{0,s})=0$.
Thus, additional non-symmetric periodic orbits at arbitrary distances from $x^{0,s}$ can occur only in the $\ell$-balanced case. Then, by the addition formula, they are balanced, and by the symmetry they appear in pairs.
\end{proof}

The behavior of persisting orbits for other antisymmetric functions $f$ (i.e. sums of odd harmonics) is left for future studies. 

\subsection{Linear Stability Analysis of Periodic Points}

The \emph{residue} of an area-preserving map $T$, at a $q$-periodic orbit is given by 
\[
Res\coloneqq\frac{2-\tr\left(D\left(T^{q}\right)\right)}{4}
\]
It gives a convenient way to classify the stability of the periodic points \cite{meiss92symplecticmaps,greene1979method}. When $Res<0$ the periodic point is hyperbolic, when $Res\in\left(0,1\right)$ it is elliptic and when $Res>1$ it is hyperbolic (with reflection).

The Jacobian matrix of $T^{q}$ is 
\begin{align*}
D\left(T^{q}\right)\left(x_{0},y_{0}\right) & =\left(\begin{matrix}1+\frac{d}{dx_{0}}\sum_{k=0}^{q-1}\omega_{\alpha_{k}}\left(y_{k}\right)\quad & \quad\frac{d}{dy_{0}}\sum_{k=0}^{q-1}\omega_{\alpha_{k}}\left(y_{k}\right)\\
\eps\frac{d}{dx_{0}}\sum_{k=1}^{q}f\left(x_{k}\right)\quad & \quad1+\eps\frac{d}{dy_{0}}\sum_{k=1}^{q}f\left(x_{k}\right)
\end{matrix}\right)\\
 & =\left(\begin{matrix}1+\sum_{k=0}^{q-1}\omega_{\alpha_{k}}'\left(y_{k}\right)\frac{\partial y_{k}}{\partial x_{0}}\quad & \quad\sum_{k=0}^{q-1}\omega_{\alpha_{k}}'\left(y_{k}\right)\frac{\partial y_{k}}{\partial y_{0}}\\
\eps\sum_{k=1}^{q}f'\left(x_{k}\right)\frac{\partial x_{k}}{\partial x_{0}}\quad & \quad1+\eps\sum_{k=1}^{q}f'\left(x_{k}\right)\frac{\partial x_{k}}{\partial y_{0}}
\end{matrix}\right)\\
 & =I+\left(\begin{matrix}\sum_{k=1}^{q-1}\omega_{\alpha_{k}}'\left(y_{k}\right)\frac{\partial y_{k}}{\partial x_{0}}\quad & \quad\sum_{k=0}^{q-1}\omega_{\alpha_{k}}'\left(y_{k}\right)\frac{\partial y_{k}}{\partial y_{0}}\\
\eps\sum_{k=1}^{q}f'\left(x_{k}\right)\frac{\partial x_{k}}{\partial x_{0}}\quad & \quad\eps\sum_{k=1}^{q}f'\left(x_{k}\right)\frac{\partial x_{k}}{\partial y_{0}}
\end{matrix}\right)
\end{align*}
The residue is (after some long calculation involving all the derivatives)
\begin{align*}
Res= & \frac{1}{4}\left(2-\tr\left(D\left(T^{q}\right)\right)\right)\\
= & -\frac{1}{4}\left(\sum_{k=1}^{q-1}\omega'\left(y_{k}\right)\frac{\partial y_{k}}{\partial x_{0}}+\eps\sum_{k=1}^{q}f'\left(x_{k}\right)\frac{\partial x_{k}}{\partial y_{0}}\right)\\
\vdots\\
= & -\frac{\eps}{4}\left(\sum_{k=1}^{q-1}f'\left(x_{k}\right)\sum_{\ell=0}^{q-1}\omega'\left(y_{\ell}\right)+f'\left(x_{q}\right)\sum_{\ell=0}^{q-1}\omega'\left(y_{\ell}\right)\right)+O\left(\eps^{2}\right)\\
= & -\frac{\eps}{4}\sum_{k=0}^{q-1}f'\left(x_{k}\right)\sum_{\ell=0}^{q-1}\omega'\left(y_{\ell}\right)+O\left(\eps^{2}\right) \\
=& -\frac{\eps}{4} M(X) + O(\eps^2)
\end{align*}
In the last step we used the periodicity $f'\left(x_{q}\right)=f'\left(x_{0}\right)$.
For $\eps=0$ all the periodic points of the unperturbed map are parabolic, as $Res=0$. 

For persisting unbalanced periodic orbits, the leading order term of the residue does not vanish, so we establish:
\begin{cor}
If $X^0$ is a persisting  periodic orbit satisfying the non-degeneracy conditions of Theorem  \ref{thm:periodic-pts-persist}, then, for sufficiently small $\eps>0$, the orbit  $X^\eps$ is elliptic if $M(X^0)<0$ and is hyperbolic if $M(X^0)>0$.  
\end{cor}

For the unperturbed standard map periodic orbits are equally spaced
along the interval (circle) and hence they are always balanced, satisfying $\sum_{d=0}^{q-1}f'\left(x_{d}\right)=0$. Using $\frac{\partial x_{n}}{\partial x_{0}}=1+O\left(\eps\right)$
and $\frac{\partial y_{n}}{\partial y_{0}}=1+O\left(\eps\right)$
we get the next order of the residue:
\begin{align*}
Res= & -\frac{\eps}{4}\sum_{k=0}^{q-1}f'\left(x_{k}\right)\sum_{\ell=0}^{q-1}\omega'\left(y_{\ell}\right)\\
 & -\frac{\eps^{2}}{4}\left(\sum_{k=1}^{q-1}\omega'\left(y_{k}\right)\sum_{\ell=1}^{k}f'\left(x_{\ell}\right)\sum_{m=1}^{\ell-1}\omega'\left(y_{m}\right)\sum_{n=1}^{m}f'\left(x_{n}\right)\right.\\
 & \left.\qquad\quad+\sum_{k=1}^{q}f'\left(x_{k}\right)\sum_{\ell=1}^{k-1}\omega'\left(y_{\ell}\right)\sum_{m=1}^{\ell}f'\left(x_{m}\right)\sum_{n=0}^{m-1}\omega'\left(y_{n}\right)\right)\\
 & +O\left(\eps^{3}\right)
\end{align*}
 A numerical check confirms that for the standard map 
the next order term does not vanish, implying that
\[
Res_{\O{std}map}=O\left(\eps^{2}\right).
\]
We thus conclude that breaking the balance condition  produces an order of magnitude
change in the residue:

\begin{cor}
If $X^\eps$ is a balanced  periodic orbit, then, generically, $Res = O(\eps^2)$. 
\end{cor}

According to \cite{lichtenberg1979determination,greene1980calculationofkam},
the residue is tightly related to the island overlap criterion of
the map. Hence, the order of magnitude change of the residue can explain
the observed growth in the size of the resonance islands, and the
possible absence of invariant curves that will be discussed in the
next section. Notably, unbalanced periodic orbits may also appear in smooth generalized standard maps. The appearance of larger chaotic regions in such maps was observed numerically in \cite{cetin2022generalization}, for example. An elegant derivation of the residue in the special case of the standard map is given in \cite{bountis1981onthestability}.

\subsection{Bifurcations leading to non-symmetric periodic orbits}

We demonstrate a mechanisms by which a pitchfork  bifurcation of symmetric periodic points give rise to non-symmetric points. We start by listing the general properties of non-symmetric periodic orbits that led us to the detection of such periodic points
in perturbed symmetric FIEMs.

First, we notice a constraint that appears for non-symmetric   period 2 orbits of $f=\sin(2 \pi x)$. We then construct an example for which this constraint is satisfied at a bifurcation point of a symmetric period 2 orbit and demonstrates the appearance of such orbits. 

Consider  a period $2$ orbit,
$\left(x_{0},y_{0}\right),\left(x_{1},y_{1}\right)$
and its symmetric pair $\left(u_{0},v_{0}\right)\coloneqq L\left(x_{0},y_{0}\right),\left(u_{1},v_{1}\right)\coloneqq L\left(x_{1},y_{1}\right)$. Then 
\begin{align*}
y_{2} & =y_{0} & \Rightarrow &  & f\left(x_{1}\right)+f\left(x_{2}\right) & =0 & \Rightarrow &  & f\left(x_{0}+\omega\left(y_{0}\right)\right)+f\left(x_{0}+\omega\left(y_{0}\right)+\omega\left(y_{1}\right)\right) & =0\\
v_{2} & =v_{0} & \Rightarrow &  & f\left(u_{1}\right)+f\left(u_{2}\right) & =0 & \Rightarrow &  & f\left(u_{0}+\omega\left(v_{0}\right)\right)+f\left(u_{0}+\omega\left(v_{0}\right)+\omega\left(v_{1}\right)\right) & =0
\end{align*}
By Proposition \ref{prop:local-symmetry-2d}, the local symmetry preserves the  $y$ coordinate and the elemental interval, so $v_0=y_0$, $v_1=y_1$, 
$\omega\left(y_{0}\right)=\omega\left(v_{0}\right)$
and $\omega\left(y_{1}\right)=\omega\left(v_{1}\right)$. Therefore, 
\begin{gather*}
\sin\left(2\pi\left(x_{0}+\omega\left(y_{0}\right)\right)\right)+\sin\left(2\pi\left(x_{0}+\omega\left(y_{0}\right)+\omega\left(y_{1}\right)\right)\right)=\\
2\sin\left(2\pi\left(x_{0}+\omega\left(y_{0}\right)+\frac{1}{2}\omega\left(y_{1}\right)\right)\right)\cos\left(\pi\omega\left(y_{1}\right)\right)=0
\end{gather*}
\begin{gather*}
\sin\left(2\pi\left(u_{0}+\omega\left(y_{0}\right)\right)\right)+\sin\left(2\pi\left(u_{0}+\omega\left(y_{0}\right)+\omega\left(y_{1}\right)\right)\right)=\\
2\sin\left(2\pi\left(u_{0}+\omega\left(y_{0}\right)+\frac{1}{2}\omega\left(y_{1}\right)\right)\right)\cos\left(\pi\omega\left(y_{1}\right)\right)=0
\end{gather*}
as $x_{0}\neq u_{0}$, we got either two roots of $\sin\left(2\pi\left(x+a\right)\right)$, or $\cos\left(\pi\omega\left(y_{1}\right)\right)=0$. 
The first case cannot occur, as the two roots $x_{0},u_{0}$ must be distanced $\frac{1}{2}$
apart, while being part of the same period 2 interval belonging to the same elemental subinterval.

The following example demonstrates a symmetric FIEM tailored to have
such non-symmetric periodic orbit, by admitting $\omega\left(y_{1}\right)=\frac{1}{2}$. To this aim we construct a family with an unbalanced symmetric period $2$ orbit which looses its transversality condition as $\eps$ is increased, exactly at $y_1^\eps$ value for which  $\omega\left(y_{1}^\eps\right)=\frac{1}{2}$.

First, to find symmetric period $2$ intervals, we choose 
for the unperturbed FIEM, a linear $4$ IEM family with $\lambda^1$ having the reverse lengths of $\lambda^0$, namely $\lambda_j(y) = y \lambda^1_j + (1-y)\lambda^0_j =  (1-y) \lambda^0_j + y\lambda^0_{5-j}, j=1,\dots 4$.  Then, at $y=\frac12$, there is a degenerate $4$-IEM of the form $\lambda(\frac12)= (a,b,b,a)$, with two  period $2$ intervals, $a,b=\frac12-a $ where $a=\frac12(\lambda^0_1+\lambda^0_4)<\frac12$. 

\begin{lem}
Provided \begin{equation}
    \lambda_1^0 \ne \lambda_4^0, \quad   \lambda^0_2-\lambda^0_3 \ne 2(\lambda^0_4-\lambda^0_1),
\end{equation}the symmetric period-2 orbits of the linear family $\lambda_j(y) =   (1-y) \lambda^0_j + y \lambda^0_{5-j}, j=1,\dots 4$  at $y=\frac12$ are persisting unbalanced symmetric periodic orbit. 
\end{lem}
  \begin{proof}
For the symmetric 4-IEM with $\lambda =(a,b,b,a)$, the midpoints of the intervals, $x \in \{\frac{a}{2},\frac{a}{2}+\frac{1}{4}\} $, are mapped to their symmetric pairs, $1-x$.  Since $f(x) + f(1-x)=0$  and $f'(x) + f'(1-x)=2 f'(x)$ by the symmetry, as long as  $\frac{a}{2}, a+\frac{b}{2}=\frac{a}{2}+\frac{1}{4}$ are not zeros of $f'(x)$ (in the interval $(0,\frac12)$)  the mid-points are unbalanced period $2$ orbits.  Since, for $f=\sin(2 \pi x)$, the only zero in this interval is $\frac{1}{4}$ and $a \in (0,\frac12)$, the orbits are unbalanced, $f'(\frac{a}{2})>0$ and  $f'(\frac{a}{2}+\frac{1}{4})<0$. 

Let us verify that their twist does not vanish. 
Indeed, since  $\omega_\alpha'(y) = \Omega (\lambda^1-\lambda^0) $ with $\Omega $ denoting the matrix for reversing permutation, for the first periodic  interval, as $\omega(y_0)=\omega_1(y_0), \omega(y_1)=\omega_4(y_0)$, we obtain $\sum_{j=0}^1 \omega'(x_0^j)=\omega_1'(y_0)+\omega_4'(y_0) = (\lambda^1_4-\lambda^0_4)-(\lambda^1_1-\lambda^0_1) = 2(\lambda^0_1-\lambda^0_4)$  and, similarly, for the second periodic interval $\sum_{j=0}^1 \omega'(x_0^j)=\omega_2'(y_0)+\omega_3'(y_0) =2(\lambda^0_2-\lambda^0_3)+4(\lambda^0_1-\lambda^0_4)$. Hence, $M_a=4 f'(\frac{a}{2})(\lambda^0_4-\lambda^0_1)$ and $M_b= 4f'(\frac{a}{2}+\frac{1}{4})((\lambda^0_3-\lambda^0_2)+2(\lambda^0_4-\lambda^0_1))$. 
%$\sum_{j=0}^1 \omega'(x_0^j)= -2(\lambda^1_1-\lambda^0_1)+(\lambda^1_3-\lambda^0_3)-(\lambda^1_2-\lambda^0_2)+ 2(\lambda^1_4-\lambda^0_4)$ $\sum_{j=0}^1 \omega'(x_0^j)=2(\lambda^0_1-\lambda^0_4)-(\lambda^0_3-\lambda^0_2)+(\lambda^0_2-\lambda^0_3)- 2(\lambda^0_4-\lambda^0_1)=-2(\lambda^0_3-\lambda^0_2)-4(\lambda^0_4-\lambda^0_1)$  

By the transversality assumption in the lemma (which is generically satisfied), $M_{a,b}\ne 0$, hence, by Theorem \ref{thm:balancedorbits},  the symmetric periodic orbits persist.
\end{proof}

 By Theorem \ref{thm:balancedorbitslargeell} there are no other persisting non-symmetric period $2$ orbits at $y=\frac12$. We suspect there can be no other persisting period 2 orbits at all, namely that the unperturbed family cannot have other $y$ values admitting period $2$ orbits, but this needs to be shown.

As $\eps$ is increased, the symmetric periodic orbits,  $(x_0^\eps,y_0^\eps),(x_1^\eps,y_1^\eps)$,  persist, with $y_{0,1}^\eps = \frac12 +O(\eps)$.
Notice that a symmetric period-2 orbit is always of the form $X^\eps=(x^\eps_0,y^\eps_0),(1-x^\eps_0,y^\eps_0+\eps f(x^\eps_0))$, so $\omega(y_0^\eps)=\omega_{\alpha(x_0^\eps)}(y^\eps_0)=1-2x_0^\eps$ and  $\omega(y_1^\eps)=\omega_{\alpha(1-x_0^\eps)}(y^\eps_0+\eps f(x^\eps_0))=-1+2x_0^\eps$. For the linear family, as long as the symmetric periodic orbit does not cross a singularity line, the twist is unchanged by the perturbation, so $M(X^\eps)=\frac{f'(x_0^\eps)  }{f'(x_0^0) }M (X^0)$. It follows (see Remark \ref{rem:persistpertorbits})  that the symmetric periodic orbit can bifurcate only when $x_0^\eps=\frac{1}{4}$, which implies that $\cos\left(\pi\omega\left(y_{1}^\eps\right)\right)=0$, namely, the condition for the birth of non-symmetric orbits is met exactly at this bifurcation point. Since $x_0^\eps=\frac{a}{2} + O(\eps)$, to obtain such a bifurcation we need to take $a \approx \frac12$ and expect to detect a bifurcation at $\eps =O(|a-\frac12|)$.   This scenario is demonstrated below,  in Figure \ref{fig:pitchfork-bif}.

\begin{example}
\label{exa:pitchfork-bifurcation}$4$-FIEM with pitchfork bifurcation of  an elliptic symmetric $2$-periodic orbit into two elliptic non-symmetric $2$-periodic orbits and a hyperbolic symmetric $2$-periodic orbit is shown in Figure \ref{fig:pitchfork-bif}.  Choosing $\lambda_{0}=\left(0.38,0.01,0.01,0.6\right),\lambda_{1}=\left(0.6,0.01,0.01,0.38\right)$ 
 we get  $a=0.49$, hence $M_a=4 f'(\frac{a}{2})(\lambda^0_1-\lambda^0_4)$ and $M_b= 4f'(\frac{a}{2}+\frac{1}{4})((\lambda^0_2-\lambda^0_3)+2(\lambda^0_1-\lambda^0_4))$. 
 $M_a =-0.88 \pi \cos (\pi a) <0$ and  $M_b =  -0.88 \pi \cos (\pi a +\pi/2 ) >0$, so the first is elliptic and the second is hyperbolic, as shown in the figure.

\begin{figure}[h]
\begin{centering}
\includegraphics[width=0.5\textwidth]{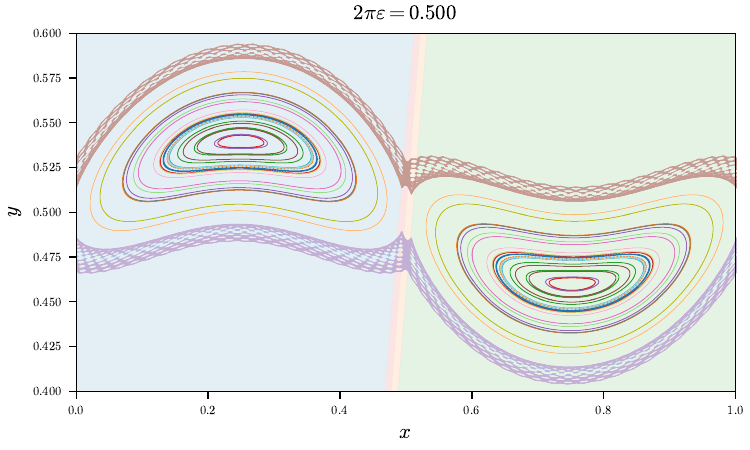}\includegraphics[width=0.5\textwidth]{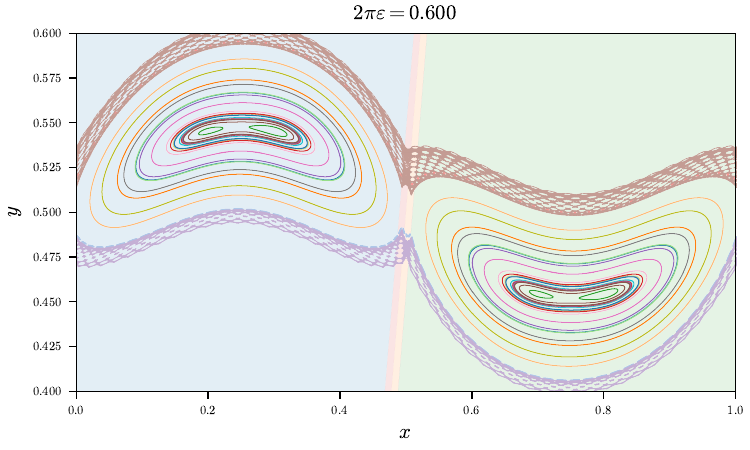}
\par\end{centering}
\begin{centering}
\includegraphics[width=0.5\textwidth]{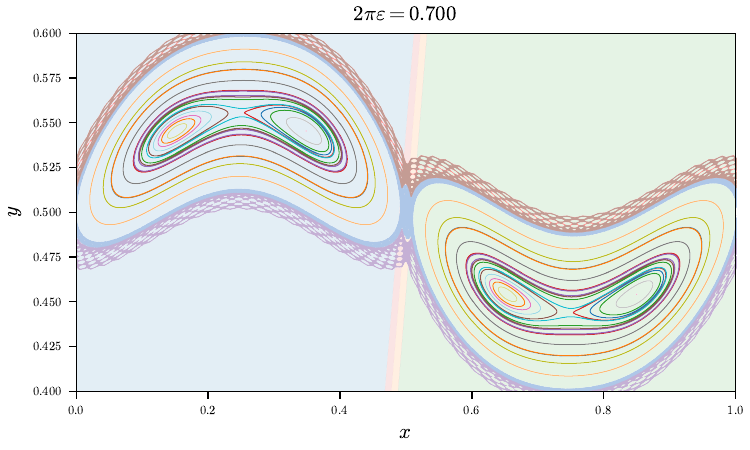}\includegraphics[width=0.5\textwidth]{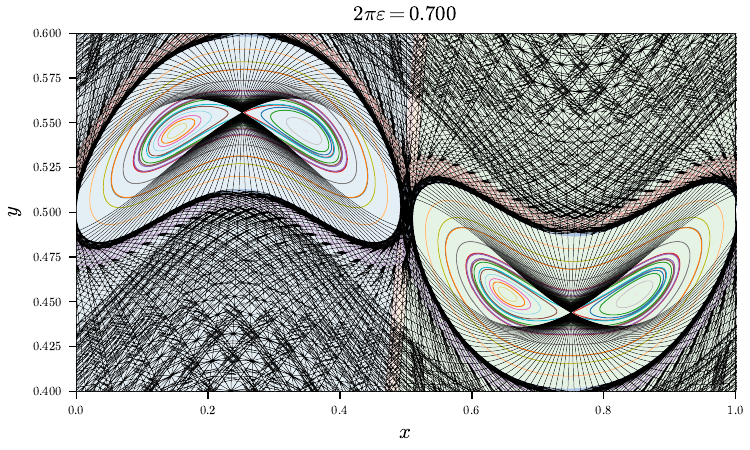}
\par\end{centering}
\centering{}\caption{\label{fig:pitchfork-bif}The pitchfork bifurcation of the symmetric
$2$-period orbit of the perturbed linear FIEM given by $\lambda_{0}=\left(0.38,0.01,0.01,0.6\right),\lambda_{1}=\left(0.6,0.01,0.01,0.38\right)$.
The bottom right figure shows the symmetry lines $\Gamma_{-60}\protect\cds\Gamma_{60}$
in black.}
\end{figure}
\end{example}

\section{Invariant Curves}\label{sec:invariantcurves}

A \emph{ rotational invariant curve} is a continuous curve $\gamma$ that is homotopic to the line $y=0$ and invariant under the map, i.e. $T\left(\gamma\right)=\gamma$. For an unperturbed FIEM, every horizontal line is invariant and the space admits a fibration by such curves. Moreover,  on most of these horizontal lines the motion is ergodic.
In the case of the standard map, the celebrated Kolmogorov--Arnold--Moser (KAM) theory establishes that most of the rotational invariant curves persist under sufficiently small perturbation. By a theorem of Birkhoff \cite{meiss92symplecticmaps}, these invariant curves of the standard map are graphs of functions. Converse KAM theory \cite{mackay1985converseKAM}  shows that for perturbations exceeding a certain critical value $\eps_{c}$, no invariant circles survive. For the standard map it is estimated that the critical threshold is $2\pi\eps_{c}\approx0.9716\ldots$ \cite{greene1979method}.

\begin{figure}[h]
\begin{centering}
${\scriptstyle 2\pi\eps=0.1}$
\par\end{centering}
\begin{centering}
\includegraphics[width=0.15\textwidth]{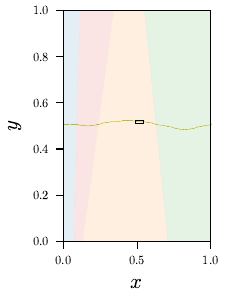}\includegraphics[width=0.15\textwidth]{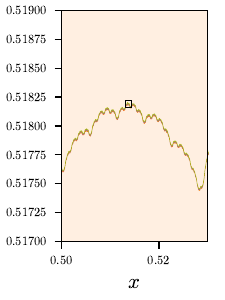}\includegraphics[width=0.7\textwidth]{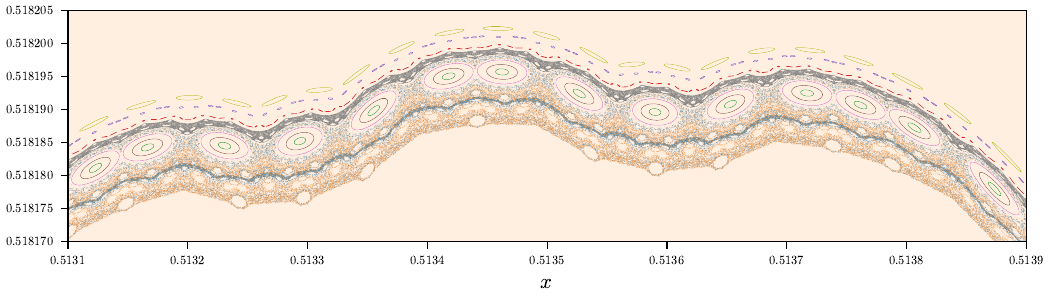}
\par\end{centering}
\centering{}\caption{\label{fig:pert-FIEM-invariant-zoom} Series of close-ups (indicated by rectangles) of allegedly rotational invariant curves of the linear FIEM from Figure \ref{fig:pert-FIEM} with small perturbation. We observe that such curves are either a collection of many small periodic islands, or chaotic. The periods of the periodic islands in this figure are $16,896$ (green), $89,971$ (purple) and $141,227$ (red).}
\end{figure}

We conjecture that, in general, for arbitrarily small $\eps>0$ no invariant curves exist for perturbed symmetric FIEMs. Our observations support this conjecture, see Figure \ref{fig:pert-FIEM-invariant-zoom}. We list several necessary conditions that any such hypothetical rotational invariant curve would have to satisfy, and argue that these are unlikely to be met in practice. 

\begin{prop}
Any minimal rotational invariant curve is symmetric, namely $S\left(\gamma\right)=\gamma$. 
\end{prop}

\begin{proof}
$T^{-1}\left(\gamma\right)=\gamma\quad\Rightarrow\quad S\circ T\circ S\left(\gamma\right)=\gamma\quad\Rightarrow\quad T\left(S\left(\gamma\right)\right)=S\left(\gamma\right)$.

Assuming the primary symmetry $S$ is continuous, $S\left(\gamma\right)$ is another rotational invariant curve, intersecting $\gamma$ at the symmetry line of $S$. The intersection of two invariant sets is an invariant set. Since we assume the invariant curve is minimal, this implies that the two invariant curves that intersect must coincide, $S\left(\gamma\right)=\gamma$.
\end{proof}

Without requiring the invariant curve to be minimal, it might be possible that there are two invariant curves of parabolic periodic orbits (see analogous rigidity results for smooth twist maps \cite{fierobe2025existence}) or that there are hyperbolic periodic orbits on this curve with coinciding stable and unstable manifolds, creating two invariant curves that are the symmetric image of each other. Both scenarios appear to be non-generic, yet plausible for specific parameter values.

Inspired by the embedding results in \cite{ashwin2020embeddings}, any rotational invariant curve that is a graph of a function that is Lipschitz continuous, would be an embedding of a generalized interval exchange map for sufficiently small $\eps$. A generalized interval exchange map is a map that, instead of being a piecewise translation, is a piecewise orientation-preserving homeomorphism \cite{marmi2012linearization}.
\begin{lem}
Let $T$ be a perturbed symmetric $d$-FIEM. If $\,\gamma$ is a rotational invariant curve that is given by a graph of a Lipschitz continuous function $Y:\s^{1}\to\r$, then it intersects the discontinuity lines of  $T$ along a single horizontal line when $d$ is odd, and along two horizontal lines when $d$ is even.
\end{lem}

\begin{proof}
Let $\tilde T$ denote the map $T$ restricted to $\gamma=Y\left(\s^{1}\right)$. Then $\tilde T$ is a generalized IEM that maps elemental subintervals to subintervals (not by translation). Denote the intersection points of $\gamma$ with the discontinuity lines of the map $T$ by $\left(\tilde x_{1},Y\left(\tilde x_{1}\right)\right)\cds\left(\tilde x_{d},Y\left(\tilde x_{d}\right)\right)$. Then the intervals of the generalized IEM $\tilde T$, $J_{1}\cds J_{d}$, are given by $J_{i}=\left\{ \left(\begin{matrix}x\\
Y\left(x\right)
\end{matrix}\right):x\in\left[\tilde x_{i},\tilde x_{i+1}\right)\right\} $, we consider the interval as a circle, so the index of subintervals is $\text{mod }d$.

We will now use the fact that the generalized IEM is symmetric, so the end point of the $i$'th interval must be mapped to the start point of the $\left(i-1\right)$'th interval. Let $\left\{ \tilde x_{i}^{\left(n\right)}\right\} $ be a monotonically increasing sequence converging to $\tilde x_{i}$. Then
\begin{align*}
T\left(\begin{matrix}\tilde x_{i}^{\left(n\right)}\\
Y\left(\tilde x_{i}^{\left(n\right)}\right)
\end{matrix}\right) & \stackrel{n\to\infty}{\longrightarrow}T\left(\begin{matrix}\tilde x_{i-2}\\
Y\left(\tilde x_{i-2}\right)
\end{matrix}\right)\\
\left(\begin{matrix}F_{Y\left(\tilde x_{i}^{\left(n\right)}\right)}\left(\tilde x_{i}^{\left(n\right)}\right)\\
Y\left(\tilde x_{i}^{\left(n\right)}\right)+\eps f\left(F_{Y\left(\tilde x_{i}^{\left(n\right)}\right)}\left(\tilde x_{i}^{\left(n\right)}\right)\right)
\end{matrix}\right) & \stackrel{n\to\infty}{\longrightarrow}\left(\begin{matrix}F_{Y\left(\tilde x_{i-2}^{\phantom{\left(n\right)}}\right)}\left(\tilde x_{i-2}^{\phantom{\left(n\right)}}\right)\\
Y\left(\tilde x_{i-2}\right)+\eps f\left(F_{Y\left(\tilde x_{i-2}^{\phantom{\left(n\right)}}\right)}\left(\tilde x_{i-2}^{\phantom{\left(n\right)}}\right)\right)
\end{matrix}\right)
\end{align*}

From the first coordinate $F_{Y\left(\tilde x_{i}^{\left(n\right)}\right)}\left(\tilde x_{i}^{\left(n\right)}\right)\stackrel{n\to\infty}{\longrightarrow}F_{Y\left(\tilde x_{i-2}\right)}\left(\tilde x_{i-2}\right)$, then, by continuity of $f$, the second coordinate gives 
\[
Y\left(\tilde x_{i}^{\left(n\right)}\right)\stackrel{n\to\infty}{\longrightarrow}Y\left(\tilde x_{i-2}^{\phantom{\left(n\right)}}\right)
\]
Since $Y$ is continuous, this gives $Y\left(\tilde x_{i}\right)=Y\left(\tilde x_{i-2}\right)$, hence for even $d$ we have two chains of equalities
\begin{align*}
Y\left(\tilde x_{d}\right) & =Y\left(\tilde x_{d-2}\right)=\cdots=Y\left(\tilde x_{1}\right)=Y\left(\tilde x_{d}\right)\\
Y\left(\tilde x_{d-1}\right) & =Y\left(\tilde x_{d-3}\right)=\cdots=Y\left(\tilde x_{0}\right)=Y\left(\tilde x_{d-1}\right)
\end{align*}
 and for odd $d$ we get a single chain of equalities.
\end{proof}

In general, the problem of finding a rotational invariant curve of \eqref{eq:fiem-T} that is a graph of a function boils down to finding a solution to the functional equation
\begin{equation}
\begin{aligned}
T\left(\begin{matrix}x\\
Y\left(x\right)
\end{matrix}\right) & =\left(\begin{matrix}x'\\
Y\left(x'\right)
\end{matrix}\right)\\
 & \Updownarrow\\
Y\left(F_{Y\left(x\right)}\left(x\right)\right)-Y\left(x\right) & =\eps f\left(F_{Y\left(x\right)}\left(x\right)\right)
\end{aligned}
\end{equation}

where $F_{Y\left(x\right)}\left(x\right)$ is a generalized IEM on the $x$-axis. 

Formally, if
\[
Y(x)=y+\varepsilon y_1(x;\varepsilon),
\]
where $y_1(x;\varepsilon)$ is bounded and Lipschitz in $x$ uniformly in $\varepsilon$, then by Arzelà--Ascoli one may extract a subsequence of the functions $y_1(\cdot;\varepsilon)$ converging uniformly to a limit which, away from the discontinuity set, satisfies\footnote{more precisely, the minus sign of its limit} a cohomological equation of the form
\begin{equation}
u\circ F_{y}^{-1}-u=f.
\label{eq:cohomol}
\end{equation}
Marmi--Moussa--Yoccoz \cite{MMY2005} proved that if $ F_{y}^{-1}$ is of Roth type and $f$ belongs to a finite-codimensional subspace that is determined by $ F_{y}^{-1}$, then the cohomological equation has bounded solutions.  The works of Forni \cite{Forni1997,Forni2002}, of  Zorich \cite{zorich1997deviation} and of Kontsevich and Zorich \cite{kontsevich1997lyapunovexponentshodgetheory} suggest that the vanishing of the obstruction functionals is a necessary condition for the existence of bounded solutions, but the exact formulation of this direction is not fully resolved yet (we thank G. Forni for explaining this point). If it holds, then we can expect that for a generic FIEM and a generic $f$, such invariant curves do not exist. A precise formulation of this conjecture,
as well as its proof, is left for future works.

\section{Perturbed Symmetric FCEM in Hamiltonian Impact Systems \label{sec:HIS}}

We end the paper with an explicit construction of a FIEM which appears as an iso-energy return map of a Hamiltonian impact system. Consider a Hamiltonian impact system describing the motion of a particle on the plane under a separable, smooth, unimodal potential.
The potential induces independent oscillations in the horizontal and vertical directions. The system has a semi-infinite barrier parallel to one of the axes, with the particle undergoing elastic reflection at the barrier. In the unperturbed case, the system is pseudointegrable: the partial energies in the horizontal and the vertical directions remain constant along trajectories, and the motion is conjugate, via action-angle transformation, to that of a billiard flow in a polygonal domain - a torus with two slits (see the right column of Figure  \ref{fig:HIS-FCEM}  which shows two polygonal domains along the FCEM). Correspondingly, return maps to suitable circular sections are CEMs, see \cite{becker2020impact}.

For a fixed total energy, the phase space admits a fibration by invariant surfaces, conjugated to translation surfaces. The collection of CEMs along the fibration gives a symmetric FCEM, where the symmetry comes from the central symmetry of the translation surface (the time-reversal symmetry of the mechanical system). Figure \ref{fig:HIS-FCEM} illustrates the FCEM arising in such a system. 

Upon introducing a small coupling potential that depends on both directions of motion, the energy distribution varies along the trajectories. This perturbation breaks the integrability and leads to return maps that can be modeled by a perturbed FCEM \cite{romkedar2025hovering}. Note that such systems start and end with rotation maps, as no impacts occur when the motion is mostly vertical or mostly horizontal. For a sufficiently small perturbation, each of the rotation domains has KAM tori, separating the non-impacting motion from orbits that enter the impacting regime.

The perturbed linear FIEM introduced in the previous sections may be viewed as a linearization of such perturbed return maps near non-singular level sets. We expect that in regions away from singularities and under sufficiently small perturbations, this linearization provides a local model of the dynamics (though  we note that the form of the perturbation forcing term, $\eps f(x)$, may depend non-trivially on the level set). Figure \ref{fig:HIS-FCEM-pert} shows simulations of the perturbed FCEM for the forcing function $f(x)=\sin(2 \pi x)$. 

\begin{figure}[h]
\centering
\begin{centering}
\begin{minipage}[c][1\totalheight][b]{0.7\textwidth}%
\begin{center}
\includegraphics[width=1\textwidth]{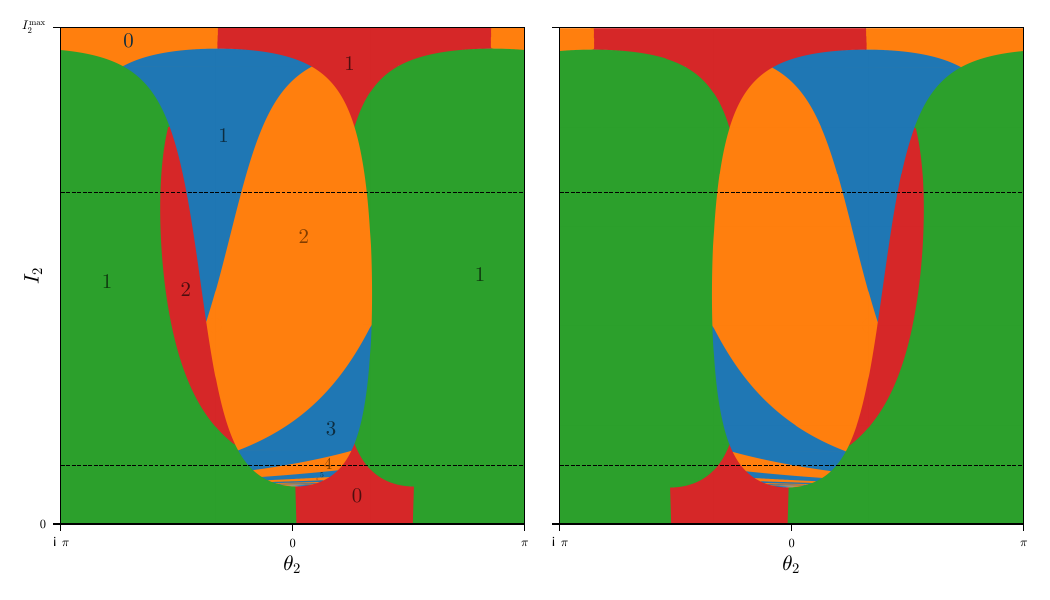}
\par\end{center}%
\end{minipage}%
\begin{minipage}[c][1\totalheight][t]{0.3\textwidth}%
\begin{center}
\includegraphics[width=1\textwidth]{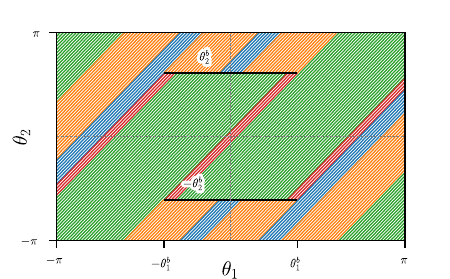}
\par\end{center}
\begin{center}
\includegraphics[width=1\textwidth]{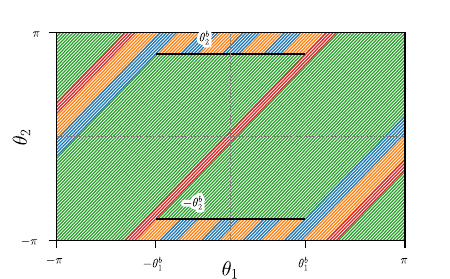}
\par\end{center}
~%
\end{minipage}
\par\end{centering}
\centering{}\caption{\label{fig:HIS-FCEM}The FCEM arising from the Hamiltonian impact system with separable unimodal potential and a horizontal barrier. The left panel shows the elemental subregions of the domain, the middle panel is their image. The numbers indicate the number of impacts per cycle on the upper side of the barrier (blue, orange) and on the lower side of the barrier (red, green). The two panels on the right demonstrate the translation surfaces corresponding to the dashed horizontal lines on the FCEM. The CEM is the return map to the $\theta_{1}=0$ circle with trajectories colored according to their elemental interval. See the starting position of the intervals at the left, and their position upon return to the section on the right. Note the symmetry around the zero angle, coming from the central symmetry of the translation surface.}
\end{figure}

\begin{figure}[h]
\centering
\begin{centering}
\includegraphics[width=0.333\textwidth]{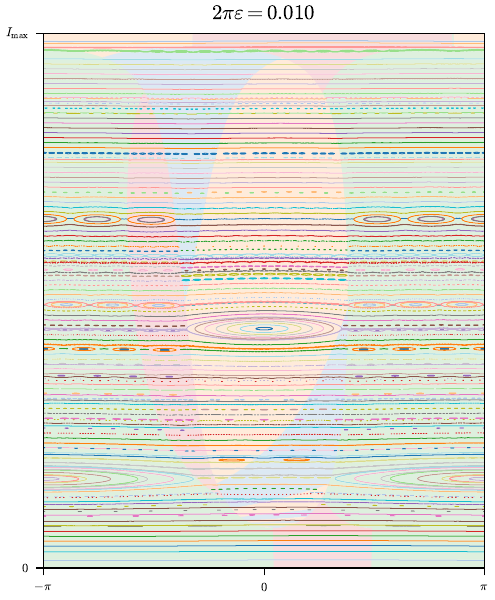}\includegraphics[width=0.333\textwidth]{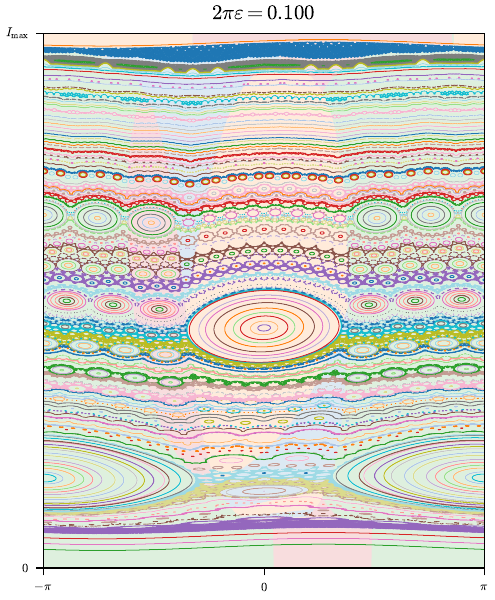}\includegraphics[width=0.333\textwidth]{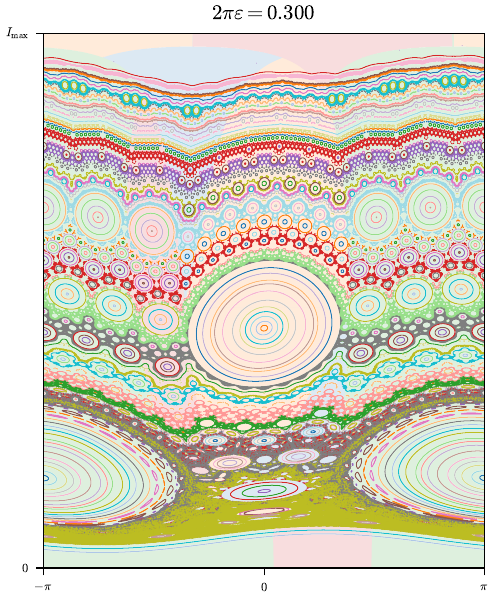}
\par\end{centering}
\centering{}\caption{\label{fig:HIS-FCEM-pert}The perturbed FCEM arising from the Hamiltonian
impact system with separable unimodal potential and a horizontal barrier
under different perturbations. Trajectories are plotted in different
colors. The shaded background colors indicate the elemental subregions
of the unperturbed FCEM. Notice the reflection symmetry of the symmetric
periodic points in each elemental subregion, in accordance with Proposition
\ref{prop:local-symmetry-2d}.}
\end{figure}

\section{Discussion}

We study perturbed families of interval exchange maps (FIEMs), focusing on the role of time-reversal symmetry in the existence and organization of periodic orbits. We show that at $\varepsilon = 0$ only symmetric periodic intervals exist, and that their midpoints correspond to isolated symmetric periodic orbits which persist for sufficiently small $\varepsilon$. We then derive conditions under which additional, non-symmetric periodic orbits persist. This leads to a distinction between balanced and non-balanced orbits, and to an analysis of how highly oscillatory perturbations of the form $f(x) = \sin(2\pi \ell x)$, with large $\ell$, give rise to additional non-balanced, non-symmetric periodic orbits. Taken together, these results provide a detailed classification of the periodic orbits that persist in symmetric perturbed FIEMs under single-wave perturbations. 

The persistence conditions involve finite Birkhoff sums of the perturbation along periodic IET orbits. Although our setting concerns finite periodic-orbit sums rather than asymptotic ergodic averages, these sums are naturally related to the broader theory of Birkhoff sums \cite{Forni2002,MMY2005}. A deeper understanding of these sums, especially for multi-wave perturbations, may reveal further connections between periodic-orbit persistence and the renormalization theory of interval exchange transformations. 

We further show that, as in the smooth setting, symmetric periodic points can be located for finite values of $\varepsilon$ by searching along the deformed symmetry lines. As $\varepsilon$ increases, new intersections of the folded symmetry lines are created, leading to the birth of additional symmetric periodic orbits, either through a smooth mechanism, as in the standard map, or through a non-smooth one. In addition, we identify a mechanism for the creation of non-symmetric period-two orbits at finite $\varepsilon$ via a pitchfork bifurcation of a symmetric period-two orbit. A systematic study of these bifurcating non-symmetric periodic points, as well as the creation of symmetric periodic orbits at saddle connections, remains an important direction for future research.

A further open problem concerns the persistence of invariant curves in perturbed FIEMs. Numerical evidence suggests that generic perturbations destroy such curves, but a rigorous proof is still lacking. Formally, expanding the invariance equation for a putative invariant curve near the unperturbed dynamics and passing to the limit $\eps \rightarrow 0$ leads to a cohomological equation over the limiting IEM, of the type studied by Forni \cite{Forni1997, Forni2002} and Marmi-Moussa-Yoccoz~\cite{MMY2005}. Based on these works, we conjecture that generically bounded solutions to the limiting cohomological equation should not exist, and therefore that invariant curves should typically be destroyed by the perturbation. However, turning this formal argument into a precise theorem remains an open problem.

Beyond the intrinsic mathematical interest of perturbed FIEMs discussed above, these systems also arise naturally as models for physical problems with impacts, as illustrated in the final section. Our results may therefore be of independent interest in that context.

We expect that the methodology developed here, combining symmetry considerations, perturbative analysis, and explicit case studies, will be applicable to a broader class of piecewise continuous dynamical systems.

\begin{Backmatter}
\section*{Acknowledgments}

This work is supported by ISF grant 787/22. We thank G. Forni for very helpful comments and explanations.

\printbibliography[heading=bibintoc]
\end{Backmatter}
\end{document}